\documentclass{article}
\usepackage[alg]{klmm}
\usepackage{tikz}
\usepackage{subcaption}
\usetikzlibrary{positioning, arrows.meta}

\usepackage{booktabs}

\numberwithin{equation}{section}

\begin{document}


\pagestyle{fancy}

\newtheorem{notation}{Notation}

\def\N{{\mathbb N}}
\def\Z{{\mathbb Z}}
\def\Q{{\mathbb Q}}
\def\R{{\mathcal{R}}}
\def\K{{\mathcal{K}}}
\def\C{{\mathbb C}}
\def\H{{\mathcal{H}}}
\def\F{{\mathbb {F}}}
\def\X{{\mathbb{X}}}
\def\res{\hbox{\rm{res}}}
\def\LC{\hbox{\rm{LC}}}
\def\cont{\hbox{\rm{cont}}}
\def\MoCont{\hbox{\rm{MoCont}}}
\def\MoPrim{\hbox{\rm{MoPrim}}}
\def\Cont{\hbox{\rm{Cont}}}
\def\Prim{\hbox{\rm{Prim}}}
\def\poly{\hbox{\rm{Poly}}}
\def\x{\vec{\boldsymbol{x}}}
\def\b{\vec{\boldsymbol{b}}}
\def\inp{{\rm{in}}}
\def\out{{\rm{o}}}
\newcommand{\fs}[1]{\Phi_{\mathbf{s}}(#1)}
\newcommand{\Hi}{\mathbb{H}_{\rm in}}
\newcommand{\Ho}{\mathbb{H}_{o}}


\title{Sparse Polynomial GCD Algorithms Asymptotically Linear in All Fundamental Parameters}

\author{
    Qiao-Long Huang\institution{Shandong University, School of Mathematics,
Jinan, China}
    \and Xiao-Shan Gao\institution{State Key Laboratory of Mathematical Sciences, Academy of Mathematics and Systems Science, Chinese Academy of Sciences; University of Chinese Academy of Sciences, Beijing China.}
 }  
\date{\today}

\maketitle

{\begin{center}
\parbox{14.5cm}{\begin{abstract}
Let $A, B \in \mathbb{Z}[x_1, \dots, x_n]$ be multivariate polynomials with integer coefficients and let $G = \gcd(A, B)$. We present an algorithm for computing $G$ whose expected bit complexity is asymptotically linear in all fundamental parameters: the number of variables $n$, the term count $T = \max\{\|A\|_0, \|B\|_0, \|G\|_0\}$, the total degree $D$, and the logarithmic coefficient sizes $\log\Hi$ and $\log\Ho$, where $\Hi$ bounds the coefficients of the inputs and $\Ho$ bounds those of the GCD. The bit complexity is characterized by the clean bound
\[
\widetilde{O}\bigl( n \cdot T \cdot D \cdot \log\Hi \cdot \log\Ho \bigr).
\]
To our knowledge, this is the first sparse GCD algorithm over the integers that achieves linear complexity in all these parameters simultaneously.

The integer algorithm is built upon a new field GCD algorithm. For $A, B \in \K[x_1, \dots, x_n]$ over a field $\K$ with $\operatorname{char}(\K) = 0$ or $\operatorname{char}(\K) > \deg G$, we give the first algorithm that computes $G = \gcd(A,B)$ with expected 
\[
\widetilde{O}\bigl( n \cdot T \cdot D \bigr)
\]
field operations, which is both input- and output-sensitive.

The key technical contribution behind both algorithms is a derivative-aided separated Hensel lifting technique introduced in this paper.
By introducing an auxiliary variable and leveraging derivative information, our scheme extracts all partial exponents via a single $z^2$-lift per variable, achieving constant sequential depth $O(1)$. This stands in sharp contrast to classical Hensel lifting, which requires $O(D)$ sequential lifting steps and suffers from representation densification in the sparse setting. The field algorithm is then extended to the integer case through modular reduction and rational reconstruction.

Benchmark experiments in Maple confirm the theoretical analysis: our algorithm complements Maple's built-in \texttt{gcd} command, and exhibits a clear performance advantage when the total degree $D$ is large.
\end{abstract}}\end{center}}

\tableofcontents

\section{Introduction}

\subsection{A Historical Overview of Polynomial GCD Computation}

The computation of polynomial greatest common divisors (GCD) is one of the most fundamental problems in symbolic computation, and its development has paralleled the evolution of computer algebra itself. This section provides a chronological account of the major milestones in the history of polynomial GCD algorithms.

The classical Euclidean algorithm can be applied to polynomials using polynomial division with remainder, but this requires the coefficients of the polynomials to be invertible. As early as 1836, Jacobi introduced pseudo--division to overcome this restriction, allowing GCDs of integral polynomials to be computed without leaving the integers. However, this approach generates a polynomial remainder sequence (PRS) with exponential coefficient growth--a phenomenon known as \emph{intermediate expression swell}--and remains impractical for all but the smallest inputs \cite{vonZurGathenLuck2003}.

A major breakthrough came with Collins (1967) \cite{collins1967}, who introduced subresultant theory and proposed a polynomial remainder sequence algorithm that controls coefficient growth. Collins's subresultant PRS has integer coefficients with only linear growth, laying the foundation for practical GCD computation. Building on this work, Brown (1971) \cite{brown1971} further developed the theory of subresultants and, in the same paper, introduced the modular GCD algorithm, which computes GCDs of integer-coefficient polynomials modulo several primes and reconstructs the integer GCD via the Chinese Remainder Theorem, effectively avoiding coefficient growth entirely.

 Moses and Yun (1973) \cite{moses1973} developed the EZ-GCD algorithm based on multivariate Hensel lifting, and later Wang (1980) \cite{wang1980} proposed the EEZ-GCD algorithm as an extension to handle non-monic inputs more effectively.
These algorithms compute the GCD by first evaluating all but one variable at a random point and then lifting the resulting univariate GCD back to the multivariate setting via Hensel lifting, significantly improving practical efficiency.

For sparse polynomials, Zippel introduced a probabilistic interpolation method in 1979 \cite{zippel1979}, based on the observation that a nonzero polynomial evaluated at a random point is almost never zero; hence, any zero coefficient in a random evaluation can be assumed to be zero in the final answer. Building on this, Zippel (1981) \cite{zippel1981} incorporated this idea into a Newton-iteration-based Hensel lifting framework, where the factorization problem is converted into a system of equations and solved iteratively, considering only non-zero terms at each step. This improved the classical EZ-GCD and EEZ-GCD algorithms and is the foundation of Zippel's sparse modular GCD algorithm, which is now the default GCD algorithm in several major computer algebra systems, including Maple, Magma, and Mathematica  \cite{HuMonagan2021} for $\Z[x_1,\dots,x_n]$.

An alternative approach was proposed by Char, Geddes, and Gonnet (1984) \cite{char1984gcdheu}, who introduced a heuristic GCD algorithm based on a single evaluation at a large integer and interpolation from the resulting integer GCD.  It is efficient for problems with few variables.

  Gianni and Trager (1985) \cite{gianni1985gcd} showed that GCDs can be computed as the least degree member of a Gr\"obner basis for an ideal defined in terms of the input polynomials, providing a connection between GCD computation and Gr\"obner basis theory.

Kaltofen (1985) \cite{kaltofen1985b} improved Zippel's sparse Hensel lifting framework by introducing a leading coefficient determination algorithm that works with just the factorization of the univariate image and does not require, unlike Wang's method, the factorization of the leading coefficient of the input polynomial.

A significant departure from traditional approaches came with Kaltofen's work on polynomials represented by straight-line programs. Kaltofen (1985, 1988) \cite{kaltofen1985,kaltofen1988} showed that most algebraic algorithms can be probabilistically applied to data given by a straight-line computation. In particular, he developed a randomized GCD algorithm for multivariate polynomials in this model, where the complexity is polynomial in the program length and the degree of the inputs.

Kaltofen and Trager (1990) \cite{kaltofen1990} further extended this framework to black-box polynomial representations, where the polynomial is accessible only through evaluations. They constructed an evaluation procedure for the greatest common divisor of multivariate polynomials given by black boxes, providing a powerful tool for manipulating polynomials with implicit representations.

Sasaki and Suzuki (1992) \cite{sasaki1992three} presented three algorithms for the multivariate polynomial GCD. The first is a Gr\"obner basis method, which establishes a theoretical connection but is not practically efficient. The other two are truncated power series methods--a subresultant variant and a PRS variant--that achieve significant efficiency by discarding higher-degree terms that do not affect the GCD.

De Kleine, Monagan, and Wittkopf (2005) \cite{dekleine2005} proposed a different approach to the non-monic case of Zippel's sparse modular GCD algorithm. Instead of determining the leading coefficient via factorization, their LINZIP algorithm treats the scaling factors as unknowns and solves a structured coupled linear system, while their RATZIP algorithm reconstructs the monic GCD over the rational function field and then clears the denominators. Both algorithms avoid polynomial factorization entirely at the cost of either solving a larger linear system or performing an additional recursive GCD computation in fewer variables.

The 2000s also saw significant progress on approximate GCD for inexact polynomials, with contributions from Boito on structured matrix methods, Christou et al. on the ERES method, Bini and Boito on fast structured algorithms, Terui on the GPGCD iteration, and Kaltofen, Yang, and Zhi on STLN-based optimization~\cite{boito2011structured,christou2010eres,bini2010fast,terui2010gpgcd,kaltofen2006approximate}.

More recently, Hu and Monagan (2021) \cite{HuMonagan2021} introduced a fast parallel sparse polynomial GCD algorithm that combines a Kronecker substitution with Ben-Or/Tiwari sparse interpolation modulo a smooth prime to determine the support of the GCD. Their algorithm is highly parallelizable and demonstrates significant performance advantages over serial implementations of Zippel's GCD algorithm in Maple and Magma.

Huang and Monagan (2024) \cite{HuangMonagan2024} presented a sparse polynomial GCD algorithm by separating terms with a detailed complexity analysis. Huang and Gao (2026) \cite{HuangGao2025} proposed another GCD algorithm that combines term separation and sparse interpolation while providing its bit complexity. The former has quadratic complexity in the degree $D$, while the latter has quadratic complexity in the term count $T$.

Demin and van der Hoeven (2025) \cite{DeminHoeven2025} proposed a new GCD algorithm within their evaluation-interpolation framework using geometric progressions, introducing an auxiliary variable to normalize the leading coefficient and enable efficient lifting. Their method provides a unified treatment of GCD computation and sparse factorization and demonstrates the power of combining sparse interpolation with structured evaluation sequences.

\subsection{Complexity Barriers for Sparse GCD}

A fundamental question for any GCD algorithm is: what is the best complexity one can hope for? To answer this, we must first understand the inherent difficulty of the problem.

\paragraph{Output-Sensitive Complexity is Unavoidable.}

The GCD of two sparse polynomials can have a size that is exponential in the sizes of the input polynomials. The following classical example illustrates this phenomenon.

\begin{example}[\cite{GCD-Schinzel2003}]
\label{ex-GCD1}
Let $p, q$ be distinct primes and consider
\[
A = x^{pq} - 1, \qquad 
B = x^{p+q} - x^p - x^q + 1 = (x^p - 1)(x^q - 1).
\]
The irreducible decompositions are
\[
x^p - 1 = (x-1)\Phi_p,\qquad
x^q - 1 = (x-1)\Phi_q,\qquad
x^{pq} - 1 = (x-1)\Phi_p\Phi_q\Phi_{pq},
\]
where $\Phi_n$ denotes the $n$-th cyclotomic polynomial. Hence
\[
\gcd(A,B) = (x-1)\Phi_p\Phi_q = (x^p-1)\Phi_q = x^p\Phi_q - \Phi_q,
\]
which contains $2q$ terms, assuming $q < p$.
\end{example}

Thus, the input polynomials $f$ and $g$ have $2$ and $4$ terms, respectively, while their GCD has $2q$ terms, which is exponential in the input size $\log q$. This observation has profound implications:

\begin{quote}
\emph{No algorithm for sparse GCD can be polynomial in the input size alone; any meaningful complexity bound must be output-sensitive, i.e., must depend explicitly on the number of terms in the GCD.}
\end{quote}

\paragraph{Output-Sensitive GCD computation is NP-hard.}

Let $T$ be an upper bound on the number of terms of the inputs and the GCD, and let $D$ be an upper bound on the total degree. A tight upper bound of the input size is $O(n T \log D + T \mathbb{H})$, where $\mathbb{H}$ denotes the bit size of the coefficients, and the output size is also $O(n T \log D + T \mathbb{H})$.
A natural question is~\citep[Challenge 5]{SparsityChallenges}: 

\begin{quote}
\emph{Does there exist a polynomial-time, output-sensitive GCD algorithm, that is, an algorithm for sparse GCD whose complexity is polynomial in the combined input-output size.}
\end{quote}

A recent result \cite{Qiu2026OutputSensitive} settles the above question in the negative direction. It is shown that

\begin{theorem}[\cite{Qiu2026OutputSensitive}]
\label{th-gcdnph}
Output-sensitive GCD computational over finite fields is NP-hard under BPP reduction.
\end{theorem}

In other words, any random algorithm with complexity $\operatorname{poly}(n, T, \log D)$ would imply $\mathrm{P} = \mathrm{NP}$. 

\paragraph{Complexity of GCD algorithms based on the input-output size and the degree.}

Most of the existing work focuses on designing GCD algorithms whose running time is polynomial in the combined input-output size and the degree, which was proposed as  Research Problem 16.17 in~\cite{vonzurGathen2013}.

In the seminal work of Kaltofen \cite{kaltofen1988,kaltofen1985}, it was proved that for polynomials represented by straight-line programs, the GCD can be computed in randomized polynomial time. 
Subsequent research--including Zippel’s probabilistic algorithms for sparse polynomials, sparse interpolation techniques, and several enhancements based on Hensel lifting--has improved the complexity in different ways. Nevertheless, to the best of our knowledge, no approach has yet produced an algorithm whose complexity is simultaneously linear in all three key parameters $n$, $T$, and $D$.
\begin{table}[ht]
    \centering
    \begin{tabular}{llll}
\toprule
\textbf{GCD Algorithm} & \textbf{Complexity} & \textbf{Type} & \textbf{Conditions} \\
\midrule
Kaltofen \cite{kaltofen1988,kaltofen1985} & $\operatorname{poly}(n,T,D)$ & Monte Carlo & None \\
Zippel \cite{zippel1979} & $\widetilde{O}(n \cdot T^2 \cdot D^2)$ & Monte Carlo & Monic inputs \\
Huang--Gao \cite{HuangGao2025} & $\widetilde{O}(n \cdot T^2 \cdot D)$ & Monte Carlo & Primitive root given \\
Huang--Monagan \cite{HuangMonagan2024} & $\widetilde{O}(n \cdot T \cdot D^2)$ & Monte Carlo & None \\
\textbf{This paper} & $\widetilde{O}(n \cdot T \cdot D)$ & Monte Carlo & $\operatorname{char}(\K)=0$ or $>\deg G$ \\
\bottomrule
\end{tabular}
    \caption{Comparison of GCD algorithms with polynomial complexity in input-output size and degree.}
    \label{tab:comparison}
\end{table}

\subsection{Our Main Result}

We design a novel GCD algorithm based on derivative-assisted separated Hensel lifting and prove that its time complexity in $n, T, D$ is
\[
\widetilde{O}(n \cdot T \cdot D).
\]
This is asymptotically \textbf{linear} in all three parameters $n$, $T$, and $D$. In particular, the dependence on $n$ and $T$ is linear, achieving the optimal exponent for these parameters.
Furthermore, the result is valid for a broad and natural class of fields: over any field of characteristic zero or exceeding $D$, there exists a universal GCD algorithm with complexity $\widetilde{O}(nTD)$. This class includes $\mathbb{Q}$, $\mathbb{R}$, $\mathbb{C}$, and finite fields $\mathbb{F}_q$ with $\operatorname{char}(\F_q) > D$, which cover the vast majority of fields of practical interest in symbolic computation.


For ease of reference, we summarize the main results of this paper as follows.

\begin{theorem}\label{thm:main-field}
Let $A, B \in \K[x_1, \dots, x_n]$ be multivariate polynomials over a field $\K$, and let $G = \gcd(A, B)$. 
Suppose $\operatorname{char}(\K) = 0$ or $\operatorname{char}(\K) > \max_{1 \le i \le n} \min\{\deg_{x_i} A, \deg_{x_i} B\}$. 
Let $D := \max\{\deg A, \deg B\}$, $T_G := \|G\|_0$, $T_A := \|A\|_0$, and $T_B := \|B\|_0$. 
Then there exists a randomized algorithm that computes $G$ with probability at least $1-\varepsilon$ and  and whose expected cost is as follows:

\begin{itemize}
    \item If $\K$ is infinite,
    \[
    \widetilde{O}\Bigl(
    n T_G D \log\frac{1}{\varepsilon}
    + n(T_A + T_B) \log^4 T_G \log D \log\frac{1}{\varepsilon}
    \Bigr)
    \]
    field operations in $\K$.

    \item If $\K = \mathbb{F}_q$ is finite,
    \[
    \widetilde{O}\Bigl(
    n T_G D \log q \log\frac{1}{\varepsilon}
    + n(T_A + T_B) \log^4 T_G \log^2 D \log q \log\frac{1}{\varepsilon}
    \Bigr)
    \]
    bit operations.
\end{itemize}
If $T = \max\{T_A, T_B, T_G\}$, the complexity simplifies to $\widetilde{O}( n T D \log\frac{1}{\varepsilon})$ field operations in $\K$ when $\K$ is infinite, and to $\widetilde{O}( n T D \log q \log\frac{1}{\varepsilon})$ bit operations when $\K = \mathbb{F}_q$.
\end{theorem}

\begin{theorem}\label{thm:main-integer}
Let $A, B \in \mathbb{Z}[x_1, \dots, x_n]$ be integer polynomials, and let $G = \gcd(A, B)$. 
Let $D := \max\{\deg A, \deg B\}$, $T_G := \|G\|_0$, $T_A := \|A\|_0$, and $T_B := \|B\|_0$, and define
\[
\Hi := \max\{\|A\|_\infty, \|B\|_\infty\}, \qquad \Ho := \|G\|_\infty.
\]
Then there exists a randomized algorithm that computes $G$ with probability at least $1-\varepsilon$ and expected bit complexity
\[
\widetilde{O}\Bigl( 
n T_G D \cdot \log\Ho \cdot \log\Hi \cdot \log\frac{1}{\varepsilon}
+ n (T_A + T_B) \cdot \log^4 T_G \cdot \log^4 D \cdot \log \Hi \cdot \log\Ho \cdot \log\frac{1}{\varepsilon}
\Bigr).
\]

If $T = \max\{T_A, T_B, T_G\}$, the expected complexity simplifies to
$\widetilde{O}\Bigl( 
n T D \cdot \log\Ho \cdot \log\Hi \cdot \log\frac{1}{\varepsilon}
\Bigr).$
\end{theorem}

We note that the complexity bound in Theorem~\ref{thm:main-integer} is asymptotically linear in each of the natural parameters of the sparse input-output representation:  apart from $\log D$: the number of variables $n$, the term counts $T_G$, $T_A$, $T_B$, and the logarithmic coefficient sizes $\log \Hi$ and $\log \Ho$. The total degree $D$ itself also appears linearly. 
Thus, the algorithm attains near-optimal scaling across all relevant parameters, except for $\log D$.

\begin{theorem}[A Sparse Coefficient Bound for Multivariate Polynomial Factors]
Let $F \in \mathbb{Z}[x_1, \dots, x_n]$ be a nonzero polynomial, and let $G \in \mathbb{Z}[x_1, \dots, x_n]$ be any factor of $F$. Let $T = \|G\|_0$ be the number of terms of $G$, and let $D = \deg G$ be its total degree. Then
\[
\|G\|_\infty \le 2^{D(T-1)} \|F\|_1.
\]
\end{theorem}
The above theorem provides a coefficient bound for any factor $G$ of a multivariate integer polynomial $F$. 
Unlike the classical Mignotte bound~\cite{mignotte1974inequality}, which depends exponentially on the number of variables $n$ (via the standard Kronecker substitution), our bound depends only on the total degree $D$ and the sparsity $T$ of the factor $G$, and is completely independent of $n$. 
This is a sparse analog of Mignotte's classical univariate bound, generalized to the multivariate setting while preserving sensitivity to the factor's own term structure.
In the extreme case where $G$ is a monomial ($T = 1$), the bound reduces to $\|G\|_\infty \le \|F\|_1$. This bound is tight when $\|F\|_0 = 1$, i.e., when $F$ is itself a single term.

As \cite{Qiu2026OutputSensitive} shows, under standard complexity-theoretic assumptions, no algorithm can achieve complexity $\operatorname{poly}(n, T, \log D)$ for sparse GCD. Thus, the total degree $D$ cannot be compressed to a logarithmic factor; any tractable algorithm must have a complexity that grows at least faster than polylogarithmic in $D$. Our algorithm achieves $\widetilde{O}(n \cdot T \cdot D)$, which is linear in all three parameters and is, to the best of our knowledge, the first algorithm with such a clean complexity bound. Whether the dependence on $D$ can be improved to sub-linear, e.g., $\widetilde{O}(n \cdot T \cdot D^{1/2})$, remains an interesting open problem.

\subsection{Technical Contribution: Derivative-Driven Hensel Lifting}
In this section, we provide an intuitive explanation of our main technical contribution: the derivative-driven Hensel lifting that is proposed in this paper.

\paragraph{History of Hensel Lifting.}
Hensel lifting is named after the German mathematician Kurt Hensel, who founded the theory of $p$-adic numbers in the early 20th century \cite{hensel1908}. In his research, Hensel realized that one can build solutions to integer equations modulo higher and higher powers of a prime incrementally, starting from a solution modulo that prime--an approach that later became known as Hensel’s Lemma.
%

Hensel lifting was first introduced into symbolic computation in the 1960s, primarily for polynomial factorization. In 1969, Zassenhaus \cite{zassenhaus1969} proposed an integer polynomial factorization algorithm based on Hensel lifting.
Subsequently, Hensel lifting was extended to multivariate polynomials. Moses and Yun (1973) \cite{moses1973} developed the EZ-GCD algorithm based on multivariate Hensel lifting, and later Wang (1980) \cite{wang1980} proposed the EEZ-GCD algorithm to improve it.

Zippel (1981) \cite{zippel1981} introduced sparse Hensel lifting, combining probabilistic interpolation with a Newton-iteration framework to solve the factorization problem iteratively while tracking only non-zero terms. This made the algorithm suitable for sparse polynomials.

Later, Kaltofen~\cite{kaltofen1985b} improved Zippel's sparse Hensel lifting framework, and more recently, Monagan and Tuncer~\cite{MonaganTuncer2016,MonaganTuncer2020} introduced a new approach based on sparse interpolation, significantly improving practical efficiency.
Since then, multivariate Hensel lifting has become a standard tool in symbolic computation, widely applied to factorization, GCD computation, and algebraic equation solving.


\paragraph{Previous Hensel Lifting Methods.}
Let \(F(\x) \in \R[x_1, \dots, x_n]\) be a multivariate polynomial over a unique factorization domain (UFD) \(\R\). Suppose that in the quotient ring modulo some ideal \(\mathcal{I}\), \(F\) admits a factorization
\[
F \equiv G \cdot H \pmod{\mathcal{I}}.
\]
The goal of Hensel lifting is to ``lift'' this factorization to higher powers \(\mathcal{I}^k\), i.e., to find \(G_k, H_k\) such that
\[
F \equiv G_k \cdot H_k \pmod{\mathcal{I}^k},
\]
with \(G_k \equiv G \pmod{\mathcal{I}}\) and \(H_k \equiv H \pmod{\mathcal{I}}\). When \(k\) is sufficiently large, the exact factorization of \(F\) can be recovered in the original ring.

In the univariate case, the classical Hensel lemma provides a sufficient condition for lifting: if \(F(x) \equiv G(x) H(x) \pmod{p}\), where \(p\) is a prime or an irreducible polynomial, and \(G, H\) are coprime modulo $p$, then there exists a unique lift modulo \(p^k\).

In practice, multivariate Hensel lifting is typically performed by choosing a main variable (e.g., \(x_1\)), treating the other variables as parameters, and considering the quotient ring modulo the ideal \(\mathcal{I} = \langle x_2 - b_2, \dots, x_n - b_n \rangle\). Under this specialization, \(F(x_1, b_2, \dots, b_n)\) becomes a univariate polynomial in \(x_1\). 
If \(G(x_1, b_2, \dots, b_n)\) and \(H(x_1, b_2, \dots, b_n)\) are coprime at this evaluation point, then by Hensel's lemma, the factorization can be lifted to the modulus \(\langle x_2 - b_2, \dots, x_n - b_n \rangle^k\), and the original multivariate factors can subsequently be recovered through adjustment.

For sparse polynomials, Zippel \cite{zippel1981} introduced a different paradigm for Hensel lifting. Instead of lifting all variables simultaneously, his method recovers variables one by one. The key idea is to reformulate the factorization problem as a system of polynomial equations. After substituting random constants for the variables not yet introduced, one obtains a simplified system; coefficients that vanish at the random evaluation are assumed to be identically zero, thereby reducing the number of unknowns. Newton iteration is then applied to this simplified system to solve for the unknown coefficient polynomials, introducing one new variable at a time. By tracking only the non-zero terms at each step, the algorithm is well-suited to sparse polynomials.

\paragraph{Advantage of our Hensel Lifting.}
Despite various improvements over the past four decades, all existing sparse Hensel lifting algorithms remain fundamentally sequential in their lifting depth. Classical Hensel lifting, as used in the EZ-GCD algorithm \cite{moses1973}, lifts the factorization successively from $\mathcal{I}$ to $\mathcal{I}^2$, $\mathcal{I}^3$, and so on, until $\mathcal{I}^{D+1}$, resulting in a sequential depth of $O(D)$. Zippel's sparse Hensel lifting \cite{zippel1981} improves this by using Newton iteration, which doubles the lifting order at each step; for each variable, the lifting depth is reduced from $O(d)$ to $O(\log d)$, where $d$ is the partial degree bound. Thus, over $n$ variables, the total sequential depth is $O(n \cdot \log d)$. 

In contrast, our method performs only a single lift $\mathcal{I} \to \mathcal{I}^2$ per variable, and all $n$ lifts are independent of one another. This reduces the sequential depth to $O(1)$, regardless of the total degree $D$ or the number of variables $n$. To recover all terms, we repeat this process for $O(\log T)$ rounds; the number of parallel lifts becomes $n\log T$, but the sequential depth remains $O(1)$.

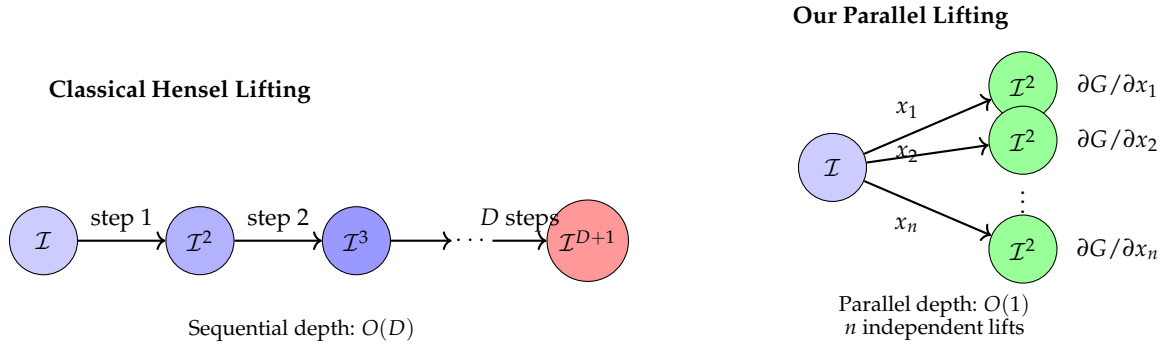
\begin{figure}[htbp]
\centering
\begin{subfigure}{0.45\textwidth}
    \centering
    \begin{tikzpicture}[node distance=0.5cm, auto, scale=0.9, every node/.style={scale=0.9}]
        \node at (0,2.2) [font=\bfseries] {Classical Hensel Lifting};
        \node[draw, circle, minimum size=1.0cm, fill=blue!20] (I) at (-2.0,0) {$\mathcal{I}$};
        \node[draw, circle, minimum size=1.0cm, fill=blue!30] (I2) at (0.3,0) {$\mathcal{I}^2$};
        \node[draw, circle, minimum size=1.0cm, fill=blue!40] (I3) at (2.6,0) {$\mathcal{I}^3$};
        \node at (4.3,0) {$\cdots$};
        \node[draw, circle, minimum size=1.0cm, fill=red!40] (ID) at (6.0,0) {$\mathcal{I}^{D+1}$};
        \draw[->, thick] (I) -- (I2) node[midway, above] {step 1};
        \draw[->, thick] (I2) -- (I3) node[midway, above] {step 2};
        \draw[->, thick] (I3) -- (4.0,0) node[midway, above]{};
        \draw[->, thick] (4.6,0) -- (ID) node[midway, above] {$D$ steps};
        \node at (1.8,-1.3) [font=\small, align=center] {Sequential depth: $O(D)$};
    \end{tikzpicture}
\end{subfigure}
\hfill
\begin{subfigure}{0.45\textwidth}
    \centering
    \begin{tikzpicture}[node distance=0.5cm, auto, scale=0.9, every node/.style={scale=0.9}]
        \node at (0,2.2) [font=\bfseries] {Our Parallel Lifting};
        \node[draw, circle, minimum size=1.0cm, fill=blue!20] (I) at (-1.0,0) {$\mathcal{I}$};
        \node[draw, circle, minimum size=1.0cm, fill=green!40] (I2_1) at (1.8,1.2) {$\mathcal{I}^2$};
        \node[draw, circle, minimum size=1.0cm, fill=green!40] (I2_2) at (1.8,0.4) {$\mathcal{I}^2$};
        \node at (1.8,-0.4) {$\vdots$};
        \node[draw, circle, minimum size=1.0cm, fill=green!40] (I2_n) at (1.8,-1.2) {$\mathcal{I}^2$};
        \draw[->, thick] (I) -- (I2_1) node[midway, above left] {$x_1$};
        \draw[->, thick] (I) -- (I2_2) node[midway, left] {$x_2$};
        \draw[->, thick] (I) -- (I2_n) node[midway, below left] {$x_n$};
        \node at (3.2,1.2) {$\partial G/\partial x_1$};
        \node at (3.2,0.4) {$\partial G/\partial x_2$};
        \node at (3.2,-1.2) {$\partial G/\partial x_n$};
        \node at (0.5,-2.2) [font=\small, align=center] {Parallel depth: $O(1)$\\$n$ independent lifts};
    \end{tikzpicture}
\end{subfigure}
\caption{Comparison of lifting strategies. Left: classical Hensel lifting requires $D$ sequential steps. Right: our method performs one lift per variable, all in parallel.}
\label{fig:lifting-comparison}
\end{figure}

For ease of illustration, Figure~\ref{fig:lifting-comparison} contrasts the sequential lifting depths of classical Hensel lifting and our approach. On the left, classical Hensel lifting proceeds serially: each step lifts the factorization from $\mathcal{I}^k$ to $\mathcal{I}^{k+1}$, requiring a total of $D$ steps to reach $\mathcal{I}^{D+1}$, with each step dependent on the previous one. On the right, our method performs a single lift $\mathcal{I} \to \mathcal{I}^2$ independently for each variable $x_1,\dots,x_n$; all $n$ lifts can be executed in parallel, and the resulting derivative information is then combined via sparse interpolation. This parallel structure is the key to achieving linear complexity in $n$, $T$, and $D$.

\subsection{Organization of the Paper}
\label{sec:roadmap}
The remainder of the paper is organized as follows. Section~\ref{sec:prelim} presents preliminary definitions and lemmas. Section~\ref{sec:lifting} develops the derivative-assisted Hensel lifting machinery. Section~\ref{sec:field-gcd} presents the GCD algorithm over a field, including the verification and guessing strategies. Section~\ref{sec:int-gcd} extends the algorithm to integer coefficients. Section~\ref{sec:experiments} reports experimental results, and Section~\ref{sec:conclusion} concludes the paper.

\section{Preliminaries}\label{sec:prelim}

\subsection{Notation and Definitions}

Throughout this paper, we adopt the following notational convention to distinguish polynomials in different variable sets.

\begin{itemize}
    \item 
    We reserve \textbf{uppercase} letters for polynomials depending solely on the   $n$ variables $x_1,\dots,x_n$. For instance,
\[
A, B, G \in \K[x_1,\dots,x_n]
\]
represent multivariate polynomials in these original variables.
    \item For polynomials in the extended variable set that includes the auxiliary variable $y$, i.e., in $\K[x_1,\dots,x_n,y]$, we use \textbf{lowercase} letters. For example,
    \[
    f, g, h \in \K[x_1,\dots,x_n,y]
    \]
    denote polynomials that may depend on $y$ as well.
\end{itemize}

This convention is consistently applied throughout the paper. In particular, when we apply the separation transformation $\fs{\cdot}$ (defined below in Definition~\ref{def-separated}) to a polynomial $F \in \K[x_1,\dots,x_n]$, the resulting polynomial $\fs{F} \in \K[x_1,\dots,x_n,y]$ is denoted by a lowercase letter. Thus, for $G = \gcd(A, B)$, we write $\fs{G}$ as the transformed GCD, and its evaluation at a point $\b$ is denoted $g(\b, y)$ or $g_0(\b,y)$ depending on the context.

\begin{definition}[Separation with Respect to the Variable $y$]
Let
\(
f(\x, y) \in \R[x_1, \dots, x_n, y],
\)
where \(\R\) is a commutative ring with identity (typically a field or the ring of integers). If there exist distinct nonnegative integers \(d_1, d_2, \dots, d_t\), and monomials in \(x_1, \dots, x_n,\)
\(
m_i(\x) = x_1^{e_{i1}} \cdots x_n^{e_{in}}, \quad c_i \in \R \setminus \{0\},
\)
such that
\[
f = \sum_{i=1}^{t} c_i \cdot m_i(\x) \cdot y^{d_i},
\]
then the polynomial \(f\) is said to be \textbf{separated with respect to the variable \(y\)}, or simply \textbf{$y$-separated}.
\end{definition}

Equivalently, every nonzero monomial in the expanded form of $f$ has a distinct $y$-degree.

\begin{definition}\label{def-separated}
Let $F \in \K[x_1, \dots, x_n]$ with $\K$ be a field. For a vector $\mathbf{s} = (s_1, \dots, s_n) \in \mathbb{N}^n$, define
\[
\fs{F}(\x, y)= \frac{F(x_1 y^{s_1}, \dots, x_n y^{s_n})}{y^{k_F}},
\]
where $k_F$ is the lowest power of $y$ in the $F(x_1 y^{s_1}, \dots, x_n y^{s_n})$.
\end{definition}

We fix the lexicographic monomial order $\succ$ on $\K[\x, y]$ with
\[
x_n \succ x_{n-1} \succ \cdots \succ x_1 \succ y,
\]
and this order is used throughout the entire paper.

This choice ensures that the monomial order is an elimination order: any monomial involving at least one of the variables $x_1,\dots,x_n$ is greater than any monomial consisting solely of a power of $y$. Consequently, when we apply the separation transformation $x_i \mapsto x_i y^{s_i}$, the relative order of two monomials is determined entirely by their $\x$-parts, and the appended $y$-factors do not affect the comparison. In other words, the order of the monomials in $\fs{F}$ is the same as the order of the corresponding monomials in the original polynomial $F$. This property is essential for identifying the lowest term of $\fs{F}$ from the lowest term of $F$, which provides the canonical alignment point in our recursive strategy.

For a polynomial $f\in \R\left[x_1,\dots,x_n,y \right]$, denote ${\operatorname{lc}}(f)$ as the leading coefficient of $f$.

\begin{definition}
Let $A, B \in \K[x_1,\dots,x_n]$ be nonzero polynomials over a field $\K$. A polynomial $G \in \K[x_1,\dots,x_n]$ is called the \textbf{greatest common divisor} of $A$ and $B$, denoted $G = \gcd(A,B)$, if it satisfies the following three conditions:

\begin{enumerate}
    \item $G$ divides both $A$ and $B$ in $\K[x_1,\dots,x_n]$;
    \item every common divisor of $A$ and $B$ in $\K[x_1,\dots,x_n]$ divides $G$;
    \item $\operatorname{lc}(G) = 1$.
\end{enumerate}
\end{definition}

\begin{definition}
Let $A, B \in \mathbb{Z}[x_1,\dots,x_n]$ be nonzero polynomials with integer coefficients. A polynomial $G \in \mathbb{Z}[x_1,\dots,x_n]$ is called the \textbf{greatest common divisor} of $A$ and $B$, denoted $G = \gcd(A,B)$, if it satisfies the following three conditions:

\begin{enumerate}
    \item $G$ divides both $A$ and $B$ in $\mathbb{Z}[x_1,\dots,x_n]$;
    \item every common divisor of $A$ and $B$ in $\mathbb{Z}[x_1,\dots,x_n]$ divides $G$;
    \item $\operatorname{lc}(G)>0$.
\end{enumerate}
\end{definition}

\begin{definition}\label{def-monocont}
Let $F = c_1 M_1 + \cdots + c_t M_t \in \R[x_1,\dots,x_n]$ with $c_i \neq 0$, and let $M_i = \x^{\mathbf{e}_i}$ for distinct exponent vectors $\mathbf{e}_i=(e_{i,1},\cdots,e_{i,n}) \in \mathbb{N}^n$. The \textbf{monomial content} of $F$, denoted $\operatorname{MoCont}(F)$, is defined as
\[
\operatorname{MoCont}(F) = \gcd(M_1,\dots,M_t)=\x^{\mathbf{e}},
\]
where $\mathbf{e} = (\min_i e_{i1}, \dots, \min_i e_{in})$ is the componentwise minimum of the exponent vectors of the nonzero terms of $F$.
\end{definition}

Let $f, g \in \R[x_1,\dots,x_n,y]$ be nonzero polynomials over an integral domain $\R$, with
\[
f = \sum_{i=0}^{d} F_i y^i, \qquad g = \sum_{i=0}^{\ell} G_i y^i,
\]
where $F_i, G_i \in \R[x_1,\dots,x_n]$ and $F_d, G_\ell \neq 0$.

The \textbf{Sylvester matrix} of $f$ and $g$ is the $(d+\ell) \times (d+\ell)$ matrix
\[
\begin{pmatrix}
F_d & F_{d-1} & \cdots & F_1 & F_0 & & \\
    & F_d & F_{d-1} & \cdots & F_1 & F_0 & \\
    & & \ddots & \ddots & \ddots & \ddots & \\
    & & & F_d & F_{d-1} & \cdots & F_0 \\
G_\ell & G_{\ell-1} & \cdots & G_1 & G_0 & & \\
    & G_\ell & G_{\ell-1} & \cdots & G_1 & G_0 & \\
    & & \ddots & \ddots & \ddots & \ddots & \\
    & & & G_\ell & G_{\ell-1} & \cdots & G_0
\end{pmatrix},
\]
where the upper $\ell$ rows contain the coefficients of $f$, and the lower $d$ rows contain the coefficients of $g$.

The \textbf{resultant} of $f$ and $g$ with respect to $y$, denoted $\operatorname{Res}_y(f,g)$, is the determinant of their Sylvester matrix. If $\deg_y f = 0$ and $\deg_y g = 0$, the Sylvester matrix is not defined; in this case, we adopt the convention
\[
\operatorname{Res}_y(F,G) = 1.
\]

\subsection{Preliminary Lemmas}
We introduce several lemmas.

The following lemma summarizes the key properties of the resultant that will be used throughout the paper.

\begin{lemma}[cf. Lemma 4 of \cite{HuMonagan2021}]\label{lm-resultant}
Let $\mathcal{R}$ be an integral domain and let $f, g \in \mathcal{R}[x_1,\dots,x_n,y]$. Let
\[
F_d = \operatorname{lc}_y(f), \qquad G_\ell = \operatorname{lc}_y(g), \qquad R = \operatorname{Res}_y(f,g),
\]
and let $\b = (b_1,\dots,b_n) \in \mathcal{R}^n$. Then the following hold:

\begin{enumerate}
    \item $F_d, G_\ell, R \in \mathcal{R}[x_1,\dots,x_n]$.
    \item If $\mathcal{R}$ is a field, $F_d(\b) \neq 0$, and $G_\ell(\b) \neq 0$, then
    \[
    \operatorname{Res}_y(f(\b, y), g(\b, y)) = R(\b),
    \]
    and
    \[
    \deg_y \gcd(f(\b, y), g(\b, y)) > 0
    \quad \Longleftrightarrow \quad
    \operatorname{Res}_y(f(\b, y), g(\b, y)) = 0.
    \]
\end{enumerate}
\end{lemma}

The following lemma states that the order of the monomials in $\fs{F}$ is completely determined by the order of their $\x$-parts; the $y$-exponents do not affect the relative ordering.
\begin{lemma}\label{thm:order-preservation}
Let
$F = c_1 M_1 + c_2 M_2 + \cdots + c_t M_t \in \K[x_1, \dots, x_n],$
where $M_i = x_1^{e_{i1}} \cdots x_n^{e_{in}}$ are distinct monomials and $c_i \neq 0$. Suppose the monomials are ordered such that
$M_1 \succ M_2 \succ \cdots \succ M_t.$
For $\mathbf{s} = (s_1, \dots, s_n) \in \mathbb{N}^n$, let
$\fs{F}(\x, y) = c_1 M_1 y^{d_1} + c_2 M_2 y^{d_2} + \cdots + c_t M_t y^{d_t}.$
Then
\[
M_1 y^{d_1} \succ M_2 y^{d_2} \succ \cdots \succ M_t y^{d_t}.
\]

\end{lemma}

\begin{proof}
For any two distinct monomials $M_i = \x^{\mathbf{e}_i}$ and $M_j = \x^{\mathbf{e}_j}$, the lexicographic order with $x_n \succ \cdots \succ x_1 \succ y$ compares monomials first by their $\x$-exponents. The $y$-exponents are only considered when the $\x$-parts are identical. Since $M_i$ and $M_j$ are distinct, their $\x$-parts differ, so the comparison between $\x^{\mathbf{e}_i} y^{d_i}$ and $\x^{\mathbf{e}_j} y^{d_j}$ is decided solely by the $\x$-parts, exactly as in the comparison between $M_i$ and $M_j$. Hence $M_i \succ M_j$ if and only if $\x^{\mathbf{e}_i} y^{d_i} \succ \x^{\mathbf{e}_j} y^{d_j}$. The claimed ordering follows immediately.
\end{proof}

\begin{lemma}\label{lem:irreducible-separation}
Let $F \in \K[x_1,\dots,x_n]$ be irreducible over a field $\K$, and let 
$\mathbf{s}=(s_1,\dots,s_n)\in \mathbb{N}^n$. 
Then $\fs{F}$ is also irreducible in $\K[\x, y]$.
\end{lemma}

\begin{proof}
Suppose $\fs{F}(\x,y)=p(\x,y)q(\x,y)$ with $p,q$ nonconstant. Then
\[
f(\x,y):=F(x_1y^{s_1},\dots,x_ny^{s_n})=y^{k_F}p(\x,y)q(\x,y). \tag{1}
\]
Setting $y=1$ gives $F(\x)=p(\x,1)q(\x,1)$. Since $F$ is irreducible, WLOG assume
\[
p(\x,1)=F(\x), \qquad q(\x,1)=1\in\K^*. \tag{2}
\]
Let $\operatorname{LM}(h)$ denote the leading monomial of $h$ with respect to $\succ$.
Since $p(\x,1)=F(\x)$, we have
\[
\operatorname{LM}(p) \succeq \operatorname{LM}(F) \cdot y^t \quad \text{for some } t\ge 0, \tag{3}
\]
where the inequality is in the monomial order. More precisely, the $\x$-part of $\operatorname{LM}(p)$ is at least $\operatorname{LM}(F)(\x)$ in the monomial order restricted to $\K[\x]$, possibly multiplied by some power of $y$.  It cannot be strictly smaller because then after setting $y=1$ it could not produce $\operatorname{LM}(F)$.

Let $\operatorname{LM}(F)(\x)=\x^{\mathbf{a}}$. From (1),
\[
\operatorname{LM}(f)=\x^{\mathbf{a}} y^{\sum a_i s_i}. \tag{4}
\]
On the other hand, using (3),
\[
\x^{\mathbf{a}} y^{\sum a_i s_i}=\operatorname{LM}(f)
= y^{k_F}\cdot \operatorname{LM}(p)\cdot \operatorname{LM}(q)
\succeq \x^{\mathbf{a}} y^{k_F+t}\cdot \operatorname{LM}(q).
\]
Since the total $\x$-degree on both sides must match, $\operatorname{LM}(q)$ must contain no $x_i$'s; otherwise the $\x$-part of $\operatorname{LM}(f)$ would exceed $\x^{\mathbf{a}}$. Hence
\[
\operatorname{LM}(q)=y^r \quad \text{for some } r\ge 0. \tag{5}
\]

Now, in the order $x_n \succ \cdots \succ x_1 \succ y$, any monomial containing some $x_i$ is strictly greater than any power of $y$. Since $\operatorname{LM}(q)=y^r$, the polynomial $q$ cannot contain any term involving $x_i$; otherwise such a term would be larger than $y^r$ and would be the leading monomial. Thus
\[
q(\x,y)\in\K[y]. \tag{6}
\]

If $q$ is nonconstant, since $\fs{F}$ has no pure $y$-factor, $q$ has a nonzero root $\alpha \neq 0$ in $\overline{\K}$ (algebraic closure of $\K$). Substituting $y=\alpha$ into $\fs{F}=pq$ gives
\[
\fs{F}(\x,\alpha)=p(\x,\alpha)q(\alpha)=0.
\]
But
\[
\fs{F}(\x,\alpha)=\alpha^{-k_F}F(\alpha^{s_1}x_1,\dots,\alpha^{s_n}x_n)\neq 0,
\]
because $\alpha\neq 0$ and the map $x_i\mapsto \alpha^{s_i}x_i$ is an automorphism of $\overline{\K}[\x]$. Contradiction. Hence $q$ is constant, and $\fs{F}$ is irreducible.
\end{proof}

\begin{lemma}\label{lem:gcd-separation}
Let $A, B \in \K[x_1,\dots,x_n]$ with $\K$ a field and let $G = \gcd(A, B)$. Then, for any $\mathbf{s} \in \mathbb{N}^n$,
\[
\fs{G} =  \gcd(\fs{A}, \fs{B}).
\]
\end{lemma}

\begin{proof}
Write the factorizations of $A$ and $B$ as
\[
A = G \cdot A_1, \qquad B = G \cdot B_1,
\]
where
\[
A_1 = \prod_{i=1}^p P_i^{a_i}, \qquad
B_1 = \prod_{j=1}^q Q_j^{b_j},
\]
with $P_i, Q_j\in \K[\x]$ irreducible, $\gcd(A_1, B_1) = 1$, and all $P_i,Q_j$ are distinct.

Applying the separation transformation to $A$ and $B$, we obtain
\[
\fs{A} = \fs{G} \cdot \fs{A_1}, \qquad \fs{B} = \fs{G} \cdot \fs{B_1}.
\]
By Lemma~\ref{lem:irreducible-separation}, the separation map preserves irreducibility. Moreover, since $\gcd(A_1, B_1) = 1$, the sets of irreducible factors of $A_1$ and $B_1$ are disjoint. By Lemma~\ref{lem:irreducible-separation}, their images under the separation map remain disjoint. Hence
\[
\gcd(\fs{A_1}, \fs{B_1}) = 1.
\]
Therefore,
\[
\gcd(\fs{A}, \fs{B}) = \gcd(\fs{G} \cdot \fs{A_1}, \fs{G} \cdot \fs{B_1})
= \fs{G} \cdot \gcd(\fs{A_1}, \fs{B_1})
= \fs{G}.
\]
This proves the lemma.
\end{proof}

\begin{theorem}[Schwartz-Zippel Lemma]\label{thm:sz}
Let $\K$ be a field, and let $P(x_1,\dots,x_n) \in \K[x_1,\dots,x_n]$ be a nonzero polynomial of total degree $D$. Let $S \subseteq \K$ be a finite subset, and let $b_1,\dots,b_n$ be independently and uniformly chosen from $S$. Then
\[
\Pr\bigl(P(b_1,\dots,b_n) \neq 0\bigr) \ge 1 - \frac{D}{|S|}.
\]
\end{theorem}

\section{Derivative-Assisted Hensel Lifting}\label{sec:lifting}

\subsection{The Separation Technique}\label{sec:separation}

In practice, a multivariate polynomial is typically expressed in $n$ variables, say $x_1,\dots,x_n$. In this paper, however, we introduce an auxiliary variable $y$. This serves two purposes. First, it allows us to designate a distinguished main variable for lifting. Second, it embeds the original $n$-variate polynomial into an $(n+1)$-variate one, namely in $x_1,\dots,x_n,y$, enabling us to distinguish different monomials by their $y$-degrees. This technique, called the \textbf{separation technique}, was introduced by Huang and Gao in their work on sparse multivariate polynomial factorization over integers \cite{HuangGao2023Factorization}.

For a polynomial $F(x_1, \dots, x_n)$, we introduce a new variable $y$ via the substitution
\[
x_i \longmapsto x_i \cdot y^{s_i},
\]
where $s_i$ are randomly chosen nonnegative integers. Substituting into the original polynomial $F$, we obtain a new polynomial in $n+1$ variables.

Let us analyze the form of each term after substitution. Let
\[
F = c_1 M_1 + \dots + c_t M_t,
\]
with $t$ nonzero terms, and
$M_i = x_1^{e_{i1}} \cdots x_n^{e_{in}}.$
After the substitution $x_i = x_i \cdot y^{s_i}$, the monomial $M_i$ becomes
\[
M_i \cdot y^{s_1 e_{i1} + s_2 e_{i2} + \dots + s_n e_{in}}.
\]
Thus, the substituted polynomial can be written as
\[
F(x_1 y^{s_1},\dots,x_n y^{s_n}) = c_1 M_1 y^{d_1} + c_2 M_2 y^{d_2} + \dots + c_t M_t y^{d_t},
\]
where
$d_i = s_1 e_{i1} + s_2 e_{i2} + \dots + s_n e_{in}.$
As long as we choose $s_1, \dots, s_n$ such that all $d_1, \dots, d_t$ are pairwise distinct, the new polynomial is separated with respect to $y$.

To this end, construct the following $n$-variate polynomial in $s_1, \dots, s_n$:
\[
S(s_1, \dots, s_n) = \prod_{i \neq j} (d_i - d_j).
\]
This polynomial is nonzero if and only if all $d_i$ are pairwise distinct. Since there are $t$ terms, the degree of the product is $\frac{t(t-1)}{2}$.

For instance, if we randomly choose $s_i$ as integers between $0$ and $t^2$, then by the Schwartz-Zippel lemma, the probability that $S(s_1, \dots, s_n) \neq 0$ is at least $1/2$. In other words, with high probability all $d_i$ are distinct, and hence the substituted $F$ is separated with respect to $y$.

Recall the separation transformation $\fs{\cdot}$ defined in Definition \ref{def-separated}: for a polynomial $F \in \K[x_1, \dots, x_n]$ and a vector $\mathbf{s} \in \mathbb{N}^n$,
\[
\fs{F}(\x, y)= \frac{F(x_1 y^{s_1}, \dots, x_n y^{s_n})}{y^{k_F}},
\]
where $k_F$ is the lowest power of $y$ in the numerator.

The normalization factor $y^{k_F}$ does not affect the separation property: subtracting the same integer from all $y$-degrees preserves their pairwise distinctness. Hence $\fs{F}(\x,y)$ is separated with respect to $y$ if and only if the substituted polynomial $F(x_1 y^{s_1}, \dots, x_n y^{s_n})$ has pairwise distinct $y$-degrees.

This full separation assumption is relaxed in our main algorithm. As we shall see in Sections~\ref{sec:params}, the actual linear-complexity algorithm only requires partial separation: in each iteration, at least half of the remaining terms are separated, while the colliding terms are detected and deferred to subsequent iterations. This reduces the range of $s_i$ from $O(T^2)$ to $O(T)$, yielding the asymptotically linear bound $\widetilde{O}(n \cdot T \cdot D)$.

\subsection{Recovering Separated Polynomials via Derivatives (Algorithm~\ref{alg:derivative_recovery})}
\label{sec:recovery}

Let
$f = \sum_{i=1}^{t} c_i \cdot m_i(\x) \cdot y^{d_i}$
be $y$-separated. We consider two different cases.

\subsubsection*{Case 1: $\R$ Contains Sufficiently Many Distinct Prime Elements}

Assume $\R$ is a UFD and contains sufficiently many distinct prime elements, say $b_1, \dots, b_n$, satisfying:
\begin{itemize}
\item[(1)] $b_1, \dots, b_n$ are pairwise coprime,
\item[(2)] each $b_j$ is coprime to all coefficients $c_i$.
\end{itemize}

Substituting $x_1, \dots, x_n$ for $b_1, \dots, b_n$, respectively, and denoting $\b = (b_1, \dots, b_n)$, we obtain
\[
f(\b, y) = \sum_{i=1}^{t} c_i \cdot b_1^{e_{i1}} \cdots b_n^{e_{in}} \cdot y^{d_i}.
\]

Let the coefficient of $y^{d_i}$ be
\[
C_i := c_i \cdot b_1^{e_{i1}} \cdots b_n^{e_{in}} \in \R \setminus \{0\}.
\]

Since $b_1, \dots, b_n$ are pairwise distinct prime elements and are coprime to $c_i$, by unique factorization, $C_i$ can be uniquely factored in $\R$ as
\[
C_i = u_i \cdot b_1^{e_{i1}} \cdots b_n^{e_{in}},
\]
where $u_i \in \R$ contains no prime factors $b_1, \dots, b_n$ (i.e., $u_i$ is coprime to $b_1, \dots, b_n$).

Thus:
\begin{itemize}
\item By counting the powers of $b_1, \dots, b_n$ in $C_i$, we can uniquely recover the exponents $e_{i1}, \dots, e_{in}$;
\item The remaining part $u_i$ is exactly the coefficient $c_i$ (since $c_i$ is coprime to $b_1, \dots, b_n$).
\end{itemize}

\begin{remark}
Thus, each term $c_i \cdot x_1^{e_{i1}} \cdots x_n^{e_{in}} \cdot y^{d_i}$ of the original polynomial can be fully recovered.
To ensure that the exponents $e_{i1}, \dots, e_{in}$ are not misidentified in the recovery process, it is necessary that $c_i$ does not contain any prime factors $b_1, \dots, b_n$. Otherwise, if $c_i$ also contains some $b_j$, one cannot distinguish $c_i$ from the powers in $m_i$, leading to incorrect exponent recovery.
Since this factorization of each $c_i$ contains only finitely many prime elements, as long as the ring $\R$ contains sufficiently many prime elements, randomly chosen $b_1, \dots, b_n$ will with high probability not divide any $c_i$. When $\R = \mathbb{Z}$, the primes (which are prime elements) are infinite, so this condition is naturally satisfied. However, the requirement of "sufficiently many prime elements" does not hold for all integral domains; for example, finite fields $\mathbb{F}_q$ contain no prime elements. In such cases, the method fails and requires further treatment.
\end{remark}

\subsubsection*{Case 2: $\R$ Does Not Contain Sufficiently Many Prime Elements}

From now on, let $\R=\K$ be a field. Since a field lacks sufficiently many prime elements, the evaluation $f(\b, y)$ alone cannot uniquely recover $f$. In this case, we can resort to partial derivative information. Specifically, assume:

\begin{itemize}
    \item[(1)] $f$ is $y$-separated;
    \item[(2)] The characteristic $p$ of $\K$ is greater than the total degree of $f$, or $\operatorname{char}(\K)=0$;
    \item[(3)] $f(\b, y)$ is known;
    \item[(4)] $\dfrac{\partial f}{\partial x_k}(\b, y)$ is known for $k=1,\dots,n$.
\end{itemize}

Substituting $x_j = b_j$ into $f$ (where $b_j$ are arbitrary nonzero elements of $\K$), we get
\[
f(\b, y) = \sum_{i=1}^{t} c_i \cdot b_1^{e_{i1}} \cdots b_n^{e_{in}} \cdot y^{d_i}.
\]

Taking the partial derivative with respect to $x_k$ ($k=1,\dots,n$), we have
$\frac{\partial f}{\partial x_k}(\x, y) = \sum_{i=1}^{t} c_i \cdot \frac{\partial m_i}{\partial x_k}(\x) \cdot y^{d_i},$
where
$\frac{\partial m_i}{\partial x_k}(\x) = e_{ik} \cdot x_1^{e_{i1}} \cdots x_k^{e_{ik}-1} \cdots x_n^{e_{in}}.$
Substituting $x_j = b_j$ yields
\[
\frac{\partial f}{\partial x_k}(\b, y) = \sum_{i=1}^{t} c_i \cdot e_{ik} \cdot b_1^{e_{i1}} \cdots b_k^{e_{ik}-1} \cdots b_n^{e_{in}} \cdot y^{d_i}.
\]

For each fixed $y$-power $d_i$, comparing the coefficients of $y^{d_i}$ in $f(\b, y)$ and $\frac{\partial f}{\partial x_k}(\b, y)$:
\begin{align*}
A_i &:= c_i \cdot b_1^{e_{i1}} \cdots b_n^{e_{in}}, \\
B_{i,k} &:= c_i \cdot e_{ik} \cdot b_1^{e_{i1}} \cdots b_k^{e_{ik}-1} \cdots b_n^{e_{in}}.
\end{align*}

Assuming $b_k \neq 0$ and $A_i \neq 0$, we have
$\frac{B_{i,k}}{A_i} = \frac{e_{ik}}{b_k}.$
Since $e_{ik}$ is a nonnegative integer and $b_k$ is known, the right-hand side is an element of $\K$. If the characteristic of $\K$ is sufficiently large ($p$ greater than the degree of $f$) or zero, then $e_{ik}$ can be uniquely identified in $\K$ (e.g., via division). Thus we can solve for
\[
e_{ik} = \frac{B_{i,k}}{A_i} \cdot b_k.
\]

Repeating for $k=1,\dots,n$, we obtain all exponents $e_{i1},\dots,e_{in}$. Then using $A_i = c_i \cdot b_1^{e_{i1}} \cdots b_n^{e_{in}}$, we can solve for
\[
c_i = A_i \cdot \bigl( b_1^{e_{i1}} \cdots b_n^{e_{in}} \bigr)^{-1}.
\]

Thus $c_i \cdot m_i(x_1,\dots,x_n)$ is fully recovered, and hence the entire $f$ is recovered.

\paragraph{Remark:} This method requires:
\begin{itemize}
    \item $b_1,\dots,b_n$ such that all $b_1^{e_{i1}}\cdots b_n^{e_{in}}$ are invertible (typically $b_j \ne 0$);
    \item The characteristic of $\K$ is $0$ or greater than the degree of $f$, to ensure that the integers $e_{ik}$ can be uniquely distinguished in $\K$ (to avoid the situation where $e_{ik}$ cannot be uniquely determined in $\K$ due to small characteristic $p$; specifically, from $\frac{B_{i,k}}{A_i} = \frac{e_{ik}}{b_k}$ we can only determine the image of $e_{ik}$ in $\K$. If $p \le e_{ik}$, then $e_{ik}$ and $e_{ik}+p$ are indistinguishable in $\K$, preventing recovery of the true integer value of $e_{ik}$ and thus losing information.)
\end{itemize}

We now present this as pseudocode for convenient later invocation, and refer to it as Algorithm \ref{alg:derivative_recovery}.

\begin{algorithm}
\caption{Derivative-Based Polynomial Recovery Algorithm}
\label{alg:derivative_recovery}
\begin{algorithmic}[1]
\Require
    \begin{itemize}
        \item Evaluation point $\b = (b_1, \dots, b_n) \in \K^{*n}$, where $\K$ is a field (characteristic $0$ or greater than the total degree of $f$);
        \item $f(\b, y) = \sum_{i=1}^{t} A_i y^{d_i}$, where the $d_i$ are distinct and $f$ is $y$-separated;
        \item $\dfrac{\partial f}{\partial x_k}(\b, y) = \sum_{i=1}^{t} B_{i,k} y^{d_i}$ for $k = 1, \dots, n$.
    \end{itemize}
\Ensure 
    $f(\x, y) = \sum_{i=1}^{t} c_i \cdot x_1^{e_{i1}} \cdots x_n^{e_{in}} \cdot y^{d_i}$.

\For{$i = 1$ to $t$}
    \State Read coefficient $A_i$ and degree $d_i$ from $f(\b,y)$
    \For{$k = 1$ to $n$}
        \State Read coefficient $B_{i,k}$ from $\dfrac{\partial f}{\partial x_k}(\b,y)$
        \State Compute $e_{ik} = \dfrac{B_{i,k}}{A_i} \cdot b_k$
    \EndFor
    \State Compute $c_i = A_i \cdot \left( b_1^{e_{i1}} \cdots b_n^{e_{in}} \right)^{-1}$
    \State Construct monomial $m_i(\x) = x_1^{e_{i1}} \cdots x_n^{e_{in}}$
\EndFor

\State Output $f(\x,y) = \sum_{i=1}^{t} c_i \cdot m_i(\x) \cdot y^{d_i}$
\State \Return $f(\x,y)$
\end{algorithmic}
\end{algorithm}

\begin{theorem}\label{the-2}
Algorithm \ref{alg:derivative_recovery} correctly recovers the complete expression of $f(\x, y)$, with time complexity
\[
O(n \cdot t \cdot \log D)
\]
field operations in $\K$, where $D$ is the total degree of $f$ and $t$ is the number of terms of $f$.
\end{theorem}

\begin{proof}
From the expression of $f$, substituting $x_j = b_j$ gives
\[
f(\b, y) = \sum_{i=1}^{t} \left( c_i \cdot b_1^{e_{i1}} \cdots b_n^{e_{in}} \right) y^{d_i} = \sum_{i=1}^{t} A_i y^{d_i},
\]
so $A_i = c_i \cdot b_1^{e_{i1}} \cdots b_n^{e_{in}}$.
Taking the partial derivative with respect to $x_k$ and substituting $\b$ gives
\[
\frac{\partial f}{\partial x_k}(\b, y) = \sum_{i=1}^{t} \left( c_i \cdot e_{ik} \cdot b_1^{e_{i1}} \cdots b_k^{e_{ik}-1} \cdots b_n^{e_{in}} \right) y^{d_i} = \sum_{i=1}^{t} B_{i,k} y^{d_i},
\]
so $B_{i,k} = c_i \cdot e_{ik} \cdot b_1^{e_{i1}} \cdots b_k^{e_{ik}-1} \cdots b_n^{e_{in}}$. (Note: if $e_{i,k}=0$, then $B_{i,k}=0$.)

Since the $d_i$ are distinct, we can match $A_i$ with $B_{i,k}$ by the exponent of $y$. Computing
\[
\frac{B_{i,k}}{A_i} = \frac{e_{ik}}{b_k} \quad \Rightarrow \quad e_{ik} = \frac{B_{i,k}}{A_i} \cdot b_k.
\]
The characteristic condition on $\K$ ensures that $e_{ik}$ can be uniquely identified in $\K$. After obtaining all $e_{ik}$, substituting back into $A_i$ gives
\[
c_i = A_i \cdot \left( b_1^{e_{i1}} \cdots b_n^{e_{in}} \right)^{-1}.
\]
Thus the algorithm correctly recovers all coefficients $c_i$ and exponents $e_{ik}$, and hence the complete expression of $f$.

Now we analyze the complexity.
For each $i = 1, \dots, t$, the algorithm performs:
\begin{itemize}
    \item Reading $A_i$ and $d_i$: $O(1)$;
    \item For $k = 1, \dots, n$, reading $B_{i,k}$ and computing $e_{ik} = \frac{B_{i,k}}{A_i} \cdot b_k$: $O(n)$ field operations;
    \item Computing $c_i = A_i \cdot (b_1^{e_{i1}} \cdots b_n^{e_{in}})^{-1}$ requires computing $O(n)$ powers, with a complexity of $O(n \log D)$ field operations.
\end{itemize}
Thus, the complexity per iteration is $O(n\log D)$, and the total complexity is $O(t \cdot n \cdot \log D)$ field operations.
\end{proof}

\subsection{Extracting Factor Derivatives by $z^2$-Lifting (Algorithm~\ref{alg:henselonestep})}
\label{sec:z2-lifting}

For convenience, we introduce the operator
\[
{\rm S}(g) := \frac{\partial g}{\partial x_1} + \cdots + \frac{\partial g}{\partial x_n}.
\]

The linearity of this operator is demonstrated by the following lemma, whose proof is straightforward:

\begin{lemma}\label{lm-1}
If $g = g_1 + \cdots + g_k\in \K[x_1,\dots,x_n,y]$, then
\(
{\rm S}(g) = {\rm S}(g_1) + \cdots + {\rm S}(g_k).
\)
\end{lemma}

The following lemma relates ${\rm S}(g)$ to the expansion of the translated polynomial modulo $z^2$:

\begin{lemma}\label{lm-4}
Let $g\in \K[x_1,\dots,x_n,y]$ and $\b=(b_1,\dots,b_n)\in \K^{n}$. Suppose $z$ is a new indeterminate. Then
\[
g(z+b_1,\dots,z+b_n,y) \equiv \mathrm{S}(g)(\b,y)\, z + g(\b,y) \pmod{z^2}.\]
\end{lemma}
\begin{proof}
By linearity (Lemma~\ref{lm-1}), it suffices to consider a single term $g=c\,x_1^{e_1}\cdots x_n^{e_n}\,y^d$. Define
\[
\mathcal{G}(z)=c\,(z+b_1)^{e_1}\cdots (z+b_n)^{e_n}\,y^d.
\]
Expanding $\mathcal{G}(z)$ as a Taylor polynomial in $z$ about $z=0$ gives
$\mathcal{G}(z)=\mathcal{G}(0)+\mathcal{G}'(0)z+O(z^2),$
hence
\[
\mathcal{G}(z)\equiv \mathcal{G}(0)+\mathcal{G}'(0)z \pmod{z^2}.
\]
A direct computation yields
$\mathcal{G}(0)=c\,b_1^{e_1}\cdots b_n^{e_n}\,y^d=f(\b,y),
$
and
$\mathcal{G}'(0)=c\,y^d\sum_{i=1}^n e_i b_i^{e_i-1}\prod_{j\ne i} b_j^{e_j}.$

Thus
$\mathcal{G}'(0)=c\,y^d\sum_{i=1}^n e_i b_1^{e_1}\cdots b_i^{e_i-1}\cdots b_n^{e_n}
= \mathrm S(g)(\b,y),$
where the last equality follows from the definition of $\mathrm S(g)$. Therefore,
\[
\mathcal{G}(z)\equiv g(\b,y)+\mathrm S(g)(\b,y)\,z \pmod{z^2}.
\]
By linearity, the lemma holds for all polynomials $g$.
\end{proof}

Similarly, if we apply a different translation coefficient to the $k$-th variable while keeping the others unchanged, we have:

\begin{lemma}\label{2shift}
Let $g\in \K[x_1,\dots,x_n,y]$ and $\b=(b_1,\dots,b_n)\in \K^{n}$. Suppose $z$ is a new indeterminate. Then
\[
g(z+b_1,\dots,2z+b_k,\dots,z+b_n,y) \equiv 
\left(\mathrm S(g)+\frac{\partial g}{\partial x_k}\right)(\b,y)\, z + g(\b,y) \pmod{z^2}.
\]
\end{lemma}

\begin{proof}
By linearity (Lemma~\ref{lm-1}), it suffices to consider a monomial
$g=c\,x_1^{e_1}\cdots x_n^{e_n}\,y^d$. Define
\[
\mathcal{G}(z)=c\,(z+b_1)^{e_1}\cdots (2z+b_k)^{e_k}\cdots (z+b_n)^{e_n}\,y^d.
\]
Hence
$\mathcal{G}(z)\equiv \mathcal{G}(0)+\mathcal{G}'(0)z \pmod{z^2}.$
Now
$\mathcal{G}(0)=c\,b_1^{e_1}\cdots b_n^{e_n}\,y^d=g(\b,y).$
By the product rule, the derivative at $0$ is
\[
\mathcal{G}'(0)=c\,y^d\sum_{i=1}^n e_i b_i^{e_i-1}\prod_{j\ne i} b_j^{e_j}
+ c\,y^d\, e_k b_k^{e_k-1}\prod_{j\ne k} b_j^{e_j}.
\]
The first sum is precisely $\mathrm S(g)(\b,y)$, and the extra term is
$c\,y^d\, e_k b_1^{e_1}\cdots b_k^{e_k-1}\cdots b_n^{e_n}
= \frac{\partial g}{\partial x_k}(\b,y).$
Therefore,
\[
\mathcal{G}'(0)=\left(\mathrm S(g)+\frac{\partial g}{\partial x_k}\right)(\b,y),
\]
and hence
\[
\mathcal{G}(z)\equiv g(\b,y)+\left(\mathrm S(g)+\frac{\partial g}{\partial x_k}\right)(\b,y)\,z \pmod{z^2}.
\]
By linearity, the lemma holds for all $g$.
\end{proof}

\paragraph{Remark:} 
In the above proofs, the expression for $\mathrm S(g)(\b,y)$ involves only nonnegative powers of the $b_i$'s; the $i$-th term vanishes when $e_i = 0$, since the corresponding partial derivative is zero. Hence the result is valid for arbitrary $b_i \in \K$.

Combining the above two lemmas, we obtain:
\begin{lemma}
Let $g\in \K[x_1,\dots,x_n,y]$ and $\b=(b_1,\dots,b_n)\in \K^{n}$. Suppose $z$ is a new indeterminate. If
\[
g(z+b_1,\dots,z+b_n,y) \pmod{z^2} \quad\text{and}\quad g(z+b_1,\dots,2z+b_k,\dots,z+b_n,y) \pmod{z^2}
\]
are known, then we can compute $g(\b,y)$ and $\dfrac{\partial g}{\partial x_k}(\b,y)$.
\end{lemma}

\begin{proof}
Define
\[
\mathcal{A}(z) := g(z+b_1,\dots,z+b_n,y), \qquad
\mathcal{B}(z) := g(z+b_1,\dots,2z+b_k,\dots,z+b_n,y).
\]
By Lemma~\ref{lm-4} and Lemma \ref{2shift}, we have the expansions modulo $z^2$:
\[
\mathcal{A}(z) \equiv g(\b,y) + {\rm S}(g)(\b,y) \cdot z \pmod{z^2},
\]
\[
\mathcal{B}(z) \equiv g(\b,y) + \bigl({\rm S}(g)+\tfrac{\partial g}{\partial x_k}\bigr)(\b,y) \cdot z \pmod{z^2}.
\]

Write these two known expansions as
\[
\mathcal{A}(z) \equiv a_0 + a_1 z \pmod{z^2}, \qquad
\mathcal{B}(z) \equiv b_0 + b_1 z \pmod{z^2},
\]
where $a_0, a_1, b_0, b_1$ are polynomials in $y$.

Comparing the coefficients of $z^0$, we get
$$a_0 = b_0 = g(\b,y).$$
Comparing the coefficients of $z^1$, we get
$a_1 = {\rm S}(g)(\b,y), b_1 = \bigl({\rm S}(g)+\tfrac{\partial g}{\partial x_k}\bigr)(\b,y).$
Subtracting the two gives
\[
b_1 - a_1 = \frac{\partial g}{\partial x_k}(\b,y).
\]

Therefore, from the known $\mathcal{A}(z) \bmod z^2$ and $\mathcal{B}(z) \bmod z^2$, we can obtain $g(\b,y)$ by reading the coefficient of $z^0$, and obtain $\frac{\partial g}{\partial x_k}(\b,y)$ by reading the coefficients of $z^1$ and subtracting.
\end{proof}

In Algorithm~\ref{alg:henselonestep}, we present the Hensel lifting algorithm that lifts a factorization from modulo $z$ to modulo $z^2$ (see \textit{Modern Computer Algebra} \cite{vonzurGathen2013}). The input consists of bivariate polynomials in $\K[z,y]$; the algorithm takes a factorization over $\K[y]$ and lifts it to a factorization over $\R[y]$, where $\R=\K[z]/\langle z^2\rangle$.

\begin{algorithm}
\caption{One-Step Hensel Lifting}
\label{alg:henselonestep}
\begin{algorithmic}[1]
\Require
    \begin{itemize}
        \item Polynomial $\mathcal{F} \in \K[z,y]$, where $\K$ is a field, and $\operatorname{lc}_y(\mathcal{F})$ is not a zero divisor modulo $z$;
        \item Polynomials $\mathcal{G}_0,\mathcal{H}_0\in \K[y]$ satisfying $\mathcal{F} \equiv \mathcal{G}_0 \mathcal{H}_0 \pmod z$;
        \item Polynomials $u, v \in \K[y]$ satisfying $u \mathcal{G}_0 + v \mathcal{H}_0 = 1$;
        \item $\mathcal{G}_0$ is monic, $\deg_y \mathcal{F} = \deg_y \mathcal{G}_0 + \deg_y \mathcal{H}_0$;
        \item $\deg_y u < \deg_y \mathcal{H}_0$, $\deg_y v < \deg_y \mathcal{G}_0$.
    \end{itemize}
\Ensure
    Polynomials $\mathcal{G}_1, \mathcal{H}_1 \in \K[z,y]$ satisfying
    \[
    \mathcal{F} \equiv \mathcal{G}_1 \mathcal{H}_1 \pmod{z^2},
    \]
    with $\mathcal{G}_1$ monic, $\mathcal{G}_1 \equiv \mathcal{G}_0 \pmod{z}$, $\mathcal{H}_1 \equiv \mathcal{H}_0 \pmod{z}$, and
    \[
    \deg_y \mathcal{G}_1 = \deg_y \mathcal{G}_0,\quad \deg_y \mathcal{H}_1 = \deg_y \mathcal{H}_0.
    \]

\State Compute $e \equiv \mathcal{F} - \mathcal{G}_0 \mathcal{H}_0 \pmod{z^2}$.
\State Compute $v \cdot e$, then perform division with remainder with respect to $\mathcal{G}_0$ in $y$, obtaining $q, r \in \K[z,y]$ satisfying
\[
v e \equiv q \mathcal{G}_0 + r \pmod{z^2},\quad \deg_y r < \deg_y \mathcal{G}_0.
\]
\State Compute $\mathcal{H}_1 \equiv \mathcal{H}_0 + u e + q \mathcal{H}_0 \pmod{z^2}$.
\State Compute $\mathcal{G}_1 \equiv \mathcal{G}_0 + r \pmod{z^2}$.

\State \Return $\mathcal{G}_1, \mathcal{H}_1$
\end{algorithmic}
\end{algorithm}

Let $$f(\x,y)=g(\x,y)h(\x,y)\in \K[\x,y],$$ where $g(\x,y)$ is monic in $y$, and let $\b=(b_1,\dots,b_n)\in \K^n$. Denote $z+\b=(z+b_1,\dots,z+b_n)$. Then we have
\[
f(z+\b, y) \equiv g(z+\b, y)\, h(z+\b, y) \pmod{z^i}
\]
for any $i\in \N$. For convenience, set
\[
\mathcal{F}(z,y):=f(z+\b, y),\qquad
\mathcal{G}(z,y):=g(z+\b, y),\qquad
\mathcal{H}(z,y):=h(z+\b, y).
\]
Since
\[
f(z+\b, y)\equiv f(\b,y),\qquad
g(z+\b, y)\equiv g(\b,y),\qquad
h(z+\b, y)\equiv h(\b,y) \pmod{z},
\]
we define
\[
\mathcal{G}_0(y):=g(\b,y),\qquad \mathcal{H}_0(y):=h(\b,y).
\]
Assume further that $g(\b,y)$ and $h(\b,y)$ are coprime, i.e., there exist $u(y),v(y)\in\K[y]$ such that
\[
u(y)g(\b,y)+v(y)h(\b,y)=1.
\]
Then, by applying Algorithm~\ref{alg:henselonestep} to $\mathcal{F}(z,y)$ with the initial factorization $\mathcal{G}_0\mathcal{H}_0$, we obtain the lift $\mathcal{G}_1(z,y):=g(z+\b,y)$ modulo $z^2$.

Similarly, for each $k=1,\dots,n$, consider the alternative translation where the $k$-th variable is shifted by $2z$:
\[
\mathcal{F}^{(k)}(z,y):=f(z+b_1,\dots,2z+b_k,\dots,z+b_n,y),
\]
\[
\mathcal{G}^{(k)}(z,y):=g(z+b_1,\dots,2z+b_k,\dots,z+b_n,y),
\]
\[
\mathcal{H}^{(k)}(z,y):=h(z+b_1,\dots,2z+b_k,\dots,z+b_n,y).
\]
Again, since
\[
\mathcal{F}^{(k)}(z,y)\equiv f(\b,y),\qquad
\mathcal{G}^{(k)}(z,y)\equiv g(\b,y),\qquad
\mathcal{H}^{(k)}(z,y)\equiv h(\b,y) \pmod{z},
\]
the same initial data $\mathcal{G}_0,\mathcal{H}_0$ apply. Applying Algorithm~\ref{alg:henselonestep} to $\mathcal{F}^{(k)}(z,y)$ yields the lift $\mathcal{G}^{(k)}(z,y)$ modulo $z^2$.

Consequently, by carrying out these two distinct translations and expanding to order $z^2$, we obtain simultaneously the evaluation $g(\b,y)$ and the partial derivative $\frac{\partial g}{\partial x_k}(\b,y)$. Given these two quantities, the full expression of $g$ can be recovered from these univariate polynomials, provided that $g$ is $y$-separated.

Thus, differentiation is naturally connected to Hensel lifting (only requiring lifting to modulo $z^2$): by taking the difference of two translations, we can extract derivative information in each direction, thereby providing complete input for recovering separable polynomials.

\subsection*{On the Monicity Condition}

There is a subtle point: Hensel lifting typically requires that the factor being lifted is monic in \(y\). However, \(g(z+b_1,\dots,z+b_n,y)\) is not necessarily monic.

To address this, we assume that the factorization is performed in the rational function field \(\K(\x)[y]\), i.e.,
\[
f(\x,y) = g(\x,y) h(\x,y) \in \K(\x)[y],
\]
and \(g(\x,y) \in \K(\x)[y]\) is monic in \(y\). If \(g\) is not monic in \(y\), we simply transfer its leading coefficient in \(y\) to \(h\).

In this case, \(g(\x,y)\) can be written as
\[
g(\x,y) = y^{d_s} + g_{s-1} y^{d_{s-1}} + \cdots + g_1 y^{d_1},
\]
where \(d_s > \cdots > d_1\), and the \(g_i\) are rational functions in \(x_1,\dots,x_n\). Clearly, \(g(\b,y)\) is also monic in \(y\).

This treatment affects our previous conclusions, so we need to adjust the following two lemmas.

\begin{lemma}\label{lm-2}
Let \(\K\) be a field and $\b=(b_1,\dots,b_n)\in\K^n$, and let \(g(\x,y) = \dfrac{g_1(\x,y)}{g_2(\x,y)}\in \K(\x,y)\) be a rational function, where \(g_1, g_2 \in \K[\x,y]\). Then
\[
g(z+b_1,\dots,z+b_n,y) \equiv g(\b,y) + S(g)(\b,y) \cdot z \pmod{z^2}.
\]
\end{lemma}

\begin{lemma}
Let \(\K\) be a field and $\b=(b_1,\dots,b_n)\in\K^n$, and let \(g(\x,y) = \dfrac{g_1(\x,y)}{g_2(\x,y)}\in \K(\x,y)\) be a rational function, where \(g_1, g_2 \in \K[\x,y]\). Then
\[
g(z+b_1,\dots,2z+b_k,\dots,z+b_n,y) \equiv g(\b,y) + \bigl(S(g)+\tfrac{\partial g}{\partial x_k}\bigr)(\b,y) \cdot z \pmod{z^2}.
\]
\end{lemma}
The proofs are given in the appendix.

\subsection{Computing Partial Derivatives of Factors (Algorithm~\ref{alg:factorparde})}
\label{sec:factor-derivative}

The following theorem is the core result of this section. While the Hensel lifting method above already allows us to compute $g(\b,y)$ and its partial derivatives $\frac{\partial g}{\partial x_k}(\b,y)$, it requires introducing an auxiliary variable $z$ and performing computations modulo $z^2$. We now take a closer look at the explicit steps of the Hensel lift and show that the procedure can be dramatically simplified: the auxiliary variable $z$ can be eliminated entirely. Specifically, we derive a direct expression for $\frac{\partial g}{\partial x_k}(\b,y)$ that depends only on $\frac{\partial f}{\partial x_k}(\b,y)$, together with the univariate evaluations $g(\b,y)$ and $h(\b,y)$. In other words, once the partial derivatives of $f$ at $\b$ are known, the partial derivatives of its factor $g$ at the same point can be computed directly, bypassing the full Hensel lifting computation.

We introduce the following notation:  for polynomials $f, g \in \R[y]$ with $g \neq 0$ and monic, performing polynomial division with remainder with respect to $y$ gives
\[
f = q(y) g + r(y), \qquad \deg_y r < \deg_y g \text{ or } r = 0,
\]
where the quotient is $q(y) = \mathbf{quo}(f, g)$ and the remainder is $r(y) = \mathbf{rem}(f, g)$.

\begin{theorem}[Factor Derivative Theorem]\label{the-1}
Let
$f(\x, y) = g(\x, y) \cdot h(\x, y) \in \K(\x)[y],$
where \(g(\x, y)\) is monic in \(y\), and there exist \(u, v \in \K[y]\) satisfying the B\'ezout identity (equivalent to \(g(\b, y)\) and \(h(\b, y)\) being coprime):
$u \cdot g(\b, y) + v \cdot h(\b, y) = 1.$
Then
\[
\frac{\partial g}{\partial x_k}(\b, y) = 
\mathbf{rem}\bigl(v \cdot \frac{\partial f}{\partial x_k}(\b, y),\; g(\b, y)\bigr),
\]
and
\[
\frac{\partial h}{\partial x_k}(\b, y) = u \cdot \frac{\partial f}{\partial x_k}(\b, y) + 
\mathbf{quo}\bigl(v \cdot \frac{\partial f}{\partial x_k}(\b, y),\; g(\b, y)\bigr) \cdot h(\b, y).
\]
\end{theorem}

\begin{proof}
For convenience, denote
\[
\mathcal{G}_0(y) = g(\b, y), \quad \mathcal{H}_0(y) = h(\b, y), \quad \mathcal{F}_0(y) = f(\b, y) = \mathcal{G}_0(y) \mathcal{H}_0(y).
\]

Introduce the translation variable \(z\) and define
\[
\mathcal{F}(z, y) := f(z+\b, y),
\]
\[
\mathcal{G}(z, y) := g(z+\b, y),
\]
\[
\mathcal{H}(z, y) := h(z+\b, y).
\]
Clearly \(\mathcal{F}(z,y) = \mathcal{G}(z,y) \mathcal{H}(z,y)\).
By Lemma \ref{lm-2},
\[
\mathcal{G}(z, y) \equiv \mathcal{G}_0(y) + {\rm S}(g)(\b,y) \cdot z \pmod{z^2},
\]
\[
\mathcal{H}(z, y) \equiv \mathcal{H}_0(y) + {\rm S}(h)(\b,y) \cdot z \pmod{z^2},
\]
\[
\mathcal{F}(z, y) \equiv \mathcal{F}_0(y) + {\rm S}(f)(\b,y) \cdot z \pmod{z^2}.
\]

By assumption, there exist \(u(y), v(y)\) satisfying the B\'ezout identity
\[
u \mathcal{G}_0 + v \mathcal{H}_0 = 1,
\]
and we may choose \(\deg_y u < \deg_y \mathcal{H}_0\), \(\deg_y v < \deg_y \mathcal{G}_0\).

Substituting \(\mathcal{F}, \mathcal{G}_0, \mathcal{H}_0, u, v\) into the Hensel lifting algorithm (Algorithm \ref{alg:henselonestep}), we obtain \(\mathcal{G}_1, \mathcal{H}_1\) satisfying:
\[
\mathcal{F} \equiv \mathcal{G}_1 \mathcal{H}_1 \pmod{z^2},
\]
\[
\mathcal{G}_1 \equiv \mathcal{G}_0 \pmod{z}, \quad \mathcal{H}_1 \equiv \mathcal{H}_0 \pmod{z},
\]
\[
\deg_y \mathcal{G}_1 = \deg_y \mathcal{G}_0, \quad \deg_y \mathcal{H}_1 = \deg_y \mathcal{H}_0,
\]
and \(\mathcal{G}_1\) is monic.

Now we derive the explicit expression for \(\mathcal{G}_1\) according to Algorithm \ref{alg:henselonestep}.

Step 1: Compute \(e = \mathcal{F} - \mathcal{G}_0 \mathcal{H}_0\). Since \(\mathcal{F}_0 = \mathcal{G}_0 \mathcal{H}_0\), we have
\[
e \equiv \mathcal{F}_0 + {\rm S}(f)(\b,y) \cdot z - \mathcal{G}_0 \mathcal{H}_0 \equiv {\rm S}(f)(\b,y) \cdot z \pmod{z^2}.
\]

Step 2: Perform division with remainder of \(v e\) by \(\mathcal{G}_0\). Since \(v e\) contains \(z\) as a factor, we have
\[
r = \mathbf{rem}(v e, \mathcal{G}_0) = \mathbf{rem}\bigl(v {\rm S}(f)(\b,y), \mathcal{G}_0\bigr) \cdot z,
\]
\[
q = \mathbf{quo}(v e, \mathcal{G}_0) = \mathbf{quo}\bigl(v {\rm S}(f)(\b,y), \mathcal{G}_0\bigr) \cdot z.
\]

Thus
\[
\mathcal{G}_1 = \mathcal{G}_0 + r = \mathcal{G}_0 + \mathbf{rem}\bigl(v {\rm S}(f)(\b,y), \mathcal{G}_0\bigr) \cdot z,
\]
\[
\mathcal{H}_1 = \mathcal{H}_0 + u e + q \mathcal{H}_0 = \mathcal{H}_0 + u {\rm S}(f)(\b,y) \cdot z + \mathbf{quo}\bigl(v {\rm S}(f)(\b,y), \mathcal{G}_0\bigr) \cdot z \mathcal{H}_0.
\]

Comparing the expansions of \(\mathcal{G}(z,y)\) and \(\mathcal{G}_1\) modulo \(z^2\), and noting that \(\mathcal{G}_1 \equiv \mathcal{G} \pmod{z^2}\) and \(\deg_y \mathcal{G}_1 = \deg_y \mathcal{G}_0\), we obtain
\[
{\rm S}(g)(\b,y) = \mathbf{rem}\bigl(v {\rm S}(f)(\b,y), \mathcal{G}_0\bigr). \tag{1}
\]

Now consider another translation:
\[
\mathcal{G}^{(k)}(z, y) := g(z+b_1, \dots, 2z+b_k, \dots, z+b_n, y),
\]
\[
\mathcal{H}^{(k)}(z, y) := h(z+b_1, \dots, 2z+b_k, \dots, z+b_n, y),
\]
\[
\mathcal{F}^{(k)}(z, y) := f(z+b_1, \dots, 2z+b_k, \dots, z+b_n, y) = \mathcal{G}^{(k)} \mathcal{H}^{(k)}.
\]

Repeating the above derivation for \(\mathcal{F}^{(k)}\), we similarly obtain
\[
\bigl({\rm S}(g) + \tfrac{\partial g}{\partial x_k}\bigr)(\b,y) = \mathbf{rem}\bigl(v ({\rm S}(f) + \tfrac{\partial f}{\partial x_k})(\b,y), \mathcal{G}_0\bigr). \tag{2}
\]

Subtracting (1) from (2), we get
\[
\frac{\partial g}{\partial x_k}(\b,y) = \mathbf{rem}\bigl(v \cdot \tfrac{\partial f}{\partial x_k}(\b,y), \mathcal{G}_0\bigr).
\]

Finally, from the B\'ezout relation \(u \mathcal{G}_0 + v \mathcal{H}_0 = 1\) and the division identity \(v e = q \mathcal{G}_0 + r\), we further derive the partial derivative expression for \(H\):
\[
\frac{\partial h}{\partial x_k}(\b,y) = u \cdot \frac{\partial f}{\partial x_k}(\b,y) + \mathbf{quo}\bigl(v \cdot \frac{\partial f}{\partial x_k}(\b,y), \mathcal{G}_0\bigr) \cdot \mathcal{H}_0.
\]

This completes the proof.
\end{proof}

\begin{remark}
That is, \(\frac{\partial g}{\partial x_k}(\b, y)\) is the remainder of dividing \(v \cdot \frac{\partial f}{\partial x_k}(\b, y)\) by \(g(\b, y)\), while \(\frac{\partial h}{\partial x_k}(\b, y)\) is given by the corresponding quotient \(q\). The asymmetry in the expressions arises from our assumption that \(g(\b,y)\) is monic, while \(h(\b,y)\) is not necessarily monic. Therefore, \(\frac{\partial g}{\partial x_k}\) can be obtained directly from the remainder of the division, while \(\frac{\partial h}{\partial x_k}(\b, y)\) requires an additional term involving the quotient and \(h(\b,y)\). If the monicity assumption is swapped, the roles of the formulas are interchanged.
\end{remark}

Algorithm~\ref{alg:factorparde} presents the corresponding algorithm for computing the partial derivatives of a factor polynomial.

\begin{algorithm}
\caption{Factor Partial Derivative Algorithm}
\label{alg:factorparde}
\begin{algorithmic}[1]
\Require
    \begin{itemize}
        \item Evaluation point $\b=(b_1, \dots, b_n) \in \K^n$;
        \item Polynomial $f(\x,y) = g(\x,y) \cdot h(\x,y)$ (where $f$ is known, $g,h$ are unknown, $f\in \K[\x,y]$, $g,h\in \K(\x)[y]$, and $g$ is monic in $y$);
        \item Univariate polynomials $\mathcal{G}_0(y) = g(\b,y)$, $\mathcal{H}_0(y) = h(\b,y)$;
        \item B\'ezout coefficients $u(y), v(y) \in \K[y]$ satisfying $u \mathcal{G}_0 + v \mathcal{H}_0 = 1$;
        \item Partial derivatives $\dfrac{\partial f}{\partial x_k}(\b,y)$, $k = 1, \dots, n$.
    \end{itemize}
\Ensure
    Partial derivatives $\dfrac{\partial g}{\partial x_k}(\b,y)$ and $\dfrac{\partial h}{\partial x_k}(\b,y)$ for $k = 1, \dots, n$.

\For{$k = 1$ to $n$}
    \State Compute $p_k(y) := v(y) \cdot \dfrac{\partial f}{\partial x_k}(\b,y)$.
    \State Perform division with remainder: $p_k(y) = q_k(y) \cdot \mathcal{G}_0(y) + r_k(y)$, where $\deg_y r_k < \deg_y \mathcal{G}_0$.
    \State Set $\dfrac{\partial g}{\partial x_k}(\b,y) := r_k(y)$.
    \State Set $\dfrac{\partial h}{\partial x_k}(\b,y) := u(y) \cdot \dfrac{\partial f}{\partial x_k}(\b,y) + q_k(y) \cdot \mathcal{H}_0(y)$.
\EndFor

\State \Return $\left\{ \dfrac{\partial g}{\partial x_k}(\b,y),\ \dfrac{\partial h}{\partial x_k}(\b,y) \right\}_{k=1}^{n}$
\end{algorithmic}
\end{algorithm}

\begin{theorem}
Algorithm \ref{alg:factorparde} correctly recovers $\frac{\partial g}{\partial x_k}(\b,y)$ and $\frac{\partial h}{\partial x_k}(\b,y)$, with time complexity
$
\widetilde{O}\bigl(n \cdot D_y\bigr)
$
field operations, where $D_y = \deg_y f$.
\end{theorem}

\begin{proof}
By Theorem \ref{the-1} (Factor Derivative Theorem), for each $k = 1,\dots,n$, we have
\[
\frac{\partial g}{\partial x_k}(\b,y) = \mathbf{rem}\bigl(v \cdot \tfrac{\partial f}{\partial x_k}(\b,y),\; g(\b,y)\bigr),
\]
\[
\frac{\partial h}{\partial x_k}(\b,y) = u \cdot \tfrac{\partial f}{\partial x_k}(\b,y) + \mathbf{quo}\bigl(v \cdot \tfrac{\partial f}{\partial x_k}(\b,y),\; g(\b,y)\bigr) \cdot h(\b,y).
\]
In the algorithm, $p_k(y) = v(y) \cdot \frac{\partial f}{\partial x_k}(\b,y)$. The division with remainder gives $q_k = \mathbf{quo}(p_k, \mathcal{G}_0)$ and $r_k = \mathbf{rem}(p_k, \mathcal{G}_0)$. Thus
\[
\frac{\partial g}{\partial x_k}(\b,y) = r_k,\qquad 
\frac{\partial h}{\partial x_k}(\b,y) = u \cdot \tfrac{\partial f}{\partial x_k}(\b,y) + q_k \cdot \mathcal{H}_0.
\]
Therefore, the algorithm outputs correctly.

Now we analyze the complexity.
For each $k$, the algorithm performs:
\begin{itemize}
    \item Computing $v(y) \cdot \frac{\partial f}{\partial x_k}(\b,y)$: polynomial multiplication, complexity $\widetilde{O}(D_y)$ field operations;
    \item Division with remainder of $p_k$ by $\mathcal{G}_0$: complexity $\widetilde{O}(D_y)$ field operations;
    \item Computing $u(y) \cdot \frac{\partial f}{\partial x_k}(\b,y)$: complexity $\widetilde{O}(D_y)$ field operations;
    \item Computing $q_k \cdot \mathcal{H}_0$: complexity $\widetilde{O}(D_y)$ field operations;
    \item Addition: $\widetilde{O}(D_y)$ field operations.
\end{itemize}
Since the loop runs $n$ times, the total complexity is $\widetilde{O}(nD_y)$ field operations.
\end{proof}

\section{GCD Algorithm over a Field}
\label{sec:field-gcd}

Before presenting the rigorous GCD algorithm, we first illustrate the core idea underlying our approach: \emph{derivative-driven Hensel lifting}. This heuristic exposition reveals the essential mechanism and justifies the subsequent formal development.

\subsection{A Heuristic View}\label{sec-HenselMoti}

Let us consider a hypothetical scenario. Suppose we have a polynomial $F \in \K[\x]$ known to factor as
\[
F(\x) = G(\x) \cdot H(\x),
\]
where $G$ and $H$ are coprime and sparse. We aim to recover the factor $G$.

Applying the separation transformation introduced in Definition \ref{def-separated}, we obtain
\[
\fs{F}(\x, y) = \fs{G}(\x, y) \cdot \fs{H}(\x, y),
\]
where $\fs{F}(\x, y) = F(x_1 y^{s_1}, \dots, x_n y^{s_n}) / y^{k_F}$. With high probability over the random choice of $\mathbf{s}$, $\fs{G}$ is $y$-separated.

Now choose an evaluation point $\b \in \K^n$ and consider the univariate polynomials
\[
\mathcal{F}_0(y) := \fs{F}(\b, y), \qquad \mathcal{G}_0(y) := \fs{G}(\b, y), \qquad \mathcal{H}_0(y) := \fs{H}(\b, y).
\]
Since $\mathcal{F}_0 = \mathcal{G}_0 \mathcal{H}_0$ and $\gcd(\mathcal{G}_0, \mathcal{H}_0) = 1$ for a good choice of $\b$, the classical Hensel lifting can lift the factorization modulo $z$ to modulo $z^2$:
\[
\fs{F}(z+\b, y) \equiv \fs{G}(z+\b, y) \cdot \fs{H}(z+\b, y) \pmod{z^2}.
\]

The key insight, formalized in Lemma~\ref{lm-4}, is that the coefficient of $z$ in this lifted expansion carries derivative information:
\[
\fs{G}(z+\b, y) \equiv \mathcal{G}_0(y) + \sum_{k=1}^n \frac{\partial \fs{G}}{\partial x_k}(\b, y) \cdot z \pmod{z^2}.
\]

By applying the Hensel lifting step to two different translations---one with all variables shifted by $z$, and another with the $k$-th variable shifted by $2z$---we obtain the differences of the $z$-coefficients, which directly yield the partial derivatives $\frac{\partial \fs{G}}{\partial x_k}(\b, y)$.

More precisely, suppose that in addition to $\mathcal{G}_0(y)$, we have obtained the partial derivative evaluations
\[
\frac{\partial \fs{G}}{\partial x_k}(\b, y), \qquad k = 1, \dots, n.
\]
Then, since $\fs{G}$ is $y$-separated, we can invoke the derivative-based interpolation algorithm (Algorithm~\ref{alg:derivative_recovery}) to recover the complete polynomial $\fs{G}(\x, y)$. Setting $y = 1$ then gives the original factor $G(\x)$.

This heuristic can be summarized as follows:

\begin{itemize}
    \item \textbf{Step 1 (Separation):} Introduce $y$ via random $\mathbf{s}$ to make $\fs{G}$ $y$-separated.
    \item \textbf{Step 2 (Evaluation):} Choose random $\b$ such that $\mathcal{G}_0$ and $\mathcal{H}_0$ are coprime.
    \item \textbf{Step 3 (Derivative recovery via Hensel lifting):} Apply one-step Hensel lifting to two translations, and extract $\frac{\partial \fs{G}}{\partial x_k}(\b, y)$ for all $k$.
    \item \textbf{Step 4 (Sparse interpolation):} Recover $\fs{G}(\x, y)$ from its values and derivatives at $\b$ using Algorithm~\ref{alg:derivative_recovery}.
    \item \textbf{Step 5 (Output):} Set $y = 1$ to obtain $G(\x)$.
\end{itemize}

\paragraph{Remark:} 
The heuristic described above serves as a guide to the intuition behind our method. It is not intended as a standalone algorithm. The formal algorithm is presented in Section~\ref{sec:recursive}.

\subsection{GCD Reduction to Hensel Lifting}

The heuristic above assumes that $\mathcal{G}_0(y) = \fs{G}(\b, y)$ is known, as required by the Hensel lifting step. In a general factorization setting, $\mathcal{G}_0$ is unknown: to determine it, one would need to factor $\mathcal{F}_0(y)$ completely and then select the correct subset of irreducible factors that corresponds to $\fs{G}$. This \emph{factor selection problem} is nontrivial and, in general, is itself a difficult combinatorial task.

However, in the context of \emph{GCD computation}, this obstacle can be circumvented. Given two polynomials $A, B \in \K[\x]$ with $G = \gcd(A, B)$, we have
\[
A = G \cdot A_1, \qquad B = G \cdot B_1, \qquad \gcd(A_1, B_1) = 1.
\]
After separation and evaluation, we can compute
\[
\mathcal{G}_0(y) = \operatorname{monic}\bigl( \gcd( \fs{A}(\b, y), \fs{B}(\b, y) ) \bigr)
\]
directly, without any factor selection. This is the key observation that enables a fully rigorous and efficient algorithm.

Before presenting the algorithmic details, we first show how the problem of computing the GCD of two multivariate polynomials can be reduced to the Hensel lifting framework developed in the preceding sections.

The crucial question is whether the two factors $G$ and $A_1+cB_1$ in the factorization
\[
F = A + cB = G \cdot (A_1 + c B_1)
\]
are coprime, as this is the prerequisite for applying Hensel lifting.
Since $\gcd(A_1, B_1) = 1$, we have
\[
\gcd(G,\; A_1 + c B_1) = 1
\]
for all but a small number of exceptional choices of $c$. Indeed, any irreducible factor of $G$ that divides $A_1 + cB_1$ would force a common factor between $A_1$ and $B_1$, unless $c$ takes a specific value.

The following theorem makes this precise: there are at most $\deg G$ bad choices of $c$ for which $\gcd(G,\; A_1 + cB_1) \neq 1$.

\begin{lemma}\label{thm:bad-c}
Let $\K$ be a field, let $P \in \K[x_1, \dots, x_n]$ be irreducible and nonconstant, and let $A_1, B_1 \in \K[x_1, \dots, x_n]$ with $P \nmid A_1$ and $P \nmid B_1$. Then there is at most one $c \in \K$ such that $P \mid (A_1 + c B_1)$.
\end{lemma}

\begin{proof}
Suppose, for contradiction, that there exist two distinct constants $c_1, c_2 \in \K$ such that
\[
P \mid (A_1 + c_1 B_1), \qquad P \mid (A_1 + c_2 B_1).
\]
Subtracting the two divisibility relations gives
\[
P \mid (c_1 - c_2) B_1.
\]
Since $c_1 \neq c_2$, we have $c_1 - c_2 \in \K^*$, which is a unit in $\K[x_1, \dots, x_n]$. Therefore,
\[
P \mid B_1,
\]
contradicting the assumption $P \nmid B_1$. Hence, there is at most one $c \in \K$ such that $P \mid (A_1 + c B_1)$.
\end{proof}

\begin{theorem}\label{thm:good-c}
Let $\K$ be a field, let $A, B \in \K[x_1, \dots, x_n]$ and $G = \gcd(A, B)$. Write $A = G A_1$, $B = G B_1$. Then there are at most $\deg G$ values of $c \in \K$ such that
\[
\gcd(G,\; A_1 + c B_1) \neq 1.
\]

\end{theorem}

\begin{proof}
Let $G = P_1 \cdots P_r$ be the irreducible factorization of $G$. For each irreducible factor $P_i$, Lemma~\ref{thm:bad-c} implies that there is at most one value of $c$ such that $P_i \mid (A_1 + c B_1)$. Hence, the total number of bad $c$'s is at most $r \le \deg G$.
\end{proof}

\paragraph{Remark:} 
Theorem~\ref{thm:good-c} ensures that by choosing $c$ uniformly at random from a sufficiently large subset $S$ of $\K$, the probability that $c$ is bad is at most $\deg G / |S|$, which can be made arbitrarily small by choosing $|S|$ large enough. In particular, if $\K$ is a finite field with $|\K| > 2\deg G$, a random $c$ is good with probability at least $1/2$. For infinite fields, the probability of choosing a bad $c$ is zero in the measure-theoretic sense.

Once a good $c$ is chosen, we set
\[
F = A + c B = G \cdot (A_1 + c B_1),
\]
where the two factors are coprime. This places us exactly in the Hensel lifting framework: we have a polynomial $F$ that factors as $F = G \cdot H$ with $G = G$ and $H = A_1 + c B_1$, and $\gcd(G, H) = 1$.

Moreover, at a suitable evaluation point $\b \in \K^n$ (i.e., one for which the univariate GCD does not lose degree or acquire spurious common factors), the univariate factor $\mathcal{G}_0(y)$ corresponding to $G$ can be computed directly as
\[
\mathcal{G}_0(y) = \operatorname{monic}\bigl( \gcd( \fs{A}(\b, y),\; \fs{B}(\b, y) ) \bigr),
\]
without any factor selection ambiguity. This is the crucial advantage of specializing to the GCD problem.

The following sections develop the complete algorithm based on this reduction.

\subsection{Conditions on the Random Vector $\mathbf{s}$}
\label{sec:params}

The heuristic described in Section~\ref{sec-HenselMoti} hinges on a strong assumption: all terms of $\fs{G}$ must be fully separated with respect to $y$. As established in Section~\ref{sec:separation}, achieving such complete separation requires choosing the random vector $\mathbf{s}$ from a range of size $O(T^2)$. Consequently, the $y$-degree of $\fs{F}$ becomes $O(T^2 D)$, introducing a quadratic dependence on $T$ that undermines our goal of linear complexity.

To overcome this $T^2$ bottleneck, we develop a recursive approximation strategy. Rather than attempting to separate all terms of $G$ in a single transformation, we proceed iteratively, recovering a subset of terms at each step. Suppose we already have an approximation $G^*$ of $G$ (initially $G^* = 0$). From this current approximation, we compute an improved approximation $G^{**}$ that recovers additional terms of $G$ beyond those already captured by $G^*$.

\subsubsection*{Condition 1: Leading Coefficient Separation}

A subtle but important issue must be addressed at the outset: the polynomial recovered by our procedure is determined only up to multiplication by a single term (i.e., a monomial with a nonzero coefficient). This ambiguity is inherent to the Hensel lifting framework. Recall that to apply the lifting step, we require the factor $\fs{G}$ to be monic with respect to the main variable $y$. Starting from the factorization
\[
F(\x) = G(\x) \cdot H(\x) \in \K[\x],
\]
we first apply the separation transformation from Definition \ref{def-separated} to obtain
\[
\fs{F}(\x, y) = \fs{G}(\x, y) \cdot \fs{H}(\x, y) \in \K[\x, y],
\]
where $\fs{F}(\x, y) = F(x_1 y^{s_1}, \dots, x_n y^{s_n}) / y^{k_F}$, and similarly for $\fs{G}$ and $\fs{H}$.

If $\fs{G}$ is not already monic in $y$, we normalize it by dividing by its leading coefficient in $y$:
\[
\mathcal{G}_0(\x, y): = \frac{\fs{G}(\x, y)}{\operatorname{lc}_y(\fs{G}(\x, y))},
\qquad
\mathcal{H}_0(\x, y):= \operatorname{lc}_y(\fs{G}(\x, y)) \cdot \fs{H}(\x, y).
\]
This normalization ensures that $\mathcal{G}_0$ is monic in $y$ and hence amenable to Hensel lifting.

Our recursive strategy does not require $\fs{G}$ to be fully $y$-separated. Instead, we impose a significantly weaker condition: the leading coefficient $\operatorname{lc}_y(\fs{G}(\x, y))$ must be a single term in $\x$. This constitutes \textbf{Condition~1} in our choice of the random vector $\mathbf{s}$.

Dividing by this single term makes $\mathcal{G}_0$ monic in $y$; moreover, each coefficient of $\mathcal{G}_0$ is obtained from the corresponding coefficient of $\fs{G}$ by dividing by a single term in $\x$. Hence the coefficients of $\mathcal{G}_0$ are Laurent polynomials in $\x$ with no increase in the number of terms. This property is essential for preventing intermediate expression swell, which is one of the main obstacles in classical Hensel lifting algorithms.

Consequently, $\mathcal{G}_0$ differs from $\fs{G}$ only by a term factor (namely $\operatorname{lc}_y(\fs{G})$), and hence from the original factor $G$ as well. This monomial ambiguity is harmless, as it is systematically handled by the alignment procedure described below.

\subsubsection*{Condition 2: Half-Remaining Separation}

Condition~1 only guarantees that the leading term is unambiguous; other terms of $\fs{G}$ may still collide after the $y$-separation. The derivative-based interpolation algorithm (Algorithm~\ref{alg:derivative_recovery}) can only recover those terms that are non-colliding in the current separation. As a result, in a single iteration we can only recover a proper subset of the terms of $G$.

To recover the full factor $G$, we therefore proceed iteratively: in each round, we recover a subset of the remaining terms, while the colliding terms are detected and deferred to subsequent iterations. To formally describe this process of gradually recovering more terms, we introduce the following notions of approximation and improvement.

\begin{definition}[Approximation]\label{def-appro}
Let $G^* \in \K[\x]$ be a polynomial. If there exists a single term $\alpha$ such that every term of $\alpha \cdot G^*$ is also a term of $G$, then $G^*$ is called an \textbf{approximation} of $G$, denoted $G^* \sqsubseteq G$. In particular, the zero polynomial is an approximation of any $G$.
\end{definition}

\begin{definition}[Improvement]
Let $G^*$ and $G^{**}$ be approximations of $G$. Suppose there exist single terms $\alpha_1, \alpha_2$ such that $\alpha_1 G^*$ and $\alpha_2 G^{**}$ are both partial sums of $G$. If
\[
\|G - \alpha_2 G^{**}\|_0 < \|G - \alpha_1 G^*\|_0,
\]
where $\|\cdot\|_0$ denotes the number of nonzero terms of a polynomial, then $G^{**}$ is called an \textbf{improvement} of $G^*$ (equivalently, a \textbf{better approximation}).
\end{definition}

In each iteration, starting from the current approximation $G^*$, we choose a new random vector $\mathbf{s}$ and invoke the derivative-based interpolation algorithm (Algorithm~\ref{alg:derivative_recovery}), together with the collision detection mechanism (Section~\ref{sec:collision}), to compute an improved approximation $G^{**}$. The goal is to recover a substantial portion of the terms not yet captured by $G^*$.

\begin{definition}[Non-Colliding Terms]\label{def:nc}
Let 
$F = c_1 M_1 + c_2 M_2 + \cdots + c_t M_t \in \K[x_1, \dots, x_n],$
where $c_i \in \K^*$ and $M_i = x_1^{e_{i1}} \cdots x_n^{e_{in}}$ are distinct monomials. 
For a vector $\mathbf{s} = (s_1, \dots, s_n) \in \mathbb{N}^n$, let 
\[
\fs{F}(\x, y) = \frac{F(x_1 y^{s_1}, \dots, x_n y^{s_n})}{y^{k_F}}
\]
be the transformed polynomial, where $k_F$ is the lowest power of $y$ in the numerator. 
For each term $c_i M_i$, let
\[
d_i = s_1 e_{i1} + s_2 e_{i2} + \cdots + s_n e_{in} - k_F
\]
be its $y$-exponent in $\fs{F}$. 
We say that $c_i M_i$ is \textbf{non-colliding} with respect to $\mathbf{s}$ if its $y$-exponent is unique among all terms of $\fs{F}$, i.e.,
\[
d_i \neq d_j \quad \text{for all } j \neq i.
\]
The set of non-colliding terms of $F$ with respect to $\mathbf{s}$ is denoted by
\[
\operatorname{NC}(F, \mathbf{s}) = \{\, c_i M_i \mid d_i \neq d_j \text{ for all } j \neq i \,\}.
\]
\end{definition}

For this to yield asymptotically linear complexity in $T$ (up to the logarithmic factor from $O(\log T)$ iterations), we need that in each iteration at least half of the remaining terms are non-colliding and thus recoverable. More precisely, suppose the current approximation $G^*$ satisfies $G^* \sqsubseteq G$ (Definition~\ref{def-appro}), meaning that there exists a single term $\alpha_1$ such that $\alpha_1 G^*$ is a partial sum of terms of $G$. Let
\[
R = G - \alpha_1 G^*
\]
denote the remaining terms after aligning the current approximation with $G$. We require that the random vector $\mathbf{s}$ be chosen such that at least half of the terms of $R$ are non-colliding after the separation transformation:
\[
\#\operatorname{NC}(R, \mathbf{s}) \ge \frac{1}{2} \cdot \|R\|_0.
\]
This property, which we refer to as \textbf{Condition 2} (\emph{half-remaining separation}) with respect to $R = G - \alpha_1 G^*$, ensures that a single iteration can recover at least half of the currently missing terms.

The core invariant of our iterative strategy is therefore
\[
\bigl\| G - \alpha_2 G^{**} \bigr\|_0 \;\le\; \frac{1}{2} \bigl\| G - \alpha_1 G^* \bigr\|_0,
\]
where the terms $\alpha_1, \alpha_2$ are the aligning factors provided by the definition of approximation. Since the number of missing terms decreases by a factor of at least two in each round, $O(\log t)$ iterations suffice to recover all terms of $G$, where $t = \|G\|_0 \le T$.

\subsubsection*{Condition 3: Lowest-Term Separation}

There is a subtle obstacle in implementing this strategy. To recover the non-colliding terms of $R = G - \alpha_1 G^*$, we would need to evaluate $\fs{R}(\b, y)$ and its partial derivatives at $\b$. While we can compute the normalized univariate polynomial
\[
\mathcal{G}_0(\b, y) = \operatorname{monic}\bigl( \gcd( \fs{A}(\b, y), \fs{B}(\b, y) ) \bigr)
= \frac{\fs{G}(\b, y)}{\operatorname{lc}_y(\fs{G}(\b, y))}
\]
directly from the univariate GCD, and $\fs{G^*}(\b, y)$ by direct computation from the known approximation $G^*$, several obstacles remain. First, the aligning term $\alpha_1$ is unknown. Second, the $y$-degrees of $\fs{G}$ and $\fs{G^*}$ may not coincide, since the separation transformation introduces different shifts for different polynomials. Third, even if $\alpha_1$ and the $y$-degree shift were known, we only have access to $\mathcal{G}_0(\b, y)$, not $\fs{G}(\b, y)$ itself; the leading coefficient $\operatorname{lc}_y(\fs{G}(\b, y))$ is also unknown and must be determined to reconstruct the true factor. 

To overcome this difficulty, we impose a third condition on the random vector $\mathbf{s}$, which we call the \textbf{lowest-term separation condition}. Recall that $G$ can be written as
\[
G = c_1 M_1 + c_2 M_2 + \cdots + c_t M_t,
\]
where $M_1 \prec M_2\prec \cdots \prec M_t$ and $M_1$ is the lowest monomial with respect to some fixed monomial order. We require that $\mathbf{s}$ be chosen such that $c_1 M_1$ is non-colliding in $\fs{G}$:
\[
c_1 M_1 \in \operatorname{NC}(G, \mathbf{s}).
\]
This is \textbf{Condition 3}.

Assume that the vector $\mathbf{s}_{\text{prev}}$ used in the previous iteration also satisfies Condition~3, so that the current approximation $G^*$ already contains the lowest term of $G$ (up to a single term factor). In the current iteration, we take a new vector $\mathbf{s}$ that likewise satisfies Condition~3. Then this lowest term serves as a canonical anchor for alignment.

We do not have direct access to $\fs{G}(\b, y)$; instead, we can only compute the normalized polynomial
\[
\mathcal{G}_0(\b, y) = \operatorname{monic}\bigl( \gcd( \fs{A}(\b, y), \fs{B}(\b, y) ) \bigr)
= \frac{\fs{G}(\b, y)}{\operatorname{lc}_y(\fs{G}(\b, y))}.
\]
By Condition~1, the leading coefficient $\operatorname{lc}_y(\fs{G})$ is a single term, so $\mathcal{G}_0$ differs from $\fs{G}$ only by a single term factor. 

Applying Algorithm~\ref{alg:derivative_recovery} to $\mathcal{G}_0(\b, y)$ and its partial derivatives 
$\frac{\partial \mathcal{G}_0}{\partial x_k}(\b, y)$ for $k=1,\dots,n$ recovers all non-colliding terms of $\mathcal{G}_0$, yielding a polynomial $\mathcal{G}_0^{\text{non}}(\x, y)$.

The partial derivatives $\frac{\partial \mathcal{G}_0}{\partial x_k}(\b, y)$ are obtained from the Factor Derivative Theorem (Theorem~\ref{the-1}). Recall that $\mathcal{G}_0(\x, y) = \fs{G}(\x, y) / \operatorname{lc}_y(\fs{G}(\x, y))$ is the normalized factor, and
\[
\mathcal{H}_0(\x, y) = \operatorname{lc}_y(\fs{G}(\x, y)) \cdot \fs{H}(\x, y) .
\]
Then 
\[
\fs{F}(\x, y) = \fs{G}(\x, y) \cdot \fs{H}(\x, y)=\mathcal{G}_0(\x, y) \cdot \mathcal{H}_0(\x, y).
\]

Now, by Theorem~\ref{the-1}, for each $k = 1,\dots,n$,
\[
\frac{\partial \mathcal{G}_0}{\partial x_k}(\b, y)
=
\mathbf{rem}\left(v(y) \cdot \frac{\partial \fs{F}}{\partial x_k}(\b, y),\; \mathcal{G}_0(y)\right),
\]
where $u(y), v(y)$ are the B\'ezout coefficients satisfying
\[
u(y) \mathcal{G}_0(\b,y) + v(y) \mathcal{H}_0(\b,y) = 1.
\]

Condition 3 guarantees that the lowest term of $\mathcal{G}_0^{\text{non}}$ corresponds exactly to the lowest term of $\fs{G^*}$. We compare this recovered lowest term with the lowest term of $\fs{G^*}$ (which we can compute directly from $G^*$) and determine:
\begin{itemize}
    \item a single term $\gamma = c \mathbf{x}^{\mathbf{e}}$ in $\x$, to align the $\x$-parts and coefficients;
    \item an integer $\delta$, to align the $y$-degrees,
\end{itemize}
such that
\[
\gamma \cdot \mathcal{G}_0^{\text{non}}(\x, y) \cdot y^{\delta} \quad \text{and} \quad \fs{G^*}(\x, y)
\]
have the same lowest term: same $\x$-part, same coefficient, and same $y$-degree.

Here $\gamma$ absorbs the leading coefficient of $\fs{G}$ and any monomial discrepancy in $\x$, while $\delta$ accounts for any $y$-degree shift between the two polynomials.

Define 
\[
\Delta(\x, y) = \gamma(\x) \cdot \mathcal{G}_0(\x, y) \cdot y^{\delta} - \fs{G^*}(\x, y).
\]
We then form the aligned difference at the evaluation point $\b$:
\[
\Delta(\b, y) = \gamma(\b) \cdot \mathcal{G}_0(\b, y) \cdot y^{\delta} - \fs{G^*}(\b, y),
\]
where $\gamma(\b)$ denotes the evaluation of the monomial $\gamma$ at $\b$. Since $\gamma(\b) \mathcal{G}_0(\b, y) y^{\delta}$ and $\fs{G^*}(\b, y)$ share the same lowest term by construction, their difference has no lowest term, and the non-colliding terms of this difference correspond exactly to the non-colliding terms of $\Delta(\x,y)$, which in turn correspond to the non-colliding terms of $R = G - \alpha_1 G^*$ (up to a nonzero single term factor). 

Similarly, we compute the corresponding partial derivatives. Since $\gamma$ depends only on $\x$, its partial derivative with respect to $x_k$ is zero if $e_k = 0$, and $\partial \gamma/\partial x_k = e_k \cdot \gamma / x_k$ otherwise:
\[
\frac{\partial \Delta}{\partial x_k}(\b, y)
=
\frac{\partial \gamma}{\partial x_k}(\b) \cdot \mathcal{G}_0(\b, y) \cdot y^{\delta}
+ \gamma(\b) \cdot \frac{\partial \mathcal{G}_0}{\partial x_k}(\b, y) \cdot y^{\delta}
- \frac{\partial \fs{G^*}}{\partial x_k}(\b, y),
\quad k = 1,\dots,n.
\]
Applying the derivative-based interpolation algorithm (Algorithm~\ref{alg:derivative_recovery}) to $\Delta(\b, y)$ and its partial derivatives recovers exactly the non-colliding terms of the difference. Let this recovered polynomial be $U$. Then we set
\[
G^{**} = G^* + U.
\]
By Condition~2, at least half of the terms of $R = G - \alpha_1 G^*$ are non-colliding, so $U$ recovers at least half of the missing terms. Hence
\[
\|G - \alpha_1 G^{**}\|_0 = \|G - \alpha_1 G^*\|_0 - \#U \le \frac{1}{2} \|G - \alpha_1 G^*\|_0,
\]
i.e., $G^{**}$ is an improvement of $G^*$ in the sense of Definition~\ref{def-appro}.

\subsubsection*{Summary of Conditions on the Random Vector $\mathbf{s}$}\label{sec-condition}

The three conditions imposed on the random vector $\mathbf{s}$ are formulated as follows:

\begin{enumerate}
    \item \textbf{Leading coefficient separation condition}: Let $G \in \K[x_1,\dots,x_n]$. Write
\[
\fs{G} = G_1 y^{d_1} + G_2 y^{d_2} + \dots + G_k y^{d_k},
\]
where $d_1 < d_2 < \dots < d_k$, and $G_i$ are polynomials (not necessarily monomials). If $G_k=\operatorname{lc}_y(\fs{G})$ is a single term, then $G$ is said to be leading-coefficient separated with respect to $\mathbf{s}$.

   \item \textbf{Half-remaining separation condition}: Let $G^* \sqsubseteq G$ be an approximation and let $\alpha$ be a single term such that every term of $\alpha G^*$ appears in $G$. Write
\[
R = G - \alpha G^*.
\]
If $\mathbf{s}$ is chosen such that
\[
\#\operatorname{NC}(R, \mathbf{s}) \ge \frac{\|R\|_0}{2},
\]
then $\mathbf{s}$ is said to be half-remaining separated with respect to the current approximation $G^*$.

    \item \textbf{Lowest term separation condition}: Write
\[
G = c_1 M_1 + c_2 M_2 + \dots + c_t M_t,
\]
where $c_i \neq 0$, $M_i$ are monomials, and ordered by  lexicographic order with $M_1 \preceq \cdots \preceq M_t$, with $M_1$ being the lowest monomial. If $\mathbf{s}$ is chosen such that $c_1 M_1$ does not collide in $\fs{G}$, i.e., $c_1 M_1 \in {\rm NC}(G,\mathbf{s})$, then $G$ is said to be lowest-term separated with respect to $\mathbf{s}$.
 
\end{enumerate}

Theorem~\ref{thm:s-prob} below establishes that a random vector $\mathbf{s}$, drawn from a suitable distribution, satisfies all three conditions simultaneously with overwhelming probability.

\subsubsection*{Probability Analysis}

We now formally state the three conditions that $\mathbf{s}$ must satisfy, prove that a random choice satisfies all three with high probability, and derive the required range for the $s_i$'s.
Let $G \in \K[x_1, \dots, x_n]$ be the target factor (the GCD), with $\|G\|_0 \le T$, and let $G^*$ be a current approximation satisfying $G^* \sqsubseteq G$ (Definition~\ref{def-appro}). Let $\mathbf{s} = (s_1, \dots, s_n) \in \mathbb{N}^n$ be a random vector.

With reference to the three conditions on $\mathbf{s}$ stated above, we first establish three auxiliary lemmas, each bounding the failure probability of one respective condition.

\begin{lemma}[Condition 1]\label{lem:leading}
Let $G \in \K[x_1, \dots, x_n]$ be a nonzero polynomial with at most $T$ terms. For any $\varepsilon \in (0,1)$, if $\mathbf{s} = (s_1, \dots, s_n)$ is chosen uniformly from $[1, N]^n$ with
$N = \left\lceil (T-1)/{\varepsilon} \right\rceil,$
then the probability that $\operatorname{lc}_y(\fs{G})$ is a single term is at least $1 - \varepsilon$.
\end{lemma}

\begin{proof}
Let \(G = \sum_{i=1}^{t} a_i x_1^{e_{i1}} \cdots x_n^{e_{in}}\), where \(t \le T\), and each \(a_i \in \K^*\). For a vector \(\mathbf{s} = (s_1,\dots,s_n) \in \mathbb{N}^n\), define the exponents
\[
d_{\mathbf{s},i} = e_{i1}s_1 + \cdots + e_{in}s_n, \qquad i = 1,\dots,t.
\]
These are the \(y\)-exponents of the terms in $G(x_1 y^{s_1},\dots,x_n y^{s_n})$ (before removing the lowest power).

As shown in the proof of Theorem 2.10 in \cite{HuangGao2025}, the set of maximizers
\[
\mathcal{C} = \{ (\mathbf{s}, d_{\mathbf{s},\max}) : \mathbf{s} \in \overline{S} \}
\]
can be decomposed into convex polyhedral cones \(\mathcal{C} = \cup_{i=1}^{\ell} Q_i\), where \(\ell \le t\), and each \(Q_i\) corresponds to a unique index \(\mu_i\) such that for all \(\mathbf{s}\) in the projection of $Q_i$ onto the $\mathbf{s}$-space, $\mu_i$ attains the maximum exponent $d_{\mathbf{s},\max}$. The boundaries between these cones are defined by equations $d_{\mathbf{s},\mu_i} = d_{\mathbf{s},\mu_{i+1}}$ for $i = 1,\dots,\ell-1$.

Define the polynomial
\[
\Gamma(\mathbf{s}) = \prod_{i=1}^{\ell-1} \bigl(d_{\mathbf{s},\mu_i} - d_{\mathbf{s},\mu_{i+1}}\bigr) \in \mathbb{Z}[s_1,\dots,s_n].
\]
Then $\Gamma(\mathbf{s}) \neq 0$ if and only if all adjacent maximum exponents are pairwise distinct, which is equivalent to the maximum exponent $d_{\mathbf{s},\max}$ being attained by a unique index. In this case, after removing $y^{k_G}$, the leading coefficient of $\fs{G}$ is precisely the single term \(a_{\mu} x_1^{e_{\mu,1}}\cdots x_n^{e_{\mu,n}}\) (where $\mu$ is the unique index attaining the maximum). Conversely, if $\Gamma(\mathbf{s}) = 0$, then at least two terms have the same maximum $y$-exponent, and the leading coefficient is a sum of at least two terms, hence not a single term.

The degree of $\Gamma$ is at most $\ell-1 \le t-1 \le T-1$.
Now apply the Schwartz-Zippel lemma. Since $\mathbb{Z}$ is an integral domain, the lemma applies to polynomials over $\mathbb{Z}$. The random vector $\mathbf{s}$ is uniformly chosen from the finite set $[1,N]^n \subset \mathbb{Z}^n$, whose size is $N^n$. Each coordinate is chosen independently, so the probability that $\Gamma(\mathbf{s}) = 0$ satisfies
$\Pr\bigl(\Gamma(\mathbf{s}) = 0\bigr) \le \frac{\deg \Gamma}{N} \le \frac{T-1}{N}.$
Taking $N = \lceil (T-1)/\varepsilon \rceil$ gives
$\Pr\bigl(\Gamma(\mathbf{s}) = 0\bigr) \le \varepsilon,$
and hence
\[
\Pr\bigl(\operatorname{lc}_y(\fs{G})\text{ is a single term}\bigr) \ge 1 - \varepsilon.
\]
\end{proof}

\begin{lemma}[Condition 2]\label{lem:half}
Let $R = G - \alpha G^* \in \K[x_1,\ldots,x_n]$, where $G^* \sqsubseteq G$ is an approximation of $G$ and $\alpha$ is a single term such that every term of $\alpha G^*$ appears in $G$. Let $T \ge \|R\|_0$.  
Let $N = \lceil 2(T-1) / \varepsilon \rceil$.  
If $\mathbf{s} \in [1,N]^n$ is chosen uniformly at random, then  
\[
\Pr\bigl[ \# \mathrm{NC}(R,\mathbf{s}) \ge \tfrac12 \cdot \|R\|_0 \bigr] \ge 1 - \varepsilon.
\]
\end{lemma}

\begin{proof}
Let $t = \|R\|_0$ and write $R = \sum_{i=1}^t c_i M_i$, where $M_i = \x^{\mathbf{e}_i}$ are distinct monomials. For $\mathbf{s} \in [1,N]^n$, the $y$-exponent of $M_i$ is $d_i(\mathbf{s}) = \mathbf{e}_i \cdot \mathbf{s}$. A term $M_i$ is colliding if $d_i(\mathbf{s}) = d_j(\mathbf{s})$ for some $j \neq i$.

For each pair $i < j$, define the nonzero linear form
\[
h_{i,j}(\mathbf{s}) = d_i(\mathbf{s}) - d_j(\mathbf{s}) = (\mathbf{e}_i - \mathbf{e}_j) \cdot \mathbf{s}.
\]
The equation $h_{i,j}(\mathbf{s}) = 0$ has at most $N^{n-1}$ solutions in $[1,N]^n$.

Let $K$ be the number of $\mathbf{s} \in [1,N]^n$ such that $\#\mathrm{C}(R,\mathbf{s}) \ge t/2$. For each such $\mathbf{s}$, at least $t/2$ terms collide, yielding at least $t/4$ distinct colliding pairs $(i,j)$. Hence
\[
\frac{t}{4} K \le \sum_{i<j} \#\{\mathbf{s} : h_{i,j}(\mathbf{s}) = 0\}
\le \binom{t}{2} N^{n-1}.
\]
Thus $K \le 2(t-1)N^{n-1}$. Since there are $N^n$ vectors total,
\[
\Pr(\#\mathrm{C}(R,\mathbf{s}) > t/2)\le \Pr(\#\mathrm{C}(R,\mathbf{s}) \ge t/2) = \frac{K}{N^n}
\le \frac{2(t-1)}{N}
\le \varepsilon,
\]
where the last inequality follows from $N = \lceil 2(T-1)/\varepsilon \rceil \ge 2(t-1)/\varepsilon$. Therefore,
\[
\Pr(\#\mathrm{NC}(R,\mathbf{s}) \ge t/2)=\Pr(\#\mathrm{C}(R,\mathbf{s}) \le t/2)
= 1 - \Pr(\#\mathrm{C}(R,\mathbf{s}) > t/2)
\ge 1 - \varepsilon.
\]
\end{proof}

\begin{lemma}[Condition 3]\label{lem:lowest}
Let $G \in \K[x_1, \dots, x_n]$ have at most $T$ terms, and let $c_1 M_1$ be its lowest term with respect to the fixed monomial order. For any $\varepsilon \in (0,1)$, if $\mathbf{s}$ is chosen uniformly from $[1, N]^n$ with
$N = \left\lceil (T-1)/\varepsilon \right\rceil,$
then the probability that $c_1 M_1$ is non-colliding in $\fs{G}$ is at least $1 - \varepsilon$.
\end{lemma}

\begin{proof}
Let the terms of $G$ be $c_1 M_1, c_2 M_2, \dots, c_t M_t$, where $c_1 M_1$ is the lowest term with respect to the fixed monomial order. Let $d_i(\mathbf{s})$ be the $y$-exponent of $M_i$ after transformation. The lowest term $c_1 M_1$ is non-colliding iff $d_1(\mathbf{s}) \neq d_i(\mathbf{s})$ for all $i = 2, \dots, t$. Consider the polynomial
$\Delta_{\text{low}}(\mathbf{s}) = \prod_{i=2}^t \bigl(d_i(\mathbf{s}) - d_1(\mathbf{s})\bigr).$
This polynomial is nonzero exactly when $c_1 M_1$ is non-colliding. Each factor is a linear form, so $\deg \Delta_{\text{low}} \le t-1 \le T-1$. By the Schwartz--Zippel lemma,
$\Pr\bigl(\Delta_{\text{low}}(\mathbf{s}) = 0\bigr) \le \frac{T-1}{N}.$
Taking $N = \lceil (T-1)/\varepsilon \rceil$ yields the result.
\end{proof}

\begin{theorem}\label{thm:s-prob}
Let $G \in \K[x_1, \dots, x_n]$ be nonzero with at most $T$ terms, and let $G^* \sqsubseteq G$ be a current approximation with $\#(G - \alpha_1 G^*) \le T$. For any $\varepsilon \in (0,1)$, if $\mathbf{s}$ is chosen uniformly from $[1, N]^n$ with $N = \lceil 6(T-1)/\varepsilon \rceil$, then Conditions 1, 2, and 3 hold simultaneously with probability at least $1 - \varepsilon$.
\end{theorem}

\begin{proof}
By Lemmas~\ref{lem:leading}, \ref{lem:half}, and \ref{lem:lowest} the failure probabilities for Conditions 1, 2, and 3 are respectively bounded by $\varepsilon/3$ when $N \ge \lceil 3(T-1)/\varepsilon\rceil$, $N \ge \lceil 3(T-1)/\varepsilon\rceil$, and $N \ge \lceil 6(T-1)/\varepsilon\rceil$. With $N = \lceil 6(T-1)/\varepsilon \rceil$, all three bounds are satisfied. The union bound gives total failure probability at most $\varepsilon$. Hence all three conditions hold with probability at least $1 - \varepsilon$.
\end{proof}

\subsection{Conditions on the Evaluation Point $\b$}

The preceding discussion establishes the requirements on the random vector $\mathbf{s}$. However, the evaluation point $\b$ must also satisfy certain conditions for the algorithm to succeed.

As $G=\gcd(A,B)$, by Lemma \ref{lem:gcd-separation}, 
\[
\fs{G}(\x, y) = \gcd( \fs{A}(\x, y), \fs{B}(\x, y) ).
\]

First, recall that in each iteration, we compute $\fs{G}(\b, y)$ as
\[
\operatorname{monic}(\fs{G}(\b, y) )= \operatorname{monic}\bigl( \gcd( \fs{A}(\b, y), \fs{B}(\b, y) ) \bigr).
\]
For this equality to hold, $\b$ must be chosen such that the univariate GCD of $\fs{A}(\b, y)$ and $\fs{B}(\b, y)$ exactly corresponds to $\fs{G}(\b, y)$, with no degree loss or spurious common factors. 

To see when this holds, write $A = G \cdot A_1$ and $B = G \cdot B_1$ with $\gcd(A_1, B_1) = 1$. By Lemma~\ref{lem:gcd-separation}, the separation transformation preserves coprimality, so
\[
\gcd\bigl( \fs{A_1}(\x, y), \fs{B_1}(\x, y) \bigr) = 1,
\]
equivalently,
\[
\gcd\bigl( \fs{A}(\x, y)/\fs{G}(\x, y),\; \fs{B}(\x, y)/\fs{G}(\x, y) \bigr) = 1.
\]
Therefore, for a given evaluation point $\b$, Lemma~\ref{lm-resultant} implies that the univariate GCD $\gcd(\fs{A}(\b, y), \fs{B}(\b, y))$ coincides with $\fs{G}(\b, y)$ up to a constant factor, provided that:
\begin{itemize}
    \item the resultant
    \[
    \operatorname{Res}_y\bigl( \fs{A}(\x, y)/\fs{G}(\x, y),\; \fs{B}(\x, y)/\fs{G}(\x, y) \bigr) \in \K[\x]
    \]
    does not vanish at $\b$, which ensures that the two factors remain coprime after evaluation; and
    \item the leading coefficients of $\fs{A}$ and $\fs{B}$ with respect to $y$ do not vanish at $\b$, which ensures that no degree loss occurs when passing from the multivariate polynomials to their evaluations.
\end{itemize}
Both the resultant and the leading coefficients are nonzero polynomials in $\K[\x]$, so the Schwartz--Zippel lemma ensures that a random $\b$ from a sufficiently large set $S \subseteq \K$ satisfies all these conditions with high probability.

Second, recall that the Hensel lifting step in our algorithm requires the univariate factors $\fs{G}(\b, y)$ and $\fs{H}(\b, y)$ to be coprime and to have the same degrees in $y$ as their multivariate counterparts $\fs{G}(\x, y)$ and $\fs{H}(\x, y)$. These two conditions are necessary for the Hensel lifting algorithm to proceed correctly.

Since $F = G \cdot H$ with $\gcd(G, H) = 1$ (recall that $F = A + c \cdot B = G \cdot (A_1 + c \cdot B_1)$, so $H = A_1 + c \cdot B_1$; the constant $c$ is chosen so that $\gcd(G, H) = 1$), Lemma~\ref{lem:gcd-separation} implies that
\[
\gcd\bigl( \fs{G}(\x, y), \fs{H}(\x, y) \bigr) = 1.
\]

To ensure these two conditions, we require that $\b$ satisfies the following:
\begin{itemize}
    \item the resultant
    \[
    \operatorname{Res}_y\bigl( \fs{G}(\x, y),\; \fs{H}(\x, y) \bigr) \in \K[\x]
    \]
    does not vanish at $\b$, which guarantees that $\fs{G}(\b, y)$ and $\fs{H}(\b, y)$ are coprime; and
    \item the leading coefficients of $\fs{G}$ and $\fs{H}$ with respect to $y$ do not vanish at $\b$, i.e.,
    \[
    \operatorname{lc}_y(\fs{G}(\b, y)) \neq 0, \qquad \operatorname{lc}_y(\fs{H}(\b, y)) \neq 0,
    \]
    which ensures that $\deg_y \fs{G}(\b, y) = \deg_y \fs{G}(\x, y)$ and $\deg_y \fs{H}(\b, y) = \deg_y \fs{H}(\x, y)$, so no degree loss occurs upon evaluation.
\end{itemize}
Since both the resultant and the leading coefficients are nonzero polynomials in $\K[\x]$, the Schwartz--Zippel lemma ensures that a random $\b$ from a sufficiently large set satisfies all these conditions with high probability.

Third, as noted above, our recursive strategy does not require $\fs{G}$ to be fully $y$-separated; collisions among terms are permitted. However, when collisions occur, the derivative-based interpolation algorithm (Algorithm~\ref{alg:derivative_recovery}) would incorrectly treat a colliding group of terms as a single ``pseudo-term'' if applied naively, introducing spurious terms into the recovered polynomial. To prevent this, we must distinguish, for each $y$-degree, whether the corresponding coefficient is a single term (non-colliding) or a sum of multiple terms (colliding). 

We achieve this using evaluations at $\b, \b^2, \b^3$. For each $y$-degree $d_i$, let $A_i, B_i, C_i \in \K$ be the coefficients of $y^{d_i}$ in $\fs{G}(\b, y)$, $\fs{G}(\b^2, y)$, and $\fs{G}(\b^3, y)$, respectively. We compute the determinant
\[
\delta_i(\b) = A_i C_i - B_i^2.
\]
A key algebraic fact, established in Lemma~\ref{lem:collision-detection}, is that $\delta_i(\b) = 0$ if the coefficient of $y^{d_i}$ in $\fs{G}$ is a single term; if it is a sum of at least two terms, then $\delta_i(\b) \neq 0$ with high probability over the choice of $\b$. Thus, by testing $\delta_i(\b) = 0$, we can reliably identify the non-colliding terms and exclude the colliding ones from interpolation.

To summarize, the evaluation point $\b$ must satisfy the following three conditions:

\begin{enumerate}
    \item \textbf{Correct univariate GCD recovery:} $\b$ must satisfy the resultant condition for $\fs{A}/\fs{G}$ and $\fs{B}/\fs{G}$, ensuring that the univariate GCD $\gcd(\fs{A}(\b,y), \fs{B}(\b,y))$ correctly recovers $\fs{G}(\b,y)$ up to a constant factor.
    \item \textbf{Coprimality for Hensel lifting:} $\b$ must satisfy the resultant condition for $\fs{G}$ and $\fs{H}$, ensuring that $\fs{G}(\b,y)$ and $\fs{H}(\b,y)$ are coprime, which is a prerequisite for applying Hensel lifting.
    \item \textbf{Correct collision detection:} $\b$ must be chosen so that the determinant test using evaluations at $\b$, $\b^2$, and $\b^3$ correctly distinguishes between single-term and multi-term coefficients, ensuring that only genuine non-colliding terms are recovered by the interpolation algorithm.
\end{enumerate}

\subsection*{Probability Analysis}\label{sec:good-point}

We now analyze the three requirements on $\b$ and prove that a random choice from a sufficiently large set satisfies all of them with high probability.

Throughout this section, we assume that $\mathbf{s}$ is already fixed and satisfies Conditions 1, 2, and 3 from Section~\ref{sec:params}, with $\|\mathbf{s}\|_\infty \le N = O(T/\varepsilon)$. Let $D$ be a bound on the total degree of $A$ and $B$.

\subsubsection*{Condition 1: Correct Recovery of $\fs{G}(\b, y)$ from the Univariate GCD}

Recall that in each iteration we compute
\[
\operatorname{monic}\bigl( \gcd( \fs{A}(\b, y), \fs{B}(\b, y) ) \bigr) = \operatorname{monic}(\fs{G}(\b, y)).
\]
For this equality to hold, $\b$ must be chosen such that the univariate GCD of $\fs{A}(\b, y)$ and $\fs{B}(\b, y)$ exactly corresponds to $\fs{G}(\b, y)$, with no degree loss or spurious common factors.

The following theorem gives a sufficient condition and its probability of success.

\begin{theorem}\label{thm:b-condition1}
Let $A, B \in \K[x_1, \dots, x_n]$ with $G = \gcd(A, B)$, and let $\mathbf{s} \in \mathbb{N}^n$ be fixed with $\|\mathbf{s}\|_\infty \le N$. Suppose $D \ge \max\{\deg A, \deg B\}$. Define
\[
\Psi_1(\mathbf{x}) = \operatorname{lc}_y(\fs{A}) \cdot \operatorname{lc}_y(\fs{B}) \cdot \operatorname{Res}_y\bigl( \fs{A}/\fs{G},\; \fs{B}/\fs{G} \bigr) \in \K[x_1, \dots, x_n].
\]
Then $\Psi_1$ is a nonzero polynomial with
$\deg \Psi_1 \le 2D^2 N + 2D.$
If $\b$ is chosen uniformly from $S^n$ with $S \subseteq \K$ and
$|S| \ge \frac{2D^2 N + 2D}{\varepsilon},$
then
$\Pr\bigl( \Psi_1(\b) \neq 0 \bigr) \ge 1 - \varepsilon.$
Moreover, whenever $\Psi_1(\b) \neq 0$, we have
\[
\operatorname{monic}\bigl( \gcd( \fs{A}(\b, y), \fs{B}(\b, y) ) \bigr) = \operatorname{monic}(\fs{G}(\b, y)).
\]
\end{theorem}

\begin{proof}
By Lemma~\ref{lem:gcd-separation}, the polynomials $\fs{A}/\fs{G}$ and $\fs{B}/\fs{G}$ are coprime in $\K[x_1, \dots, x_n, y]$; hence their resultant is nonzero. The leading coefficients $\operatorname{lc}_y(\fs{A})$ and $\operatorname{lc}_y(\fs{B})$ are also nonzero polynomials. Thus $\Psi_1 \not\equiv 0$.

Since $\deg A, \deg B \le D$ and $\|\mathbf{s}\|_\infty \le N$, the $y$-degrees of the transformed polynomials are bounded by
\[
\deg_y(\fs{A}), \deg_y(\fs{B}) \le D N.
\]
Indeed, for any term $c x_1^{e_1}\cdots x_n^{e_n}$ with $\sum_i e_i \le D$, its $y$-exponent after separation is $\sum_i s_i e_i \le N \sum_i e_i \le N D$.

Let $p = \fs{A}/\fs{G}$ and $q = \fs{B}/\fs{G}$. Then $\deg_y p, \deg_y q \le D N$, and $\deg_{\x} p, \deg_{\x} q \le D$. By the standard resultant degree bound,
\[
\deg \operatorname{Res}_y(p, q)
\le (\deg_y p)(\deg_{\x} q) + (\deg_y q)(\deg_{\x} p)
\le 2D^2 N.
\]
The leading coefficients $\operatorname{lc}_y(\fs{A})$ and $\operatorname{lc}_y(\fs{B})$ each have total degree at most $D$. Therefore,
\[
\deg \Psi_1 \le 2D^2 N + 2D.
\]

By the Schwartz--Zippel lemma (Theorem~\ref{thm:sz}),
\[
\Pr(\Psi_1(\b) = 0) \le \frac{\deg \Psi_1}{|S|}
\le \frac{2D^2 N + 2D}{|S|}
\le \varepsilon.
\]
Thus $\Pr(\Psi_1(\b) \neq 0) \ge 1 - \varepsilon$.

When $\Psi_1(\b) \neq 0$, the evaluated polynomials $p(\b, y)$ and $q(\b, y)$ are coprime, and the leading coefficients $\operatorname{lc}_y(\fs{A})(\b)$ and $\operatorname{lc}_y(\fs{B})(\b)$ are nonzero. Consequently, $\gcd(\fs{A}(\b, y), \fs{B}(\b, y)) =c\cdot \fs{G}(\b, y)$ up to a nonzero constant. Normalizing by the leading coefficient yields the desired equality.
\end{proof}

\subsubsection*{Condition 2: Coprimality for Hensel Lifting}

Recall that $F = A + cB = G \cdot H$ with $\gcd(G, H) = 1$ (by Theorem~\ref{thm:good-c}). The Hensel lifting step requires that the univariate factors $\fs{G}(\b, y)$ and $\fs{H}(\b, y)$ be coprime. By Lemma~\ref{lem:gcd-separation}, we have
\[
\gcd\bigl( \fs{G}(\x, y), \fs{H}(\x, y) \bigr) = 1.
\]
However, after evaluating at $\b$, it is possible that $\fs{G}(\b, y)$ and $\fs{H}(\b, y)$ acquire a common factor.

The following theorem gives a sufficient condition and its probability of success.

\begin{theorem}\label{thm:b-condition2}
Let $F = G \cdot H$ with $\gcd(G, H) = 1$, and let $\mathbf{s} \in \mathbb{N}^n$ be fixed with $\|\mathbf{s}\|_\infty \le N$. Suppose $D \ge \max\{\deg G, \deg H\}$. Define
\[
\Psi_2(\mathbf{x}) = \operatorname{lc}_y(\fs{G}) \cdot \operatorname{lc}_y(\fs{H}) \cdot \operatorname{Res}_y(\fs{G}, \fs{H}) \in \K[x_1, \dots, x_n].
\]
Then $\Psi_2$ is a nonzero polynomial with
$\deg \Psi_2 \le 2D^2 N + 2D.$
If $\b$ is chosen uniformly from $S^n$ with $S \subseteq \K$ and
$|S| \ge \frac{2D^2 N + 2D}{\varepsilon},$
then
\[
\Pr\bigl( \Psi_2(\b) \neq 0 \bigr) \ge 1 - \varepsilon.
\]
Moreover, whenever $\Psi_2(\b) \neq 0$, the univariate polynomials $\fs{G}(\b, y)$ and $\fs{H}(\b, y)$ are coprime, and their degrees in $y$ are preserved.
\end{theorem}

\begin{proof}
Since $\gcd(G, H) = 1$, Lemma~\ref{lem:gcd-separation} implies that $\fs{G}$ and $\fs{H}$ are coprime in $\K[x_1, \dots, x_n, y]$; hence their resultant is nonzero. The leading coefficients $\operatorname{lc}_y(\fs{G})$ and $\operatorname{lc}_y(\fs{H})$ are also nonzero polynomials. Thus $\Psi_2 \not\equiv 0$.

Since $\deg G, \deg H \le D$ and $\|\mathbf{s}\|_\infty \le N$, the $y$-degrees of the transformed polynomials are bounded by
\[
\deg_y(\fs{G}), \deg_y(\fs{H}) \le D N.
\]

By the standard resultant degree bound,
\[
\deg \operatorname{Res}_y(\fs{G}, \fs{H})
\le (\deg_y \fs{G})(\deg_{\x} \fs{H}) + (\deg_y \fs{H})(\deg_{\x} \fs{G})
\le 2D^2 N.
\]
The leading coefficients $\operatorname{lc}_y(\fs{G})$ and $\operatorname{lc}_y(\fs{H})$ each have total degree at most $D$. Therefore,
\[
\deg \Psi_2 \le 2D^2 N + 2D.
\]

By the Schwartz--Zippel lemma (Theorem~\ref{thm:sz}),
\[
\Pr(\Psi_2(\b) = 0) \le \frac{\deg \Psi_2}{|S|}
\le \frac{2D^2 N + 2D}{|S|}
\le \varepsilon.
\]
Thus $\Pr(\Psi_2(\b) \neq 0) \ge 1 - \varepsilon$.

When $\Psi_2(\b) \neq 0$, the resultant $\operatorname{Res}_y(\fs{G}(\b, y), \fs{H}(\b, y))$ is nonzero and the leading coefficients do not vanish. Hence $\fs{G}(\b, y)$ and $\fs{H}(\b, y)$ are coprime and their degrees in $y$ are preserved. This is exactly the condition required for Hensel lifting.
\end{proof}

\subsubsection*{Condition 3: Collision Detection via Evaluations at $\b, \b^2, \b^3$}

As noted earlier, collisions among terms are permitted in our recursive strategy. However, when collisions occur, we must distinguish, for each $y$-degree, whether the corresponding coefficient is a single term or a sum of multiple terms; otherwise, the derivative-based interpolation algorithm would incorrectly treat a colliding group as a single ``pseudo-term'' and introduce spurious terms.

The following lemma serves as the theoretical foundation for our collision detection mechanism. It establishes that the vanishing of a certain determinant is equivalent to the absence of collisions among the terms of a Laurent polynomial. This lemma underpins all subsequent developments in this section.

Denote by $\x^i$ the tuple $(x_1^i,\dots,x_n^i)$, and by $\b^i$ the tuple $(b_1^i,\dots,b_n^i)$.

\begin{lemma}\label{lem:collision-detection}
Let $U = \sum_{i=1}^t c_i M_i$, where $M_i = x_1^{e_{i1}} \cdots x_n^{e_{in}}$ are distinct Laurent monomials and $c_i \neq 0$. Define
\[
\mathcal{D} = \det \begin{pmatrix}
U(x_1, \dots, x_n) & U(x_1^2, \dots, x_n^2) \\
U(x_1^2, \dots, x_n^2) & U(x_1^3, \dots, x_n^3)
\end{pmatrix}.
\]
Then $\mathcal{D} \equiv 0$ if $t = 1$, and $\mathcal{D} \not\equiv 0$ if $t \ge 2$.
\end{lemma}

\begin{proof}
If $t = 1$, then $U = c_1 M_1$ and $\mathcal{D} = 0$ directly.

If $t \ge 2$, expanding the determinant gives
\[
\mathcal{D} = \sum_{1 \le i < j \le t} c_i c_j \left( M_i M_j^3 + M_j M_i^3 - 2 M_i^2 M_j^2 \right).
\]
Order the monomials lexicographically so that $M_t \succ M_{t-1} \succ \cdots \succ M_1$. For each pair $(i,j)$, the leading term comes from $M_i M_j^3$ when $i<j$. For the pair $(t-1,t)$, this gives $M_{t-1} M_t^3$, which is strictly larger than any monomial arising from other pairs. Hence the unique leading term of $\mathcal{D}$ is $c_{t-1} c_t \, M_{t-1} M_t^3$ with nonzero coefficient, so $\mathcal{D} \not\equiv 0$.
\end{proof}

Guided by this lemma, we now apply the determinant test to the coefficients of $\fs{G}$ with respect to $y$. Write
\[
\fs{G}(\x, y) = \sum_i C_i(\x) y^{d_i}.
\]
For each $y$-degree $d_i$, define
\[
\delta_i(\x) = C_i(\x) C_i(\x^3) - C_i(\x^2)^2.
\]
By Lemma~\ref{lem:collision-detection}, $\delta_i \not\equiv 0$ if and only if $C_i$ is a sum of at least two distinct Laurent monomials (i.e., colliding), and $\delta_i \equiv 0$ if and only if $C_i$ is a single term (non-colliding).

In practice, however, we do not have direct access to the coefficients $C_i(\x)$ of $\fs{G}$. Fix $\mathbf{s}$ satisfying Condition~1, so that $\operatorname{lc}_y(\fs{G}) = \beta(\x)$ is a single term. After normalization, we obtain
\[
\mathcal{G}_0(\x, y) = \frac{\fs{G}(\x, y)}{\beta(\x)}
= \sum_i \widetilde{C}_i(\x) y^{d_i},
\]
where $\widetilde{C}_i(\x) = C_i(\x) / \beta(\x)$ are Laurent polynomials. For these normalized coefficients, define
\[
\mathcal{D}_i(\x) = \widetilde{C}_i(\x) \widetilde{C}_i(\x^3) - \widetilde{C}_i(\x^2)^2.
\]
Since $\beta$ is a single term, we have $\beta(\x) \beta(\x^3) = \beta(\x^2)^2$, and hence
\[
\mathcal{D}_i(\x) = \frac{\delta_i(\x)}{\beta(\x) \beta(\x^3)}.
\]
The denominator is a nonzero single term, so for any $\b \in (\K^*)^n$,
\[
\mathcal{D}_i(\b) = 0 \quad \Longleftrightarrow \quad \delta_i(\b) = 0.
\]
Thus, the normalized coefficients serve as a faithful proxy for collision detection.

To ensure that all non-colliding terms are correctly identified, we define
\[
\Psi_3(\x):=
\prod_{\substack{i : C_i(\x) \\ \text{is not a single term}}} \delta_i(\x),
\]
where the product is taken over all non-single-term coefficients; if no such coefficient exists, we set $$\Psi_3(\x) := 1.$$

This is a nonzero polynomial, since each $\delta_i$ is a polynomial (not a Laurent polynomial). The following theorem quantifies the probability that evaluation at a random point $\b$ successfully distinguishes all collisions.

\begin{theorem}\label{thm:b-condition3}
Let $G \in \K[x_1, \dots, x_n]$ and let $\mathbf{s} \in \mathbb{N}^n$ be fixed, satisfying Condition~1. Suppose $D \ge \deg G$ and let $t = \|G\|_0$ and $\varepsilon\in(0,1)$. Define $\Psi_3$ as above. Then $\Psi_3$ is a nonzero polynomial with $\deg \Psi_3 \le 4 D (t-1)$. If $\b$ is chosen uniformly from $S^n$ with $S \subseteq \K^*$ and $|S| \ge 4 D (t-1) / \varepsilon$, then
\[
\Pr(\Psi_3(\b) \neq 0) \ge 1 - \varepsilon.
\]
Moreover, whenever $\Psi_3(\b) \neq 0$, for every $y$-degree $d_i$,
\[
\mathcal{D}_i(\b) = 0 \quad \Longleftrightarrow \quad C_i(\x) \text{ is a single term}.
\]
That is, the collision detection mechanism correctly identifies all non-colliding terms.
\end{theorem}

\begin{proof}
By Lemma~\ref{lem:collision-detection}, for each non-single-term coefficient $C_i$, $\delta_i \not\equiv 0$. Hence $\Psi_3 \not\equiv 0$. Since $\deg C_i \le D$, we have $\deg \delta_i \le 4D$. 
Since a non-single-term coefficient must arise from a collision of two or more of the $t$ monomials, the number of non-single-term coefficients is at most $t-1$. Therefore $\Psi_3$ has at most $t-1$ factors, so $\deg \Psi_3 \le 4D(t-1)$.
The Schwartz--Zippel lemma then gives
\[
\Pr(\Psi_3(\b) = 0) \le \frac{\deg \Psi_3}{|S|} \le \varepsilon.
\]
When $\Psi_3(\b) \neq 0$, none of the $\delta_i(\b)$ for non-single-term coefficients vanish. Since $\mathcal{D}_i(\b)$ differs from $\delta_i(\b)$ by a nonzero single term factor (by $\beta(\b)\beta(\b^3)$), we have $\mathcal{D}_i(\b) = 0$ if and only if $\delta_i(\b) = 0$. Thus $\mathcal{D}_i(\b) = 0$ holds exactly for the single-term coefficients. Therefore, the collision detection mechanism correctly identifies all non-colliding terms.
\end{proof}

\subsection{Collision Detection Mechanism (Algorithm~\ref{alg:collideTest})}
\label{sec:collision}

We now present the concrete algorithm that implements the collision detection mechanism described above. Given the evaluations of a polynomial at $\b, \b^2, \b^3$, the algorithm determines, for each $y$-degree, whether the corresponding coefficient is a single term (non-colliding) or a sum of multiple terms (colliding).

The algorithm is invoked in two distinct contexts within our main procedure. First, it is applied to the normalized GCD factor $\mathcal{G}_0(\b, y)$ to identify its non-colliding terms, which are then recovered via Algorithm~\ref{alg:derivative_recovery}. Second, it is applied to the difference polynomial
\[
\Delta(\x, y) = \gamma(\x) \cdot \mathcal{G}_0(\x, y) \cdot y^{\delta} - \fs{G^*}(\x, y)
\]
obtained after aligning the current approximation, enabling the recovery of additional terms in subsequent iterations. In both cases, we let $R$ denote the polynomial under consideration, i.e., $R = \mathcal{G}_0$ or $R = \Delta$. The subroutine relies on the theoretical guarantee established in Theorem~\ref{thm:b-condition3}: whenever $\Psi_3(\b) \neq 0$, the determinant test $\mathcal{D}_i = 0$ correctly characterizes whether the $i$-th coefficient is a single term. We assume throughout that $\b$ satisfies $\Psi_3(\b) \neq 0$, so that the algorithm is deterministic and its correctness is unconditional conditioned on this assumption. The probability analysis of this condition being satisfied is deferred to the overall success probability analysis of the main algorithm, where it is governed by the random choice of $\b$.

\begin{algorithm}
\caption{Collision Detection Subalgorithm}
\label{alg:collideTest}
\begin{algorithmic}[1]
\Require
    \begin{itemize}
        \item $\b$ satisfies $\Psi_3(\b) \neq 0$ as defined in Theorem~\ref{thm:b-condition3};
        \item Polynomial $R_1(y) = \sum_{i=1}^{t} A_i y^{d_i}$, where the $d_i$ are distinct; this is the evaluation of $R$ at $\b$;
        \item Polynomial $R_2(y) = \sum_{i=1}^{t} B_i y^{d_i}$; this is the evaluation at $\b^2$;
        \item Polynomial $R_3(y) = \sum_{i=1}^{t} C_i y^{d_i}$; this is the evaluation at $\b^3$.
    \end{itemize}
\Ensure
    Non-colliding term index set $I_{\text{non}}$ and colliding term index set $I_{\text{coll}}$.

\For{$i = 1$ to $t$}
    \State Compute $\mathcal{D}_i \gets A_i C_i - B_i^2$.
    \If{$\mathcal{D}_i = 0$}
        \State $I_{\text{non}} \gets I_{\text{non}} \cup \{d_i\}$.
    \Else
        \State $I_{\text{coll}} \gets I_{\text{coll}} \cup \{d_i\}$.
    \EndIf
\EndFor

\State \Return $I_{\text{non}}, I_{\text{coll}}$
\end{algorithmic}
\end{algorithm}

\subsection{One-Step Improvement (Algorithm~\ref{alg:one-step})}
\label{sec:one-step}

Based on the above lemmas, we now present Algorithm~\ref{alg:one-step}. Given a current approximation $G^*$ of $G$ -- either $G^* = 0$ or a partial sum of terms of $G$ that includes the lowest term of $G$ (up to the single term factor $\alpha_1$) -- the algorithm computes an improved approximation $G^{**}$ satisfying
\[
\bigl\| G - \alpha_2 G^{**} \bigr\|_0 \le \frac{1}{2} \bigl\| G - \alpha_1 G^* \bigr\|_0,
\]
for some single terms $\alpha_1, \alpha_2$.
The algorithm chooses $\mathbf{s}$ and $\b$ randomly so that, with high probability, all the required conditions hold simultaneously: the three conditions on $\mathbf{s}$ (leading coefficient separation, lowest-term separation, and half-remaining separation) as established in Theorem~\ref{thm:s-prob}, as well as the conditions on $\b$ (correct recovery of the univariate GCD, coprimality for Hensel lifting, and correct collision detection) as established in Section~\ref{sec:good-point}. Thus, the algorithm succeeds with high probability.

Each iteration reduces the number of missing terms by at least half, so after $O(\log t)$ iterations the complete $G$ is recovered.

\begin{algorithm}
\caption{One-Step Improvement of Approximation}
\label{alg:one-step}
\begin{algorithmic}[1]
\Require
    \begin{itemize}
        \item Polynomials $A, B \in \K[\x]$ with $G = \gcd(A,B)$;
        \item A current approximation $G^*$ of $G$, with $G^* \sqsubseteq G$: either $G^* = 0$ or a partial sum of terms of $G$ that includes the lowest term of $G$ (up to a single term factor);
        \item An upper bound $T$ on the number of terms of $G$;
        \item The field $\K$ has characteristic $0$ or characteristic greater than $\max_{i=1}^n\min\{\deg_{x_i} A,\deg_{x_i} B\}$.
    \end{itemize}
\Ensure
    With probability $\ge 1-\mu$, output an improved approximation $G^{**}$ satisfying
    \[
    \|G - \alpha_2 G^{**}\|_0 \le \frac{1}{2} \|G - \alpha_1 G^*\|_0
    \]
    for some single terms $\alpha_1, \alpha_2$, or output ``Failure''.

\State $\varepsilon \gets \mu/3$.
\State $N \gets \lceil 6(T-1)/\varepsilon \rceil$.
\State Choose $S \subseteq \K^*$ with $|S| \ge (12D^2 N + 12D)/\varepsilon$. If $\K$ is small, extend it to a field $\K'$ containing a sufficiently large $S$.
\State Choose random $\mathbf{s} \in [1,N]^n$, $\b \in S^n$, and $c \in S$.

\State Compute $\fs{A}(\x,y)$, $\fs{B}(\x,y)$, and $\fs{G^*}(\x,y)$.
\State Form $F(\x,y) \gets \fs{A}(\x,y) + c \cdot \fs{B}(\x,y)$.
\State Compute evaluations at $\b$: $\fs{A}(\b,y)$, $\fs{B}(\b,y)$, $\fs{G^*}(\b,y)$, $F(\b,y)$, and their partial derivatives at $\b$.
\State Compute evaluations at $\b^2$ and $\b^3$: $\fs{A}(\b^2,y)$, $\fs{B}(\b^2,y)$, $\fs{G^*}(\b^2,y)$, $\fs{A}(\b^3,y)$, $\fs{B}(\b^3,y)$, $\fs{G^*}(\b^3,y)$.

\State \textbf{Compute univariate GCDs:}
\[
\mathcal{G}_0(\b,y) \gets \operatorname{monic}\bigl( \gcd( \fs{A}(\b,y), \fs{B}(\b,y) ) \bigr),
\]
\[
\mathcal{G}_0(\b^2,y) \gets \operatorname{monic}\bigl( \gcd( \fs{A}(\b^2,y), \fs{B}(\b^2,y) ) \bigr),
\]
\[
\mathcal{G}_0(\b^3,y) \gets \operatorname{monic}\bigl( \gcd( \fs{A}(\b^3,y), \fs{B}(\b^3,y) ) \bigr).
\]

\textbf{Prepare Hensel lifting initial factorization:}
\State $\mathcal{H}_0(\b,y) \gets F(\b,y) / \mathcal{G}_0(\b,y)$.
\If{$\gcd(\mathcal{G}_0(\b,y), \mathcal{H}_0(\b,y)) \neq 1$} \State \Return ``Failure'' \EndIf

\State Compute $u(y), v(y)$ such that $u(y) \mathcal{G}_0(\b,y) + v(y) \mathcal{H}_0(\b,y) = 1$.

\For{$k = 1$ to $n$}
    \State Compute $\dfrac{\partial \mathcal{G}_0}{\partial x_k}(\b,y) \gets \mathbf{rem}\bigl( v(y) \cdot \dfrac{\partial \fs{F}}{\partial x_k}(\b,y),\; \mathcal{G}_0(\b,y) \bigr)$ via Algorithm~\ref{alg:factorparde}.
\EndFor

 \textbf{Collision detection for $\mathcal{G}_0$:}
\State Call Algorithm~\ref{alg:collideTest} with inputs $\mathcal{G}_0(\b,y)$, $\mathcal{G}_0(\b^2,y)$, $\mathcal{G}_0(\b^3,y)$ to obtain the non-colliding index set $I_{\text{non}}$.

\State Construct $\mathcal{G}_0^{\text{non}}(\b,y) = \sum_{d_i \in I_{\text{non}}} C_i y^{d_i}$, where $C_i$ is the coefficient of $y^{d_i}$ in $\mathcal{G}_0(\b,y)$.
\For{$k = 1$ to $n$}
    \State Extract the corresponding non-colliding partial derivative terms $\dfrac{\partial \mathcal{G}_0^{\text{non}}}{\partial x_k}(\b,y)$ from $\dfrac{\partial \mathcal{G}_0}{\partial x_k}(\b,y)$.
\EndFor

\algstore{one-step}
\end{algorithmic}
\end{algorithm}


\begin{algorithm}[p]
\begin{algorithmic}[1]
\algrestore{one-step}

\State Call Algorithm~\ref{alg:derivative_recovery} with inputs $\mathcal{G}_0^{\text{non}}(\b,y)$ and $\dfrac{\partial \mathcal{G}_0^{\text{non}}}{\partial x_k}(\b,y)$ ($k=1,\dots,n$) to recover all non-colliding terms of the normalized factor, yielding $\mathcal{G}_0^{\text{non}}(\x,y)$.

\If{$G^* = 0$} \State \Return $\mathcal{G}^{\text{non}}(\x,1)$ \EndIf

\State Align $G^*$ with $\mathcal{G}^{\text{non}}(\x,y)$ by matching their lowest terms, obtaining a single term $\gamma\in \K[\x]$ and an integer $\delta$ such that
\[
\gamma \cdot \mathcal{G}^{\text{non}}(\x,y) \cdot y^{\delta} \quad \text{and} \quad \fs{G^*}(\x,y)
\]
have the same lowest term.

\State Form the difference at $\b$:
\[
\Delta(\b,y) \gets \gamma(\b) \cdot \mathcal{G}_0(\b,y) \cdot y^{\delta} - \fs{G^*}(\b,y).
\]
\State Compute the partial derivatives of the difference:
\[
\frac{\partial \Delta}{\partial x_k}(\b,y)
=
\frac{\partial \gamma}{\partial x_k}(\b) \cdot \mathcal{G}_0(\b,y) \cdot y^{\delta}
+ \gamma(\b) \cdot \frac{\partial \mathcal{G}_0}{\partial x_k}(\b,y) \cdot y^{\delta}
- \frac{\partial \fs{G^*}}{\partial x_k}(\b,y),
\quad k=1,\dots,n.
\]

\State Compute the evaluations of the difference at $\b^2$ and $\b^3$:
\[
\Delta(\b^2,y) \gets \gamma(\b^2) \cdot \mathcal{G}_0(\b^2,y) \cdot y^{\delta} - \fs{G^*}(\b^2,y),
\]
\[
\Delta(\b^3,y) \gets \gamma(\b^3) \cdot \mathcal{G}_0(\b^3,y) \cdot y^{\delta} - \fs{G^*}(\b^3,y).
\]

 \textbf{Collision detection for $\Delta$:}
\State Call Algorithm~\ref{alg:collideTest} with inputs $\Delta(\b,y)$, $\Delta(\b^2,y)$, $\Delta(\b^3,y)$ to obtain the non-colliding index set $J_{\text{non}}$.

\State Construct $\Delta^{\text{non}}(\b,y) = \sum_{e_j \in J_{\text{non}}} Q_j y^{e_j}$, where $Q_j$ is the coefficient of $y^{e_j}$ in $\Delta(\b,y)$.
\For{$k = 1$ to $n$}
    \State Extract the corresponding non-colliding partial derivative terms $\dfrac{\partial \Delta^{\text{non}}}{\partial x_k}(\b,y)$ from $\dfrac{\partial \Delta}{\partial x_k}(\b,y)$.
\EndFor

\State Call Algorithm~\ref{alg:derivative_recovery} with inputs $\Delta^{\text{non}}(\b,y)$ and $\dfrac{\partial \Delta^{\text{non}}}{\partial x_k}(\b,y)$ ($k=1,\dots,n$) to recover $U(\x,y)$.

\State $G^{**} \gets G^* + U(\x,1)$.
\State \Return $G^{**}$.
\end{algorithmic}
\end{algorithm}

\subsubsection*{Algorithm Success Probability and Complexity Analysis}

\begin{theorem}[Algorithm \ref{alg:one-step} Correctness and Complexity]\label{them-one-step}
Let $A, B \in \K[\x]$ with $G = \gcd(A,B)$, and let $G^*$ be a current approximation of $G$ such that $G^* \sqsubseteq G$ -- that is, either $G^* = 0$ or $G^*$ is a partial sum of terms of $G$ that includes the lowest term of $G$ (up to a single term factor). Suppose $\operatorname{char}(\K) = 0$ or $\operatorname{char}(\K) > \max_{i=1}^n\min\{\deg_{x_i} A,\deg_{x_i} B\}$. Let $D$ be a bound on the total degrees of $A$ and $B$, and $T$ a bound on the number of terms of $G$. Then Algorithm~\ref{alg:one-step} outputs an improved approximation $G^{**}$ satisfying
$\|G - \alpha_2 G^{**}\|_0 \le \frac{1}{2} \|G - \alpha_1 G^*\|_0$
for some single terms $\alpha_1, \alpha_2$, with probability at least $1-\mu$.
If $\K$ is infinite, the expected complexity is
\[
\widetilde{O}\Bigl( n D T /\mu+ n(\|A\|_0 + \|B\|_0) \log (TD/\mu) \Bigr)
\]
field operations in $\K$.
If $\K = \mathbb{F}_q$ is a finite field, the expected bit complexity is
\[
\widetilde{O}\Bigl( n D T \log q / \mu+ n(\|A\|_0 + \|B\|_0) \log D\log(DT/\mu)\log q \Bigr).
\]

\end{theorem}

\begin{proof}
The correctness of Algorithm~\ref{alg:one-step} depends on several random events: the choice of the vector $\mathbf{s}$, the choice of the evaluation point $\b$, and the choice of the constant $c$. We analyze each in turn.

\paragraph{Probability of Choosing the Vector $\mathbf{s}$.}
Algorithm~\ref{alg:one-step} requires $\mathbf{s}$ to satisfy three conditions:
\begin{enumerate}
    \item \textbf{Leading coefficient separation:} $\operatorname{lc}_y(\fs{G})$ is a single term;
    \item \textbf{Lowest-term separation:} the lowest term of $G$ is non-colliding in $\fs{G}$;
    \item \textbf{Half-remaining separation:} at least half of the terms of $G - \alpha_1 G^*$ are non-colliding.
\end{enumerate}
By Theorem~\ref{thm:s-prob}, choosing $\mathbf{s}$ uniformly from $[1,N]^n$ with $N = \lceil 6(T-1)/\varepsilon \rceil$ ensures that all three conditions hold simultaneously with probability at least $1 - \varepsilon$.

\paragraph{Probability of Choosing the Evaluation Point $\b$.}
We require $\b$ to satisfy three conditions (see Section~\ref{sec:good-point}):

First, $\b$ must ensure that $\mathcal{G}_0(\b^i,y)$ is correctly obtained from the univariate GCD for $i=1,2,3$:
\[
\mathcal{G}_0(\b^i,y) = \operatorname{monic}\bigl( \gcd( \fs{A}(\b^i,y), \fs{B}(\b^i,y) ) \bigr),\qquad i=1,2,3.
\]
By Theorem~\ref{thm:b-condition1}, for each $i=1,2,3$, this holds whenever $\Psi_1(\b^i) \neq 0$, where $\Psi_1$ is a nonzero polynomial of degree at most $2D^2N + 2D$. Thus it suffices to require
\[
\Psi_1(\b)\cdot \Psi_1(\b^2)\cdot \Psi_1(\b^3) \neq 0.
\]
Note that $\deg \Psi_1(\x^i) = i \cdot \deg \Psi_1(\x) \le i(2D^2N + 2D)$. Hence the product has degree at most
\[
(1+2+3)(2D^2N + 2D) = 12D^2N + 12D.
\]
Choosing $|S| \ge (12D^2N + 12D)/\varepsilon$ and applying the Schwartz--Zippel lemma gives failure probability at most $\varepsilon$.

Second, $\b$ must ensure that $\fs{G}(\b,y)$ and $\fs{H}(\b,y)$ are coprime for Hensel lifting. By Theorem~\ref{thm:b-condition2}, this holds whenever $\Psi_2(\b) \neq 0$, where $\Psi_2$ is a nonzero polynomial of degree at most $2D^2N + 2D$. With the same choice of $|S|$, the failure probability is at most $\varepsilon/6$.

Third, $\b$ must ensure that the collision detection mechanism correctly distinguishes between non-colliding and colliding terms. In Algorithm~\ref{alg:one-step}, this mechanism is invoked twice: once for the normalized GCD factor $\mathcal{G}_0$ and once for the difference polynomial $\Delta$. For each invocation, by Theorem~\ref{thm:b-condition3}, correctness holds whenever the corresponding polynomial $\Psi_3$ does not vanish at $\b$.

Specifically, let $\Psi_3^{(\mathcal{G}_0)}$ and $\Psi_3^{(\Delta)}$ denote the polynomials associated with the two collision detection steps. Both are nonzero polynomials with degree at most $4D(T-1)$. 

With our choice of $N = \lceil 6(T-1)/\varepsilon \rceil$, we have $|S| \ge (12D^2N + 12D)/\varepsilon \ge 72D(T-1)/\varepsilon$. By the Schwartz--Zippel lemma, for each invocation,
\[
\Pr(\Psi_3(\b) = 0) \le \frac{\deg \Psi_3}{|S|}
\le \frac{4D(T-1)}{72D(T-1)/\varepsilon}
= \frac{\varepsilon}{18}.
\]
Taking the union bound over the two invocations gives a total failure probability of at most $\varepsilon/9$ for the collision detection steps.

Taking the union bound over these three conditions on $\b$, with $|S| \ge (12D^2N + 12D)/\varepsilon$, the total failure probability is at most $\varepsilon+\varepsilon/6+\varepsilon/9=23\varepsilon/18$.

\paragraph{Probability of Choosing the Constant $c$.}
By Theorem~\ref{thm:good-c}, choosing $c$ uniformly from a sufficiently large set $S$ ensures $\gcd(G, A_1 + cB_1) = 1$ with probability at least $1 - \deg G / |S|$. Since $\deg G \le D$ and our choice of $|S| \ge (12D^2N + 12D)/\varepsilon$ implies $|S| \ge 12D/\varepsilon$, this failure probability is at most $\varepsilon/12$.

\paragraph{Total Success Probability.}
Taking $\varepsilon = \mu/3$ and applying the union bound over all random events, the total failure probability is at most
\[
\varepsilon + 23\varepsilon/18 + \varepsilon/12 \le 3\varepsilon=\mu.
\]
Thus Algorithm~\ref{alg:one-step} succeeds with probability at least $1-\mu$.

\paragraph{Complexity Analysis.}
The complexity of Algorithm~\ref{alg:one-step} is dominated by the following steps:
\begin{enumerate}
    \item 
    \textbf{Transformation and Evaluation:} Computing $\fs{A}$, $\fs{B}$, $\fs{G^*}$, and their evaluations at $\b, \b^2, \b^3$ requires processing each term of $A$, $B$, and $G^*$. Consider a single term $c x_1^{e_1}\cdots x_n^{e_n}$. Under the separation transformation $x_i \mapsto x_i y^{s_i}$, the $y$-degree of this term becomes $s_1 e_1 + \cdots + s_n e_n$. Since each $s_i$ is bounded by $N = O(T/\mu)$, the $y$-degree is $O(ND) = O(TD/\mu)$, and computing it for all terms contributes
\[
\widetilde O\bigl( n(\|A\|_0+\|B\|_0+\|G^*\|_0)\cdot\log (TD/\mu) \bigr)
\]
bit operations.

Additionally, evaluating the transformed polynomials at $\b, \b^2, \b^3$ for a term $c \x^{\mathbf{e}}$ with $|\mathbf{e}|_1 \le D$ requires $O(n \log D)$ field operations per point. Thus the total evaluation step costs
\[
O\bigl( n(\|A\|_0 + \|B\|_0 + \|G^*\|_0) \log D \bigr)
\]
field operations.

    \item \textbf{Univariate GCDs:} Computing three univariate GCDs $\gcd(\fs{A}(\b^i,y), \fs{B}(\b^i,y))$ for $i=1,2,3$. Each GCD operates on polynomials of degree $O(DN) = O(DT/\mu)$, with complexity $\widetilde{O}(DT/\mu)$ field operations. 
    
    \item \textbf{B\'ezout Coefficients:} Computing $u(y), v(y)$ for $\mathcal{G}_0$ and $\mathcal{H}_0$ costs $\widetilde{O}(DT/\mu)$ field operations.
    
    \item \textbf{Partial Derivatives of $\fs{F}$ and $\fs{G^*}$:} For each $k=1,\dots,n$, we need to compute
    \[
    \frac{\partial \fs{F}}{\partial x_k}(\b,y), \qquad \frac{\partial \fs{G^*}}{\partial x_k}(\b,y).
    \]
    For a single term $c \x^{\mathbf{e}}$, let $a = c b_1^{e_1} \cdots b_n^{e_n}$. Then $\partial (c \x^{\mathbf{e}})/\partial x_k = a \cdot e_k / b_k$, and computing all $n$ partial derivatives costs $O(n\log D)$ field operations. Summing over all terms of $F$ and $G^*$ yields
\[
O\bigl( n(\|F\|_0 + \|G^*\|_0) \log D \bigr),
\]
which is bounded by $O\bigl( n(\|A\|_0 + \|B\|_0 + \|G^*\|_0) \log D \bigr)$ since $F = A + cB$.

    \item \textbf{Factor Partial Derivative Computation:} For each $k=1,\dots,n$, computing $\partial \mathcal{G}_0/\partial x_k$ via Algorithm~\ref{alg:factorparde} costs $\widetilde{O}(DT/\mu)$ field operations. Total: $\widetilde{O}(n DT/\mu)$.
    
    \item \textbf{Collision Detection:} Algorithm~\ref{alg:collideTest} is called twice (for $\mathcal{G}_0$ and for $\Delta$), each processing $O(T)$ coefficients with $O(1)$ work per coefficient. Total: $O(T)$ field operations.
    
    \item \textbf{Interpolation Recovery:} Algorithm~\ref{alg:derivative_recovery} is called twice (for $\mathcal{G}_0^{\text{non}}$ and for $\Delta^{\text{non}}$), each costing $O(n T \log D)$ field operations. Total: $O(nT \log D)$.
    
    \item \textbf{Merge and Align:} Merging $G^*$ with the recovered terms costs $\widetilde{O}(T)$ field operations.
\end{enumerate}

Combining the above, the total complexity of Algorithm~\ref{alg:one-step} is
\[
\widetilde{O}\Bigl( n D T / \mu + n(\|A\|_0 + \|B\|_0) \log D \Bigr)
\]
field operations, plus an additional
\[
\widetilde{O}\bigl(n(\|A\|_0 + \|B\|_0)\cdot\log (TD/\mu) \bigr)
\]
bit operations.

Note that the field-operation count above is stated without specifying the underlying field; this distinction matters because operations in a proper extension field $\K'$ are more expensive than operations in the base field $\K$. We now clarify the dependence on $\K$.
The complexity depends on whether the underlying field $\K$ is infinite or finite. When $\K$ is infinite, no field extension is needed; the field operations are performed directly in $\K$, and since each field operation costs at least one bit operation, the bit complexity is bounded by the same asymptotic bound. To keep the presentation clean, we extend the field-operation bound to
\[
\widetilde{O}\Bigl( n D T / \mu + n(\|A\|_0 + \|B\|_0) \log (TD/\mu) \Bigr),
\]
which dominates the bit operations above.

When $\K = \mathbb{F}_q$ is finite, if $q \ge (12D^2 N + 12D)/\varepsilon = O(D^2 T/\mu)$, no extension is needed. Since each operation in $\K$ costs $\log q$ bit operations, the bit complexity is
\[
\widetilde{O}\Bigl( n D T \log q / \mu + n(\|A\|_0 + \|B\|_0) \log (TD/\mu)\log q \Bigr).
\]

If $q < (12D^2 N + 12D)/\varepsilon = O(D^2 T/\mu)$, we extend $\K$ to $\mathbb{F}_{q^{\ell}}$ in Step~3 with
\[
\ell = \log_q\left(O(D^2 T/\mu)\right) = O(\log_q(DT/\mu)).
\]
Constructing such an extension costs $\operatorname{poly}(\ell)\log q$ bit operations \cite{Shoup1994}. Since $\ell = \widetilde{O}(\log_q(DT/\mu))$, we have $\operatorname{poly}(\ell)\log q = O(\log(DT/\mu)^{O(1)}\log q)$, which is dominated by the cost of the subsequent GCD computations. 

Each arithmetic operation in $\mathbb{F}_{q^{\ell}}$ requires $\widetilde{O}(\ell)$ operations in $\mathbb{F}_q$. Consequently, substituting $\ell$ into the field-operation bound yields the following bit complexity:
\[
\widetilde{O}\Bigl( n D T \log q / \mu + n(\|A\|_0 + \|B\|_0) \log D \log(DT/\mu)\log q \Bigr).
\]

\end{proof}

\subsection{GCD Algorithm over a Field with a Given Term Bound  (Algorithm~\ref{alg:recursive})
\label{sec:recursive}}

With the one-step improvement procedure (Algorithm~\ref{alg:one-step}) at hand, we now present the complete GCD algorithm over a field. The algorithm assumes that a term bound $T \ge \|G\|_0$ is given as input. Starting from the initial approximation $G^* = 0$, it repeatedly applies Algorithm~\ref{alg:one-step} to refine the approximation. Each iteration reduces the number of missing terms by at least half; hence $O(\log T)$ iterations suffice to recover the full GCD. The algorithm is summarized as Algorithm~\ref{alg:recursive}.

\begin{algorithm}
\caption{GCD Algorithm}
\label{alg:recursive}
\begin{algorithmic}[1]
\Require
    \begin{itemize}
        \item Polynomials $A, B \in \K[\x]$ with $G = \gcd(A,B)$;
        \item An upper bound $T$ on the number of terms of $G$;
        \item The field $\K$ has characteristic $0$ or characteristic greater than $\max_{i=1}^n\min\{\deg_{x_i} A,\deg_{x_i} B\}$.
    \end{itemize}
\Ensure
    With probability $\ge 2/3$, output $G$.

\State $G^* \gets 0$.
\State $k \gets \max\{\lceil \log_2 T \rceil,1\}$.
\State $\mu \gets 1/(3k)$.

\For{$i = 1$ to $k$}
    \State Call Algorithm~\ref{alg:one-step} with inputs $A, B, G^*, T, \mu$ to obtain $G^{**}$.
    \If{$G^{**} =$ ``Failure''}
        \State \Return ``Failure''
    \Else
        \State $G^* \gets G^{**}$
    \EndIf
\EndFor

\State Let $M_A = \operatorname{MoCont}(A)$ be the monomial content of $A$, as defined in Definition~\ref{def-monocont}.
\State Let $M_B = \operatorname{MoCont}(B)$ be the monomial content of $B$.
\State $M \gets \gcd(M_A, M_B)$, the common monomial content of $A$ and $B$.
\State Let $Q = \operatorname{MoCont}(G^*)$ be the monomial content of $G^*$.
\State Let $\ell = \operatorname{lc}(G^*)$ be the leading coefficient of $G^*$ (which is the scalar part of the monomial content).
\State $G \gets M \cdot G^* / (Q \cdot \ell)$.
\Comment{Normalize $G^*$ so that its leading coefficient is $1$ and its monomial content matches that of $\gcd(A,B)$}.
\State \Return $G$.
\end{algorithmic}
\end{algorithm}

\subsubsection*{Analysis of the GCD Algorithm}

We now analyze the correctness, success probability, and complexity of Algorithm~\ref{alg:recursive}.

\begin{theorem}[Correctness of Algorithm~\ref{alg:recursive}]\label{thm:recursive}
Let $A, B \in \K[\x]$ and let $G = \gcd(A,B)$. Suppose $\operatorname{char}(\K) = 0$ or $\operatorname{char}(\K) > \max_{i=1}^n\min\{\deg_{x_i} A,\deg_{x_i} B\}$. Then Algorithm~\ref{alg:recursive} outputs $G$ with probability at least $2/3$.
\end{theorem}

\begin{proof}
We prove correctness by establishing the following invariant: after the $i$-th iteration of the loop, the approximation $G^*$ satisfies
\[
G^* \sqsubseteq G \quad \text{and} \quad \|G - \alpha_i G^*\|_0 \le \frac{1}{2^i} \|G\|_0,
\]
for some single term $\alpha_i$.

The invariant is trivially true before any iterations, since $G^* = 0$ and $\|G - 0\|_0 = \|G\|_0$.

Assume the invariant holds after iteration $i-1$. By Theorem~\ref{them-one-step}, Algorithm~\ref{alg:one-step} computes an improved approximation $G^{**}$ satisfying
\[
\|G - \alpha_{i} G^{**}\|_0 \le \frac{1}{2} \|G - \alpha_{i-1} G^*\|_0 \le \frac{1}{2^i} \|G\|_0,
\]
with probability at least $1 - 1/(3k)$. Thus the invariant holds after iteration $i$.

After $k = \max\{\lceil \log_2 T \rceil,1\}$ iterations, we have
\[
\|G - \alpha_k G^*\|_0 \le \frac{1}{2^k} \|G\|_0 \le \frac{1}{2T} \cdot T = \frac{1}{2}.
\]
Since $\|G - \alpha_k G^*\|_0$ is an integer, it must be $0$. Hence $\alpha_k G^*$ contains all terms of $G$, i.e., $\alpha_k G^* = G$.

The final step recovers the scalar and monomial content of $G$. Let $M_A = \operatorname{MoCont}(A)$ and $M_B = \operatorname{MoCont}(B)$ be the monomial contents of $A$ and $B$, and let $M = \gcd(M_A, M_B)$ be their common monomial content. Since $G = \gcd(A,B)$, the monomial content of $G$ is exactly $M$. Let $Q = \operatorname{MoCont}(G^*)$ be the monomial content of the recovered approximation, and let $\ell = \operatorname{lc}(G^*)$ be its leading coefficient (the scalar factor). Then the true GCD is recovered as
\[
G = \frac{M \cdot G^*}{Q \cdot \ell}.
\]
This final normalization is deterministic and always correct.

The success probability follows from the union bound over the $k$ iterations. Each call to Algorithm~\ref{alg:one-step} succeeds with probability at least $1 - \mu = 1 - 1/(3k)$. By the union bound,
\[
\Pr(\text{all iterations succeed}) \ge 1 - k \cdot \frac{1}{3k} = \frac{2}{3}.
\]
If any iteration fails, the algorithm either returns ``Failure'' (if the failure is detected) or the correctness of the final output is still guaranteed by the theorem. 
Thus Algorithm~\ref{alg:recursive} outputs the correct GCD $G$ with probability at least $2/3$.
\end{proof}

\begin{theorem}[Complexity of Algorithm~\ref{alg:recursive}]\label{thm:recursive-complexity}
Let $A, B \in \K[\x]$ and let $G = \gcd(A,B)$. Suppose $\operatorname{char}(\K) = 0$ or $\operatorname{char}(\K) > \max_{i=1}^n\min\{\deg_{x_i} A,\deg_{x_i} B\}$.
If $\K$ is infinite, the expected complexity of Algorithm~\ref{alg:recursive} is
\[
\widetilde{O}\Bigl( n D T + n(\|A\|_0 + \|B\|_0) \log^2 T \log D \Bigr)
\]
field operations in $\K$.
If $\K = \mathbb{F}_q$ is a finite field, the expected bit complexity of Algorithm~\ref{alg:recursive}   is
\[
\widetilde{O}\Bigl(
n D T \log q
+ n(\|A\|_0 + \|B\|_0) \log^2 T \log^2 D \log q
\Bigr).
\]

\end{theorem}

\begin{proof}
Algorithm~\ref{alg:recursive} calls Algorithm~\ref{alg:one-step} for $k = O(\log T)$ iterations. By Theorem~\ref{them-one-step}, if $\K$ is infinite, each call has complexity
\[
\widetilde{O}\Bigl( n D T / \mu + n(\|A\|_0 + \|B\|_0) \log (TD/\mu) \Bigr)
\]
field operations in $\K$, where $\mu = 1/(3k)$ is the failure probability per iteration. Substituting $\mu = 1/(3k)$, the per-iteration field complexity is
\[
\widetilde{O}\Bigl( n D T k + n(\|A\|_0 + \|B\|_0) \log (TDk) \Bigr).
\]
Summing over $k = O(\log T)$ iterations, the total field complexity is
\[
\widetilde{O}\Bigl( n D T \log^2 T + n(\|A\|_0 + \|B\|_0) \log D \log^2 T \Bigr),
\]
which simplifies to
\[
\widetilde{O}\Bigl( n D T + n(\|A\|_0 + \|B\|_0) \log D \log^2 T \Bigr).
\]

If $\K$ is finite, each call to Algorithm~\ref{alg:one-step} has bit complexity
\[
\widetilde{O}\Bigl( n D T \log q / \mu + n(\|A\|_0 + \|B\|_0) \log D \log(DT/\mu)\log q \Bigr),
\]
where $\mu = 1/(3k)$ again. Substituting $\mu = 1/(3k)$, the per-iteration bit complexity is
\[
\widetilde{O}\Bigl( n D T k \log q + n(\|A\|_0 + \|B\|_0) \log D \log (TDk)\log q \Bigr).
\]
Summing over $k = O(\log T)$ iterations yields the total bit complexity
\[
\widetilde{O}\Bigl( n D T \log^2 T \log q + n(\|A\|_0 + \|B\|_0) \log^2 D \log^2 T \log q \Bigr),
\]
which simplifies to
\[
\widetilde{O}\Bigl( n D T \log q + n(\|A\|_0 + \|B\|_0) \log^2 D \log^2 T \log q \Bigr).
\]

\end{proof}

\subsection*{Amplifying the Success Probability of Algorithm~\ref{alg:recursive}}

The success probability of Algorithm~\ref{alg:recursive} is at least \(2/3\). For applications requiring higher reliability, we can amplify this probability to any desired \(1-\varepsilon\) by repeating the algorithm independently \(k\) times and taking a majority vote.

The following Chernoff bound for independent Bernoulli random variables is a special case of Theorem~4.4 in \cite{MitzenmacherUpfal2017} (which is stated for Poisson trials).

\begin{lemma}\cite{MitzenmacherUpfal2017}[Chernoff Bound, Upper Tail]\label{thm:chernoff-upper}
Let \(X_1, \dots, X_n\) be independent Bernoulli random variables, where \(\Pr(X_i = 1) = p_i\) and \(\Pr(X_i = 0) = 1 - p_i\). Let \(X = \sum_{i=1}^n X_i\) and \(\mu = \mathbb{E}[X] = \sum_{i=1}^n p_i\). Then for \(0 < \delta < 1\),
\[
\Pr\bigl(X \ge (1 + \delta)\mu\bigr) \le \exp\left(-\frac{\delta^2 \mu}{3}\right).
\]
\end{lemma}

\begin{lemma}\label{lem:amplify}
Let $\varepsilon > 0$ be given. If Algorithm~\ref{alg:recursive} is executed independently
$k = \left\lceil 36 \ln \frac{1}{\varepsilon} \right\rceil$
times, and the output is chosen by majority vote, then the probability that the final output is correct is at least \(1 - \varepsilon\).
\end{lemma}

\begin{proof}
Each execution of Algorithm~\ref{alg:recursive} succeeds with probability at least \(2/3\) and fails with probability at most \(1/3\). Let \(X_i\) be the indicator that the \(i\)-th execution fails. Then \(X_1,\dots,X_k\) are independent Bernoulli random variables with \(\Pr(X_i = 1) \le 1/3\). Let \(X = \sum_{i=1}^k X_i\) be the total number of failures.

We use the standard Chernoff bound for Bernoulli variables with bounded expectation. Since \(\Pr(X_i = 1) \le 1/3\), the random variable \(X\) is stochastically dominated by a \(\mathrm{Binomial}(k, 1/3)\) random variable \(Y\). That is, for any threshold \(t\),
\[
\Pr(X \ge t) \le \Pr(Y \ge t).
\]
Therefore,
\[
\Pr\left(X \ge \frac{k}{2}\right)
\le \Pr\left(Y \ge \frac{k}{2}\right).
\]
Applying the Chernoff bound to \(Y\), with \(\mu_Y = \mathbb{E}[Y] = k/3\) and \(\delta = 1/2\) (since \((1+\delta)\mu_Y = (3/2)(k/3) = k/2\)), we obtain
\[
\Pr\left(Y \ge \frac{k}{2}\right)
\le \exp\left(-\frac{(1/2)^2 \cdot (k/3)}{3}\right)
= \exp\left(-\frac{k}{36}\right)
\le \varepsilon,
\]
where the last inequality follows from \(k = \left\lceil 36 \ln\frac{1}{\varepsilon} \right\rceil\).

Thus, with probability at least \(1-\varepsilon\), the majority of the \(k\) executions returns the correct GCD.
\end{proof}

\begin{algorithm}
\caption{Recursive GCD Algorithm (Amplified)}
\label{alg:recursive-amplified}
\begin{algorithmic}[1]
\Require
    \begin{itemize}
        \item Polynomials $A, B \in \K[\x]$ with $G = \gcd(A,B)$;
        \item An upper bound $T\ge \|G\|_0$;
        \item A desired failure probability $\varepsilon > 0$;
        \item The field $\K$ has characteristic $0$ or characteristic greater than $\max_{i=1}^n\min\{\deg_{x_i} A,\deg_{x_i} B\}$.
    \end{itemize}
\Ensure
    With probability $\ge 1-\varepsilon$, output $G$; or ``Failure".

\State $k \gets \left\lceil 36 \ln\frac{1}{\varepsilon} \right\rceil$.
\State Initialize an empty list $\mathcal{L}$.

\For{$i = 1$ to $k$}
    \State Run Algorithm~\ref{alg:recursive} once on inputs $A, B, T$, obtaining a candidate $G_i$.
    \If{$G_i \neq$ ``Failure''}
        \State Append $G_i$ to $\mathcal{L}$.
    \EndIf
\EndFor

\If{$\mathcal{L}$ is empty}
    \State \Return ``Failure''.
\EndIf

\State \Return the polynomial that appears most frequently in $\mathcal{L}$ (majority vote).
\end{algorithmic}
\end{algorithm}

\begin{theorem}\label{rem-ampl-com}
Algorithm~\ref{alg:recursive-amplified} succeeds with probability at least \(1-\varepsilon\). Its expected complexity is:

\begin{itemize}
    \item If $\K$ is infinite,
   \[
\widetilde{O}\!\left( n D T \log \frac{1}{\varepsilon}
+ n(\|A\|_0 + \|B\|_0) \log^2 T \log D \log \frac{1}{\varepsilon}
\right)
\]
    field operations in $\K$.

    \item If $\K = \mathbb{F}_q$,
   \[
\widetilde{O}\!\left(
n D T \log q \log \frac{1}{\varepsilon}
+ n(\|A\|_0 + \|B\|_0) \log^2 T \log^2 D \log q \log \frac{1}{\varepsilon}
\right)
\]
    bit operations.
\end{itemize}
\end{theorem}

\begin{proof}
The success probability follows directly from Lemma~\ref{lem:amplify}. For the complexity, Algorithm~\ref{alg:recursive-amplified} executes Algorithm~\ref{alg:recursive} exactly $k = \left\lceil 36 \ln\frac{1}{\varepsilon} \right\rceil$ times. Multiplying the corresponding bounds from Theorem~\ref{thm:recursive-complexity} by this factor $k$ yields the stated complexities.
\end{proof}

\subsection{Probabilistic Verification (Algorithm~\ref{alg:verify-gcd})}
\label{sec:verification}

In the preceding algorithms, we assumed that the term bound $T$ is given as input. However, in practical applications, $T$ is usually unknown. We will remove this assumption through a guessing strategy, but before doing so, we need a verification algorithm to ensure the correctness of the output. In this section, we give such an algorithm.

Given polynomials $A, B, G \in \K[x_1, \dots, x_n]$, we wish to verify whether $G = \gcd(A, B)$ with high probability. The verification consists of three conditions:

\begin{enumerate}
    \item \textbf{Divisibility:} $G \mid A$ and $G \mid B$ in $\K[x_1,\dots,x_n]$.
    \item \textbf{Degree matching:} For each variable $x_i$ ($i=1,\dots,n$),
    \[
    \deg_{x_i}\bigl(\gcd(A, B)\bigr) = \deg_{x_i}(G).
    \]
    \item \textbf{Leading coefficient:} $G$ is monic (with respect to the fixed monomial order).
\end{enumerate}

If all three conditions hold, then $G = \gcd(A, B)$.

\subsubsection*{Checking Divisibility}

The divisibility $G \mid F$ can be checked by exact polynomial division in $\K[x_1,\dots,x_n]$. However, exact division of multivariate polynomials can be expensive. We use a probabilistic reduction: it suffices to check divisibility after a random evaluation.

For a vector $\mathbf{a} = (a_1,\dots,a_n) \in \K^n$, denote
\[
\mathbf{\check{a}}_{k \mapsto y}  := (a_1,\dots,a_{k-1}, y, a_{k+1},\dots,a_n),
\]
i.e., we replace the $k$-th coordinate with the variable $y$.

The following lemma shows that multivariate divisibility can be reduced to univariate divisibility with high probability.

\begin{lemma}\label{lem:divisibility-test}
Let $F, G \in \K[x_1,\dots,x_n]$ with $\deg_{x_i} G \le \deg_{x_i} F$ for all $i$. If $G \nmid F$, then for a random vector $\mathbf{a} \in S^n$ with $S \subseteq \K$ and $|S| \ge 8d^2 + 8d$, where $d$ is a bound on the partial degree of $F$ and $G$, with probability at least $3/4$, there exists at least one $k \in \{1,\dots,n\}$ such that
\[
G(\mathbf{\check{a}}_{k \mapsto y} ) \nmid F(\mathbf{\check{a}}_{k \mapsto y} ) \quad \text{in } \K[y].
\]
\end{lemma}

\begin{proof}
Since $G \nmid F$, let $Q = \gcd(F, G)$. If $\deg_{x_k} Q = \deg_{x_k} G$ for every $k$, then $Q$ and $G$ have the same degree in each variable, so $Q = c \cdot G$ for some $c \in \K^*$, implying $G \mid F$, a contradiction. Hence there exists some $k$ with $\deg_{x_k} Q < \deg_{x_k} G$.

Fix such a $k$. Let
\[
\tilde{F} = \frac{F}{Q}, \qquad \tilde{G} = \frac{G}{Q}.
\]
Then $\gcd(\tilde{F}, \tilde{G}) = 1$. Define
\[
\Gamma_k(\mathbf{x}) := \operatorname{lc}_{x_k}(F) \cdot \operatorname{lc}_{x_k}(G) \cdot \operatorname{Res}_{x_k}(\tilde{F}, \tilde{G}),
\]
where $\operatorname{Res}_{x_k}$ denotes the resultant with respect to $x_k$. Since $\tilde{F}$ and $\tilde{G}$ are coprime, $\Gamma_k$ is a nonzero polynomial. Its degree is bounded by
\[
\deg \Gamma_k \le 2d^2 + 2d.
\]
By the Schwartz--Zippel lemma,
\[
\Pr\bigl(\Gamma_k(\mathbf{a}) \neq 0\bigr) \ge 1 - \frac{2d^2 + 2d}{|S|}.
\]
If $\Gamma_k(\mathbf{a}) \neq 0$, then the leading coefficients do not vanish and the resultant is nonzero, so the univariate GCD satisfies
\[
\gcd(F(\mathbf{\check{a}}_{k \mapsto y} ), G(\mathbf{\check{a}}_{k \mapsto y} )) = Q(\mathbf{\check{a}}_{k \mapsto y} ) \cdot c
\]
for some $c \in \K^*$. Hence
\[
\deg\bigl(\gcd(F(\mathbf{\check{a}}_{k \mapsto y} ), G(\mathbf{\check{a}}_{k \mapsto y} ))\bigr)
= \deg_{x_k} Q
< \deg_{x_k} G
= \deg G(\mathbf{\check{a}}_{k \mapsto y} ).
\]
Therefore $G(\mathbf{\check{a}}_{k \mapsto y} ) \nmid F(\mathbf{\check{a}}_{k \mapsto y} )$. Choosing $|S| \ge 8d^2 + 8d$ gives success probability at least $3/4$.
\end{proof}

\subsubsection*{Checking Degree Matching}

For each variable $x_k$, we need to verify that
\[
\deg_{x_k}\bigl(\gcd(A, B)\bigr) = \deg_{x_k}(G).
\]
Let $Q = \gcd(A, B)$. We need to compute $\deg_{x_k} Q$.

For a random evaluation $\mathbf{a} \in S^n$, define
\[
A_k(y) = A(\mathbf{\check{a}}_{k \mapsto y} ), \qquad B_k(y) = B(\mathbf{\check{a}}_{k \mapsto y} ).
\]
If the leading coefficients of $A$ and $B$ with respect to $x_k$ do not vanish at $\mathbf{a}$, and the resultant $\operatorname{Res}_{x_k}(A/Q, B/Q)$ is nonzero, then
\[
\deg\bigl(\gcd(A_k, B_k)\bigr) = \deg_{x_k} Q.
\]
Define
\[
\Gamma_k(\mathbf{x}) := \operatorname{lc}_{x_k}(A) \cdot \operatorname{lc}_{x_k}(B) \cdot \operatorname{Res}_{x_k}(A/Q, B/Q).
\]
Then $\Gamma_k$ is a nonzero polynomial with $\deg \Gamma_k \le 2d^2 + 2d$. Let
\[
\Gamma(\mathbf{x}) := \prod_{k=1}^n \Gamma_k(\mathbf{x}).
\]
Then $\deg \Gamma \le 2nd^2 + 2nd$. If $\Gamma(\mathbf{a}) \neq 0$, then for every $k$,  the leading coefficients do not vanish and the resultant is nonzero, so the univariate GCD satisfies
\[
\gcd(A_k, B_k) = Q(a_1,\dots,a_{k-1}, y, a_{k+1},\dots,a_n) \cdot c_k
\]
for some $c_k \in \K^*$. Therefore $\deg(\gcd(A_k, B_k)) = \deg_{x_k} Q$ for all $k$.
Thus we can compute all partial degrees of $Q = \gcd(A, B)$ from the univariate GCDs at a single random point.

By the Schwartz--Zippel lemma, choosing $|S| \ge 8nd^2 + 8nd$ ensures that $\Gamma(\mathbf{a}) \neq 0$ with probability at least $3/4$.

We summarize the above discussion in the following lemma.
\begin{lemma}\label{lem:degree-matching}
Let $A, B \in \K[x_1,\dots,x_n]$ and let $Q = \gcd(A, B)$. Suppose $d$ is a bound on the partial degree of $A$ and $B$. For a random vector $\mathbf{a}=(a_1,\dots,a_n)\in S^n$ with $S \subseteq \K$ and $|S| \ge 8nd^2 + 8nd$, define
\[
A_k(y) = A(a_1,\dots,a_{k-1}, y, a_{k+1},\dots,a_n), \qquad
B_k(y) = B(a_1,\dots,a_{k-1}, y, a_{k+1},\dots,a_n).
\]
Then with probability at least $3/4$, for every $k = 1,\dots,n$,
\[
\deg\bigl(\gcd(A_k, B_k)\bigr) = \deg_{x_k} Q.
\]
\end{lemma}

\subsubsection*{The Complete Verification Algorithm}

Algorithm~\ref{alg:verify-gcd} verifies whether $G = \gcd(A, B)$ with high probability. It repeats both divisibility and degree-matching tests $\rho = \lceil \frac{\log_2\frac{1}{\mu}}{2} \rceil$ times to amplify the success probability.

\begin{algorithm}
\caption{Probabilistic GCD Verification over a Field}
\label{alg:verify-gcd}
\begin{algorithmic}[1]
\Require
    \begin{itemize}
        \item $A, B, G \in \K[x_1,\dots,x_n]$;
        \item A desired failure probability $\mu > 0$.
    \end{itemize}
\Ensure
    With probability $\ge 1-\mu$, output ``true'' if $G = \gcd(A,B)$, and ``false'' if $G \nmid \gcd(A,B)$.

\State $\rho \gets \lceil \frac{\log_2\frac{1}{\mu}}{2} \rceil$.
\State Choose a finite subset $S \subseteq \K^*$ with $|S| \ge 8nD^2 + 8nD$. If $\K$ is too small, extend it to a field $\K'$ containing such an $S$.

 \textbf{Quick degree check:}
\For{$k = 1$ to $n$}
    \If{$\deg_{x_k} G > \deg_{x_k} A$ or $\deg_{x_k} G > \deg_{x_k} B$}
        \State \Return ``false''
    \EndIf
\EndFor

\textbf{Check divisibility:}
\For{$r = 1$ to $\rho$}
    \State Choose random $\mathbf{a} \in S^n$.
    \For{$F \in \{A, B\}$}
        \For{$k = 1$ to $n$}
            \If{$G(\mathbf{\check{a}}_{k \mapsto y} ) \nmid F(\mathbf{\check{a}}_{k \mapsto y} )$ in $\K[y]$}
                \State \Return ``false''
            \EndIf
        \EndFor
    \EndFor
\EndFor

\textbf{Check degree matching:}
\State Initialize $d_k \gets \min(\deg_{x_k}(A), \deg_{x_k}(B))$ for $k = 1,\dots,n$.

\For{$r = 1$ to $\rho$}
    \State Choose random $\mathbf{a} \in S^n$.
    \For{$k = 1$ to $n$}
        \If{$\operatorname{lc}_{x_k}(A)(\mathbf{a}) = 0$ or $\operatorname{lc}_{x_k}(B)(\mathbf{a}) = 0$}
            \State \textbf{break} this iteration and choose a new $\mathbf{a}$
        \EndIf
        \State $H_k(y) \gets \gcd(A(\mathbf{\check{a}}_{k \mapsto y} ), B(\mathbf{\check{a}}_{k \mapsto y} ))$.
        \State $d_k \gets \min(d_k, \deg(H_k))$.
    \EndFor
\EndFor

\For{$k = 1$ to $n$}
    \If{$d_k \neq \deg_{x_k}(G)$}
        \State \Return ``false''
    \EndIf
\EndFor

\textbf{Check leading coefficient:}
\If{$G$ is not monic}
    \State \Return ``false''
\EndIf

\State \Return ``true''
\end{algorithmic}
\end{algorithm}

\begin{theorem}[Correctness of Algorithm~\ref{alg:verify-gcd}]\label{the:verify-gcd}
Let $A, B, G \in \K[x_1,\dots,x_n]$ be nonzero polynomials. Algorithm~\ref{alg:verify-gcd} returns ``true'' if $G = \gcd(A,B)$ with probability at least $1-\mu$, and returns ``false'' if $G \neq \gcd(A,B)$ with probability at least $1-\mu$.
\end{theorem}

\begin{proof}
We analyze the three possible cases.

\paragraph{Case 1: $G = \gcd(A, B)$.}
The quick degree check passes since $\deg_{x_k} G \le \deg_{x_k} A, \deg_{x_k} B$ for all $k$. Divisibility holds deterministically: $G(\mathbf{\check{a}}_{k \mapsto y} ) \mid F(\mathbf{\check{a}}_{k \mapsto y} )$ for every evaluation and every $F \in \{A, B\}$.

For degree matching, by Lemma~\ref{lem:degree-matching}, a single random $\mathbf{a}$ gives the correct partial degrees with probability at least $3/4$. Since a bad evaluation can only increase the computed degree (due to spurious common factors), taking the minimum over $\rho$ independent trials ensures that with probability at least $1 - (1/4)^\rho \ge 1-\mu$, the minimum equals the true degree $\deg_{x_k}(G)$. Thus the algorithm returns ``true'' with probability at least $1-\mu$.

\paragraph{Case 2: $G \neq \gcd(A, B)$ and ($G \nmid A$ or $G \nmid B$).}
If $\deg_{x_k} G > \deg_{x_k} A$ or $\deg_{x_k} G > \deg_{x_k} B$ for some $k$, the quick degree check catches it immediately and the algorithm returns ``false''.

Otherwise, assume $\deg_{x_k} G \le \deg_{x_k} A, \deg_{x_k} B$ for all $k$, but $G \nmid A$ or $G \nmid B$. By Lemma~\ref{lem:divisibility-test}, a single random $\mathbf{a}$ detects the failure with probability at least $3/4$. Repeating $\rho$ times ensures detection with probability at least $1 - (1/4)^\rho \ge 1-\mu$. Hence the algorithm returns ``false'' with probability at least $1-\mu$.

\paragraph{Case 3: $G \neq \gcd(A, B)$ but $G \mid A$ and $G \mid B$.}
In this case, $G$ is a proper common divisor. Hence there exists some $k$ such that
\[
\deg_{x_k} G < \deg_{x_k} \gcd(A,B).
\]
The algorithm initializes $d_k = \min(\deg_{x_k} A, \deg_{x_k} B)$, which already strictly exceeds $\deg_{x_k} G$. In each trial, either the computed degree is at least $\deg_{x_k} \gcd(A,B)$ (by Lemma~\ref{lem:degree-matching}), or the trial is discarded and $d_k$ remains unchanged. Taking the minimum over all trials can never reduce $d_k$ below $\deg_{x_k} \gcd(A,B)$; consequently, the final value of $d_k$ is always strictly larger than $\deg_{x_k} G$. Therefore, the algorithm always detects the mismatch and returns ``false'' in this case, with no probabilistic failure.
\end{proof}

\begin{remark}
The key insight in the degree-matching step is that a bad random evaluation can only overestimate the partial degree of the GCD, never underestimate it. Therefore, taking the minimum over multiple independent trials gives the true degree with high probability.
The failure probability can be made arbitrarily small by increasing $\rho = \lceil \frac{\log_2\frac{1}{\mu}}{2} \rceil$. The cost of repetition is $O(\log\frac{1}{\mu})$ times the cost of a single verification.
\end{remark}

\subsubsection*{Complexity Analysis of the Verification Algorithm}

We now analyze the complexity of Algorithm~\ref{alg:verify-gcd}. Let $d$ be a bound on the partial degree of $A$ and $B$. We assume operations in $\K$ cost $O(1)$ field operations; for finite fields $\mathbb{F}_q$, we multiply by $\log q$ for the bit complexity.

\begin{theorem}[Complexity of Algorithm~\ref{alg:verify-gcd}]\label{thm-veri-comple}
Let $A, B, G \in \K[x_1,\dots,x_n]$ with partial degree at most $d$. Algorithm~\ref{alg:verify-gcd} has the following expected complexity:

\begin{itemize}
    \item If $\K$ is infinite, it runs in
    \[
    \widetilde{O}\bigl( n (\|A\|_0 + \|B\|_0 + \|G\|_0) \log d \log\frac{1}{\mu} + n d \log\frac{1}{\mu} \bigr)
    \]
    field operations in $\K$.

    \item If $\K = \mathbb{F}_q$ is finite, it runs in
    \[
    \widetilde{O}\bigl( n (\|A\|_0 + \|B\|_0 + \|G\|_0) \log^2 d \log q \log\frac{1}{\mu} + n d \log q \log\frac{1}{\mu} \bigr)
    \]
    bit operations.
\end{itemize}
\end{theorem}

\begin{proof}
We analyze each step of the algorithm.

\paragraph{Quick degree check.}
Checking $\deg_{x_k} G \le \deg_{x_k} A$ and $\deg_{x_k} G \le \deg_{x_k} B$ for all $k$ requires inspecting the degrees of $O(\|A\|_0 + \|B\|_0 + \|G\|_0)$ terms. This costs $\widetilde{O}(n(\|A\|_0 + \|B\|_0 + \|G\|_0)\log d)$ bit operations, which is dominated by the main term.

\paragraph{Divisibility check.}
For each of the $\rho = O(\log\frac{1}{\mu})$ repetitions, we evaluate $G(\mathbf{\check{a}}_{k \mapsto y})$ and $F(\mathbf{\check{a}}_{k \mapsto y})$ for $F = A, B$ and $k = 1,\dots,n$.

We first show how to evaluate a single term $c x_1^{e_1}\cdots x_n^{e_n}$ at all $n$ univariate images $\mathbf{\check{a}}_{k \mapsto y}$ efficiently. At the point $\mathbf{a} = (a_1,\dots,a_n)$, the term evaluates to
\[
C = c a_1^{e_1}\cdots a_n^{e_n},
\]
which can be computed in $O(n \log d)$ field operations using fast exponentiation. For the $k$-th evaluation $\mathbf{\check{a}}_{k \mapsto y} = (a_1,\dots,a_{k-1}, y, a_{k+1},\dots,a_n)$, the term becomes
\[
C \cdot \frac{y^{e_k}}{a_k^{e_k}},
\]
which can be obtained from $C$ in $O(\log d)$ operations by computing $y^{e_k}$ and multiplying by the inverse of $a_k^{e_k}$. Thus, evaluating all $n$ univariate images of one term costs $O(n \log d)$ field operations. Summing over all terms gives $O(n(\|A\|_0 + \|B\|_0 + \|G\|_0) \log d)$ field operations per repetition.

Then we perform univariate polynomial division of degree at most $d$, which costs $\widetilde{O}(d)$ field operations. Since this is done for $F = A, B$ and $k = 1,\dots,n$, the divisibility cost per repetition is
\[
\widetilde{O}\bigl( n (\|A\|_0 + \|B\|_0 + \|G\|_0) \log d + n d \bigr).
\]
Repeating $\rho$ times gives
\[
\widetilde{O}\bigl( n (\|A\|_0 + \|B\|_0 + \|G\|_0) \log d \cdot \log\frac{1}{\mu} + n d \log\frac{1}{\mu} \bigr)
\]
field operations.

\paragraph{Degree matching check.}
For each repetition, we compute $H_k(y) = \gcd(A(\mathbf{\check{a}}_{k \mapsto y}), B(\mathbf{\check{a}}_{k \mapsto y}))$ for $k = 1,\dots,n$. The evaluations of $A$ and $B$ at the $n$ univariate points $\mathbf{\check{a}}_{k \mapsto y}$ can be performed using the same technique as in the divisibility check: for each term, we first evaluate it at $\mathbf{a}$ in $O(n \log d)$ operations, then obtain all $n$ univariate images in $O(n \log d)$ operations per term. Thus, evaluating all terms of $A$ and $B$ costs $O(n (\|A\|_0 + \|B\|_0 + \|G\|_0) \log d)$ field operations per repetition. Computing each univariate GCD of two polynomials of degree at most $d$ costs $\widetilde{O}(d)$ field operations, and there are $n$ such GCDs per repetition. Hence, for one repetition,
\[
\widetilde{O}\bigl( n (\|A\|_0 + \|B\|_0 + \|G\|_0) \log d + n d \bigr).
\]
Repeating $\rho$ times gives
\[
\widetilde{O}\bigl( n (\|A\|_0 + \|B\|_0 + \|G\|_0) \log d \cdot \log\frac{1}{\mu} + n d \log\frac{1}{\mu} \bigr)
\]
field operations.

\paragraph{Leading coefficient check.}
Checking whether $G$ is monic requires inspecting the leading term of $G$, which costs $O(1)$ field operations.

\paragraph{Field extension for finite fields.}
Step 2 of the algorithm requires a subset $S \subseteq \K^*$ with $|S| \ge 8nd^2 + 8nd$. If $\K$ is infinite, such a subset exists at no additional cost. If $\K = \mathbb{F}_q$ is finite and $q < 8nd^2 + 8nd$, we extend $\mathbb{F}_q$ to $\mathbb{F}_{q^{\ell}}$ with
\[
\ell = \left\lceil \log_q(8nd^2 + 8nd) \right\rceil = O(\log_q(nd)).
\]
Constructing such an extension costs $\operatorname{poly}(\ell)\log q$ bit operations \cite{Shoup1994}, and each arithmetic operation in $\mathbb{F}_{q^{\ell}}$ costs $\widetilde{O}(\ell)$ operations in $\mathbb{F}_q$, i.e., $\widetilde{O}(\log(nd)\log q)$ bit operations. This introduces an extra logarithmic factor $\log(nd)$ in the bit complexity.

Combining all steps, the total complexity is dominated by the divisibility and degree matching checks. If $\K$ is infinite, the complexity is
\[
\widetilde{O}\bigl( n (\|A\|_0 + \|B\|_0 + \|G\|_0) \log d \log\frac{1}{\mu} + n d \log\frac{1}{\mu} \bigr)
\]
field operations in $\K$.

If $\K = \mathbb{F}_q$, the bit complexity is
\[
\widetilde{O}\bigl( n (\|A\|_0 + \|B\|_0 + \|G\|_0) \log^2 d \log q \log\frac{1}{\mu} + n d \log q \log\frac{1}{\mu} \bigr).
\]
This proves the theorem.
\end{proof}

\subsection*{Extracting the Divisibility Test as a Standalone Subroutine}

The GCD verification algorithm presented in the previous section contains an independent divisibility-checking component: testing whether a given polynomial $G$ divides another polynomial $F$ in $\K[x_1,\dots,x_n]$. Since this subroutine will be used separately in later algorithms--particularly in the guessing strategy for term bounds--we extract it here as a standalone probabilistic algorithm.

The algorithm reduces multivariate divisibility to univariate divisibility tests via random evaluation, following the same one-sided error guarantee as in the verification algorithm: if $G \mid F$, it always returns ``true''; if $G \nmid F$, it returns ``false'' with high probability.

\begin{algorithm}
\caption{Probabilistic Divisibility Test over a Field}
\label{alg:divisibility-field}
\begin{algorithmic}[1]
\Require
    \begin{itemize}
        \item $F, G \in \K[x_1,\dots,x_n]$;
        \item A desired failure probability $\mu > 0$;
        \item A bound $d$ on the partial degree of $F$ and $G$.
    \end{itemize}
\Ensure
    \begin{itemize}
        \item If $G \mid F$, returns ``true'' with probability $1$;
        \item If $G \nmid F$, returns ``false'' with probability at least $1-\mu$.
    \end{itemize}
\end{algorithmic}
\end{algorithm}

The correctness and complexity of this subroutine are given in the following lemma, which follows directly from the divisibility-checking analysis in the proof of Theorem~\ref{the:verify-gcd} and Theorem~\ref{thm-veri-comple}.

\begin{lemma}\label{lem:divisibility-field}
Let $F, G \in \K[x_1,\dots,x_n]$ be nonzero polynomials over a field $\K$, and let $d$ be a bound on their partial degrees. Algorithm~\ref{alg:divisibility-field} satisfies:
\begin{enumerate}
    \item If $G \mid F$, it returns ``true'' with probability $1$;
    \item If $G \nmid F$, it returns ``false'' with probability at least $1-\mu$.
\end{enumerate}
The complexity of Algorithm~\ref{alg:divisibility-field} is as follows:
\begin{itemize}
    \item If $\K$ is infinite, it runs in
    \[
    \widetilde{O}\bigl( n(\|F\|_0+\|G\|_0)\log d \log\frac{1}{\mu} + nd \log\frac{1}{\mu} \bigr)
    \]
    field operations.
    \item If $\K = \mathbb{F}_q$ is finite, it runs in
    \[
    \widetilde{O}\bigl( n(\|F\|_0+\|G\|_0)\log^2 d \log q \log\frac{1}{\mu} + nd \log q \log\frac{1}{\mu} \bigr)
    \]
    bit operations.
\end{itemize}
\end{lemma}

\subsection{GCD Algorithm over a Field (Algorithm~\ref{alg:gcd-guess})}
\label{sec:field-no-t}

In this section, we remove the assumption that an upper bound $T$ on the number of terms of the GCD is supplied as input. Instead, we employ a doubling guess-and-verify strategy. The algorithm successively guesses $T = 2^k$ for $k = 0,1,2,\dots$, invokes an amplified version of the GCD algorithm (Algorithm~\ref{alg:recursive-amplified}) with the current guess, and verifies each candidate using the probabilistic GCD verification algorithm (Algorithm~\ref{alg:verify-gcd}). 

For each guess $k$, we run only a single GCD computation followed by a single verification, with the verification error budget $\delta_k$ decreasing exponentially with $k$. This ensures that the total error probability is controlled while keeping the number of verification calls minimal.

\begin{algorithm}[H]
\caption{GCD Computation without Term Bound}
\label{alg:gcd-guess}
\begin{algorithmic}[1]
\Require
    \begin{itemize}
        \item Polynomials $A, B \in \K[x_1,\dots,x_n]$ with $G = \gcd(A,B)$;
        \item A target error bound $0<\varepsilon<1$;
        \item The field $\K$ has characteristic $0$ or characteristic greater than $\max_{i=1}^n\min\{\deg_{x_i} A,\deg_{x_i} B\}$.
    \end{itemize}
\Ensure
    The GCD $G$ with probability at least $1-\varepsilon$, or ``Failure".

\State Compute $k_{\max} = \lceil \log_2 (d+1)^n \rceil$.

\For{each $k = 0, 1, 2, \dots, k_{\max}$}
    \State Set $T \gets 2^k$.
    \State Set $\delta_k \gets \varepsilon / 3^{k+1}$.
    \State Run Algorithm~\ref{alg:recursive-amplified} with term bound $T$ and tolerance $\delta_k$ to obtain a candidate $G_{\text{cand}}$.
    \If{$G_{\text{cand}} \neq$ ``Failure''}
        \State Call Algorithm~\ref{alg:verify-gcd} with inputs $A, B, G_{\text{cand}}$ and failure probability $\delta_k$.
        \If{verification returns \texttt{true}}
            \State \Return $G_{\text{cand}}$.
        \EndIf
    \EndIf
\EndFor

\State \Return ``Failure''.
\end{algorithmic}
\end{algorithm}

\subsection*{Probability Analysis}

Let $\varepsilon \in (0,1)$ be the desired overall error bound, and let $T_0 = \|G\|_0$ be the true number of terms of the GCD. For each guess $k$, both the amplified GCD algorithm and the verification algorithm have failure probability at most $\delta_k = \varepsilon / 3^{k+1}$.

\paragraph{Total success probability.}
For a fixed guess $k$, let $T = 2^k$. The amplified GCD algorithm (Algorithm~\ref{alg:recursive-amplified}) is guaranteed to succeed with probability at least $1-\delta_k$ only when the supplied bound $T$ satisfies $T \ge T_0$. If $T < T_0$, the algorithm may return ``Failure'' or an incorrect polynomial; in either case, the verification algorithm will reject an incorrect candidate with probability at least $1-\delta_k$.

Thus, for $k \ge k^*$ where $2^{k^*} \ge T_0>2^{k^*-1}$, the failure probability at step $k$ is at most $2\delta_k$. For $k < k^*$, the only way to fail is if verification incorrectly accepts an incorrect candidate, which occurs with probability at most $\delta_k$.

We analyze the success probability by considering a specific successful execution path. Let $k^* = \lceil \log_2 T_0 \rceil$. Consider the following events:

\begin{enumerate}
    \item For all guesses $k < k^*$, the verification algorithm correctly rejects any candidate produced. This occurs with probability at least $1-\delta_k$ for each such guess.
    \item At the correct guess $k = k^*$, the amplified GCD algorithm returns the correct GCD (probability $\ge 1-\delta_{k^*}$), and the verification algorithm correctly accepts it (probability $\ge 1-\delta_{k^*}$).
\end{enumerate}

By the union bound, the probability that all these events occur simultaneously is at least
\[
(1-\delta_0)(1-\delta_1)\cdots(1-\delta_{k^*-1})(1-2\delta_{k^*})
\ge 1 - \left(\sum_{k=0}^{k^*-1} \delta_k + 2\delta_{k^*}\right).
\]
Since $\delta_k = \varepsilon / 3^{k+1}$,
\[
\sum_{k=0}^{k^*-1} \delta_k + 2\delta_{k^*}
= \frac{\varepsilon}{3}\left(1 + \frac{1}{3} + \cdots + \frac{1}{3^{k^*-1}}\right) + \frac{2\varepsilon}{3^{k^*+1}}
= \frac{\varepsilon}{2} + \frac{\varepsilon}{2 \cdot 3^{k^*+1}}
< \varepsilon.
\]
Thus the success probability is at least
\[
\Pr(\text{success}) \ge 1 - \varepsilon.
\]

\begin{remark}
The geometric decay of $\delta_k$ ensures that the total error is bounded by a convergent series, regardless of the number of guesses. This avoids the need for a separate analysis of the ``correct guess'' event, as the error budget automatically covers all possibilities.
\end{remark}

\subsection*{Complexity Analysis}

Let $d$ be a bound on the partial degree of $A$ and $B$, and $D$ a bound on the degree of $A$ and $B$.

For a guess $T = 2^k$, the amplified GCD algorithm (Algorithm~\ref{alg:recursive-amplified}) has the following complexities, depending on the field $\K$:

\begin{itemize}
    \item If $\K$ is infinite,
    \[
    \widetilde{O}\bigl( n T D \log\frac{1}{\delta_k}+ n(\|A\|_0 + \|B\|_0) \log^2 T \log D \log\frac{1}{\delta_k}\bigr)
    \]
    field operations in $\K$.

    \item If $\K = \mathbb{F}_q$ is finite,
    \[
    \widetilde{O}\bigl( n T D \log q \log\frac{1}{\delta_k}+ n(\|A\|_0 + \|B\|_0) \log^2 T \log^2 D \log q \log\frac{1}{\delta_k}\bigr)
    \]
    bit operations.
\end{itemize}

The verification algorithm (Algorithm~\ref{alg:verify-gcd}) has analogous complexities:

\begin{itemize}
    \item If $\K$ is infinite,
    \[
    \widetilde{O}\bigl( n(T+\|A\|_0+\|B\|_0) \log d \log\frac{1}{\delta_k}+ n d \log\frac{1}{\delta_k}\bigr)
    \]
    field operations in $\K$.

    \item If $\K = \mathbb{F}_q$ is finite,
    \[
    \widetilde{O}\bigl( n(T+\|A\|_0+\|B\|_0) \log^2 d \log q \log\frac{1}{\delta_k}+ n d \log q \log\frac{1}{\delta_k}\bigr)
    \]
    bit operations.
\end{itemize}

In both cases, the verification cost is dominated by the GCD computation. Since the argument is identical for infinite and finite fields up to the extra logarithmic factors $\log q$ and $\log D$, we present the analysis for the infinite field case for simplicity. The finite field case follows by replacing each field operation with its bit cost $\widetilde{O}(\log q)$ and accounting for the additional factor $\log D$ in the second term.

Thus, for a guess $k$ (so that $T = 2^k$), the total cost is dominated by the amplified GCD computation:
\[
C_k = \widetilde{O}\Bigl( n T D \log\frac{1}{\delta_k}+ n(\|A\|_0 + \|B\|_0) \log^2 T \log D \log\frac{1}{\delta_k}\Bigr),
\]
for infinite fields.

\paragraph{Average-case complexity.}
Let $E_k$ be the event that the algorithm stops at guess $k$. The events $E_0, E_1, \dots, E_{k_{\max}}$ form a partition of the probability space. When $E_k$ occurs, the algorithm performs $C_0 + C_1 + \cdots + C_k$ operations. Thus the expected cost is
\[
\mathbb{E}[\text{cost}]
= \sum_{k=0}^{k_{\max}} \Pr(E_k) \sum_{i=0}^{k} C_i.
\]
Equivalently,
\[
\mathbb{E}[\text{cost}]
= \sum_{k=0}^{k_{\max}} C_k \cdot \Pr(E_k \cup E_{k+1} \cup \cdots \cup E_{k_{\max}}).
\]

Recall that $k^* = \lceil \log_2 T_0 \rceil$. We split the sum into two parts: $k \le k^*$ and $k > k^*$.

For $k \le k^*$, the probabilities are bounded by $1$. Since $T = 2^k$ and $\delta_k = \varepsilon / 3^{k+1}$, we have $\log T = O(\log T_0)$ and $\log\frac{1}{\delta_k}= O(\log T_0 + \log\frac{1}{\varepsilon}) = O(\log T_0 \log\frac{1}{\varepsilon})$ for $k \le k^*$. Hence the contribution is at most
\[
\sum_{k=0}^{k^*} C_k 
= \widetilde{O}\Bigl( k^* \cdot C_{k^*} \Bigr)
= \widetilde{O}\Bigl( 
n T_0 D \log\frac{1}{\varepsilon} 
+ n(\|A\|_0 + \|B\|_0) \log^4 T_0 \log D \log\frac{1}{\varepsilon}
\Bigr),
\]
since the sequence $C_k$ is increasing and is dominated by its largest term.

For $k > k^*$, the probability of reaching guess $k$ decays exponentially.

Now consider the tail probability for $k = k^* + s$ where $s \ge 1$. We need to bound
\[
\Pr(E_{k^*+s}) + \Pr(E_{k^*+s+1}) + \cdots + \Pr(E_{k_{\max}})
= 1 - \Pr(E_0 \cup E_1 \cup \cdots \cup E_{k^*+s-1}).
\]

Let $A = E_0 \cup E_1 \cup \cdots \cup E_{k^*+s-2}$. Then
\[
\Pr(E_0 \cup E_1 \cup \cdots \cup E_{k^*+s-1})=\Pr(A \cup E_{k^*+s-1})
= \Pr(A) + \Pr(E_{k^*+s-1} \mid \bar{A}) \Pr(\bar{A}).
\]
Since $\Pr(E_{k^*+s-1} \mid \bar{A}) \ge 1 - 2\delta_{k^*+s-1}$ (conditioned on not having stopped earlier, the algorithm either fails to compute or fails to verify at step $k^*+s-1$), we have
\[
\Pr(A \cup E_{k^*+s-1})
\ge \Pr(A) + (1 - 2\delta_{k^*+s-1}) \Pr(\bar{A}).
\]
Thus
\[
1 - \Pr(A \cup E_{k^*+s-1})
\le 2\delta_{k^*+s-1} \Pr(\bar{A})
\le 2\delta_{k^*+s-1}.
\]
Therefore,
\[
\Pr(E_{k^*+s}) + \Pr(E_{k^*+s+1}) + \cdots + \Pr(E_{k_{\max}})
\le 2\delta_{k^*+s-1}.
\]

Substituting this bound into the expected cost formula
\[
\mathbb{E}[\text{cost}]
= C_0 + C_1 \Pr(E_1 \cup \cdots \cup E_{k_{\max}}) + \cdots + C_{k_{\max}} \Pr(E_{k_{\max}}),
\]
we obtain
\[
\mathbb{E}[\text{cost}]
\le \sum_{k=0}^{k^*} C_k
+ C_{k^*+1} \cdot 2\delta_{k^*}
+ C_{k^*+2} \cdot 2\delta_{k^*+1}
+ \cdots + C_{k_{\max}} \cdot 2\delta_{k_{\max}-1}.
\]

Therefore,
\[
\sum_{k=0}^{k^*} C_k \le (k^*+1) C_{k^*}
= \widetilde{O}\Bigl( 
n T_0 D \log\frac{1}{\varepsilon} 
+ n(\|A\|_0 + \|B\|_0) \log^4 T_0 \log D \log\frac{1}{\varepsilon}
\Bigr).
\]

For the tail, writing $s = k - k^*$, we have
\[
\sum_{k=k^*+1}^{k_{\max}} 2\delta_{k-1} C_k
= 2 \sum_{s=1}^{k_{\max}-k^*} \frac{\varepsilon}{3^{k^*+s}} \, C_{k^*+s}.
\]

Since $\log(1/\delta_{k^*+s}) = \log(3^{k^*+s+1}/\varepsilon) = O((k^*+s)\log\frac{1}{\varepsilon})$,
\[
C_{k^*+s} = \widetilde{O}\!\left(
n 2^{k^*+s} D (k^*+s) \log\frac{1}{\varepsilon}
+ n(\|A\|_0 + \|B\|_0) (k^*+s)^3 \log D \log\frac{1}{\varepsilon}
\right).
\]

By the definition of $\widetilde{O}$, there exists a constant $c > 0$ and an integer $\lambda \ge 1$ such that
\[
\frac{1}{3^{k^*+s}} \cdot \widetilde{O}\bigl(2^{k^*+s} (k^*+s)\bigr)
\le c \cdot \left(\frac{2}{3}\right)^{k^*+s} (k^*+s)^{\lambda}.
\]
Hence
\[
\sum_{s=1}^{\infty} \frac{1}{3^{k^*+s}} \cdot \widetilde{O}\bigl(2^{k^*+s} (k^*+s)\bigr)
\le c \sum_{s=1}^{\infty} \left(\frac{2}{3}\right)^{k^*+s} (k^*+s)^{\lambda}
\le c \sum_{s=1}^{\infty} \left(\frac{2}{3}\right)^s s^{\lambda}.
\]

The last series converges by the root test:
\[
\limsup_{s \to \infty} \left( \left(\frac{2}{3}\right)^s s^{\lambda} \right)^{1/s}
= \frac{2}{3} \lim_{s \to \infty} s^{\lambda/s}
= \frac{2}{3} < 1.
\]
Thus $\sum_{s=1}^{\infty} 3^{-(k^*+s)} \cdot \widetilde{O}((k^*+s)^3)$ also converges.

Since $n$, $D$, and $\log\frac{1}{\varepsilon}$ are fixed with respect to the summation index $s$, we conclude that
\[
\sum_{k=k^*+1}^{k_{\max}} 2\delta_{k-1} C_k
= 2 \sum_{s=1}^{k_{\max}-k^*} \frac{\varepsilon}{3^{k^*+s}} C_{k^*+s}
= \widetilde{O}\Bigl( nD\log\frac{1}{\varepsilon} + n(\|A\|_0 + \|B\|_0) \log D \log\frac{1}{\varepsilon} \Bigr).
\]

Consequently, the expected cost is dominated by the $k \le k^*$ terms, yielding
\[
\mathbb{E}[\text{cost}]
= \widetilde{O}\Bigl(
n T_0 D \log\frac{1}{\varepsilon}
+ n(\|A\|_0 + \|B\|_0) \log^4 T_0 \log D \log\frac{1}{\varepsilon}
\Bigr).
\]

For finite fields $\K = \mathbb{F}_q$, the analysis is entirely parallel to the infinite case. Replacing each field operation by its bit cost $\widetilde{O}(\log q)$ and accounting for the additional factor $\log D$ in the second term, the expected bit complexity becomes
\[
\mathbb{E}[\text{cost}]
= \widetilde{O}\Bigl(
n T_0 D \log q \log\frac{1}{\varepsilon}
+ n(\|A\|_0 + \|B\|_0) \log^4 T_0 \log^2 D \log q \log\frac{1}{\varepsilon}
\Bigr).
\]

We summarize the above analysis in the following theorem.

\begin{theorem}\label{thm:gcd-no-t}
Let $A, B \in \K[x_1,\dots,x_n]$ over a field $\K$ with $G = \gcd(A,B)$, and let $\varepsilon \in (0,1)$. Suppose $\operatorname{char}(\K) = 0$ or $\operatorname{char}(\K) > \max_{i=1}^n\min\{\deg_{x_i} A,\deg_{x_i} B\}$. Then there exists a randomized algorithm that computes $G$ with probability at least $1-\varepsilon$ and whose expected cost is as follows:

\begin{itemize}
    \item If $\K$ is infinite,
    \[
    \widetilde{O}\Bigl(
    n T_0 D \log\frac{1}{\varepsilon}
    + n(\|A\|_0 + \|B\|_0) \log^4 T_0 \log D \log\frac{1}{\varepsilon}
    \Bigr)
    \]
    field operations in $\K$, where $T_0 = \|G\|_0$ is the true term count of the GCD and $D$ is a bound on the total degree of $A$ and $B$.

    \item If $\K = \mathbb{F}_q$ is finite,
    \[
    \widetilde{O}\Bigl(
    n T_0 D \log q \log\frac{1}{\varepsilon}
    + n(\|A\|_0 + \|B\|_0) \log^4 T_0 \log^2 D \log q \log\frac{1}{\varepsilon}
    \Bigr)
    \]
    bit operations.
\end{itemize}
\end{theorem}

\begin{remark}
If $T \ge \max\{\|A\|_0, \|B\|_0, \|G\|_0\}$, the complexity simplifies to
\[
\widetilde{O}\bigl( n T D \log\frac{1}{\varepsilon} \bigr)
\]
field operations in $\K$ for infinite fields, and
\[
\widetilde{O}\bigl( n T D \log q \log\frac{1}{\varepsilon} \bigr)
\]
bit operations for finite fields $\mathbb{F}_q$.
\end{remark}

\section{GCD Algorithm over the Integers}
\label{sec:int-gcd}

In this section, we present three algorithms for computing the GCD of polynomials with integer coefficients, with progressively weaker assumptions on the available input bounds. The first (Algorithm~\ref{alg:intgcd-crt}) assumes prior knowledge of both a term bound $T \ge \|G\|_0$ and a coefficient bound $\Ho \ge \|G\|_\infty$. The second (Algorithm~\ref{alg:intgcd-no-bounds}) removes the need for the coefficient bound and requires only the term bound $T$. The third (Algorithm~\ref{alg:intgcd-guess}) eliminates both assumptions entirely, requiring no a priori bounds on the GCD. The core strategy common to all three algorithms is to reduce the problem modulo several small primes, compute the GCD over each finite field, and then reconstruct the integer coefficients via rational reconstruction.

We first recall the definition of the greatest common divisor for polynomials with integer coefficients.

\begin{definition}
Let $A, B \in \mathbb{Z}[x_1,\dots,x_n]$ be nonzero polynomials with integer coefficients. A polynomial $G \in \mathbb{Z}[x_1,\dots,x_n]$ is called the \textbf{greatest common divisor} of $A$ and $B$, denoted $G = \gcd(A,B)$, if it satisfies the following three conditions:

\begin{enumerate}
    \item $G$ divides both $A$ and $B$ in $\mathbb{Z}[x_1,\dots,x_n]$;
    \item every common divisor of $A$ and $B$ in $\mathbb{Z}[x_1,\dots,x_n]$ divides $G$;
    \item the leading coefficient of $G$ is positive.
\end{enumerate}
\end{definition}

\begin{remark}
The third condition ensures the uniqueness of the GCD over $\mathbb{Z}[x_1,\dots,x_n]$. Over a field $\K$, the GCD is typically normalized to be monic (leading coefficient $1$). Over $\mathbb{Z}$, however, the leading coefficient is required to be positive to fix the sign ambiguity, since $G$ and $-G$ are associates in $\mathbb{Z}[x_1,\dots,x_n]$.
\end{remark}

One possible approach would be to embed $\mathbb{Z}[x_1,\dots,x_n]$ into $\mathbb{Q}[x_1,\dots,x_n]$ and apply the algorithm directly over $\mathbb{Q}$, since $\mathbb{Q}$ is a field of characteristic zero and the algorithm is fully applicable. However, this approach introduces a significant practical and theoretical issue: the complexity analysis of our algorithm counts field operations, not bit operations. When working over $\mathbb{Q}$, each field element may have a large bit size. For example, evaluating a polynomial $f(x)$ of degree $D$ at a rational point $x = 2$ produces numbers with a bit size of $O(D)$, which would introduce an extra factor of $D$ in the bit complexity. Consequently, the overall bit complexity would become quadratic in $D$, rather than linear.

A natural alternative would be to choose a single prime \(p > 2\Ho^2\), where \(\Ho\) is a bound on the
coefficients of the GCD. Under such a choice, the reduction modulo \(p\) uniquely determines the integer
GCD, and only one finite-field GCD computation is needed. 
However, this approach introduces an undesirable overhead: finding such a prime requires testing candidates
in an interval of size proportional to \(\Ho^2\), and each primality test costs \(\operatorname{poly}(\log \Ho)\)
bit operations. Even with efficient primality testing, the overall cost includes a factor of \(\log^4 \Ho\), which would prevent us from achieving the desired linear complexity in \(\log \Ho\).

To avoid this overhead, we instead use \(\kappa = O(\log \Ho)\) small primes \(p_1,\dots,p_\kappa\). 
For each such prime, we compute the modular GCD \(\mathcal{G}_{i} = \gcd(A \bmod p_i, B \bmod p_i)\) over
\(\mathbb{F}_{p_i}[x_1,\dots,x_n]\) using the finite-field algorithm developed in previous sections.
The product of these primes is chosen to exceed \(2\Ho^2\), where \(\Ho\) is a bound on the coefficients
of $G$. The rational reconstruction then uniquely recovers all coefficients of \(G\).

This strategy has two key advantages:
\begin{enumerate}
    \item The primes are small, so primality testing and modular arithmetic are inexpensive.
    \item The number of primes is \(O(\log \Ho)\), which is linear in the \(\widetilde{O}\) notation.
\end{enumerate}

As a result, the overall bit complexity contains only a linear factor in \(\log \Ho\), 
avoiding the \(\log^4 \Ho\) overhead of the single-large-prime approach.

The framework proceeds as follows. We first compute the partial degrees \(d_i = \deg_{x_i} \gcd(A,B)\)
using Algorithm~\ref{alg:int-degree-match}. These degrees are used to identify lucky primes and to
verify that no degree loss occurs modulo any chosen prime. We then collect \(\kappa\) lucky primes,
compute the modular GCDs, reconstruct the candidate integer GCD via rational reconstruction.

\subsection{Content, Primitive Part, and Lucky Primes}
\label{sec:lucky-primes}

Let $F \in \mathbb{Z}[x_1, \dots, x_n]$ be a nonzero polynomial. The \textbf{content} of $F$, denoted $\operatorname{cont}(F)$, is the greatest common divisor of all coefficients of $F$, taken as a positive integer. That is,
\[
\operatorname{cont}(F) = \gcd\{ \text{all nonzero coefficients of } F \} \in \mathbb{Z}_{>0}.
\]
The \textbf{primitive part} of $F$, denoted $\operatorname{pp}(F)$, is defined by
\[
\operatorname{pp}(F) = \frac{F}{\operatorname{cont}(F)} \in \mathbb{Z}[x_1, \dots, x_n].
\]
A polynomial is called \textbf{primitive} if $\operatorname{cont}(F) = 1$.

If $F = 0$, we define $\operatorname{cont}(0) = 0$ and $\operatorname{pp}(0) = 0$ by convention.


For polynomials $A, B \in \mathbb{Z}[x_1, \dots, x_n]$, the content of their GCD satisfies
\[
\operatorname{cont}(\gcd(A, B)) = \gcd(\operatorname{cont}(A), \operatorname{cont}(B)).
\]
Moreover,
\[
\gcd(A, B) = \gcd(\operatorname{cont}(A), \operatorname{cont}(B)) \cdot \gcd(\operatorname{pp}(A), \operatorname{pp}(B)).
\]
Thus, the GCD computation reduces to computing the GCD of two primitive polynomials, followed by a simple content recovery.

\subsection*{Lucky Primes}

The following definition characterizes primes for which the finite field GCD computation correctly reflects the integer GCD. In accordance with Brown's work, we assume throughout that $A$ and $B$ are primitive. This assumption is without loss of generality, as the GCD computation reduces to computing $\gcd(\operatorname{pp}(A), \operatorname{pp}(B))$.

\begin{definition}[Lucky Prime]
Assume $A,B\in \mathbb{Z}[x_1, \dots, x_n]$ are primitive. A prime $p$ is called \textbf{lucky} for the pair $(A, B)$ if, for every variable $x_i$ ($i=1,\dots,n$),
\[
\deg_{x_i}(\mathcal{G}) = \deg_{x_i}(G),
\]
where $\mathcal{G} = \gcd(A \mod p, B \mod p) \in \mathbb{F}_p[x_1, \dots, x_n]$ and $G = \gcd(A,B) \in \mathbb{Z}[x_1, \dots, x_n]$. 
That is, the degree in each variable of the modular GCD matches the degree in the corresponding variable of the integer GCD. A prime that is not lucky is called \textbf{unlucky}.
\end{definition}

By Brown \cite{brown1971}, every unlucky prime divides a fixed nonzero integer $\sigma$ determined by $A$ and $B$; hence the number of unlucky primes is finite and bounded. 

Adapting the theorem to our notation, we have:

\begin{theorem}\label{Brownthe2}[Brown, Theorem 2]\cite{brown1971} Assume $A,B\in \mathbb{Z}[x_1, \dots, x_n]$ are primitive. 
Let $u$ be the number of unlucky primes $p > \lambda$ for some $\lambda \ge 2$. Then
\[
u \le \frac{nd}{\ln \lambda} \cdot \ln(2d \Hi^2 t^2),
\]
where $t = \max\{\|A\|_0, \|B\|_0\}$, $\Hi=\max\{\|A\|_{\infty},\|B\|_{\infty}\}$, $d=\max_{i=1}^n\max\{\deg_{x_i} A,\deg_{x_i} B\}$.
\end{theorem}

\subsection*{Existence of Sufficiently Many Small Lucky Primes}

We now establish that there are sufficiently many lucky primes in a relatively short interval.

\begin{lemma}\label{lem:lucky-prime-any-eps}
Let $A,B\in\mathbb{Z}[x_1,\dots,x_n]$ be primitive, and let $\varepsilon\in(0,1)$. Suppose $\lambda\in\mathbb{N}$ satisfies
\[
\lambda \ge \max\left(21,\ \frac{3nd}{5\varepsilon}\ln(2d \Hi^2 t^2)\right),
\]
where $d=\max_{i=1}^n\max\{\deg_{x_i} A,\deg_{x_i} B\}$, $t=\max(\|A\|_0,\|B\|_0)$, and $\Hi=\max\{\|A\|_\infty,\|B\|_\infty\}$. If a prime $p$ is chosen uniformly at random from $[\lambda,2\lambda]$, then $p$ is lucky with probability at least $1-\varepsilon$.
\end{lemma}

\begin{proof}
By Theorem~\ref{Brownthe2}, the number $u$ of unlucky primes greater than $\lambda$ satisfies
\[
u \le \frac{nd}{\ln\lambda}\ln(2d \Hi^2 t^2).
\]
By the Rosser--Schoenfeld theorem \cite[Cor.3]{1962Approximate}, as $\lambda\ge 21$, the interval $[\lambda,2\lambda]$ contains at least $\frac{5}{3}\lambda/\ln\lambda$ primes. Hence
\[
\Pr(p\text{ is unlucky})
\le \frac{u}{\frac{5}{3}\lambda/\ln\lambda}
\le \frac{3nd}{5\lambda}\ln(2d \Hi^2 t^2)
\le \varepsilon,
\]
where the last inequality follows from $\lambda \ge \frac{3nd}{5\varepsilon}\ln(2d \Hi^2 t^2)$. Therefore, the probability that $p$ is lucky is at least $1-\varepsilon$.
\end{proof}

The following theorem shows that if we randomly select $\kappa$ distinct primes from $[\lambda, 4\lambda]$, with high probability all of them are lucky, and their product is large enough for rational reconstruction.

\begin{theorem}\label{thm:lucky-primes-crt}
Let $A,B\in\mathbb{Z}[x_1,\dots,x_n]/\mathbb{Z}$ be primitive, and let $\mu \in (0,\frac{1}{2})$. Suppose $\lambda$ satisfies
\[
\lambda \ge \max\left(
    21,\ 
    \frac{3nd}{5\mu} \ln(2d \Hi^2 t^2)\cdot \ln(2\Ho^2+1) 
\right),
\]
and let $\kappa: = \left\lceil \ln(2\Ho^2+1) / \ln \lambda \right\rceil$. If $\kappa$ distinct primes $p_1,\dots,p_\kappa$ are chosen uniformly at random from $[\lambda, 4\lambda]$, then:

\begin{enumerate}
    \item All chosen primes are lucky with probability at least $1 - \mu$.
    \item The product $p_1 \cdots p_\kappa > 2\Ho^2$, so the rational reconstruction can uniquely recover all coefficients of $G$.
\end{enumerate}
Here $d=\max_{i=1}^n\max\{\deg_{x_i} A,\deg_{x_i} B\}$, $t=\max(\|A\|_0,\|B\|_0)$, and $\Hi=\max\{\|A\|_\infty,\|B\|_\infty\},\Ho=\|G\|_\infty$.

\end{theorem}

\begin{proof}
By the Rosser--Schoenfeld theorem \cite[Cor.~3]{1962Approximate}, the interval $[\lambda, 2\lambda]$ contains at least
\[
\frac{5}{3}\frac{\lambda}{\ln\lambda}
\ge 
\frac{5}{3}\cdot \frac{3nd}{5\mu} \ln(2d \Hi^2 t^2)\cdot \frac{\ln(2\Ho^2+1)}{\ln \lambda}
\ge \frac{\ln(2\Ho^2+1)}{\ln \lambda}
\]
primes, where the first inequality follows from our choice of $\lambda\ge \frac{3nd}{5\mu} \ln(2d \Hi^2 t^2)\cdot \ln(2\Ho^2+1)$, and the second holds because $\frac{5}{3}\cdot \frac{3nd}{5\mu} \ln(2d \Hi^2 t^2)\ge 1$ whenever $\mu < 1/2$ and $n,d,\Hi,t \ge 1$.

Thus $[\lambda, 2\lambda]$ contains at least $\kappa=\left\lceil \ln(2\Ho^2+1) / \ln \lambda \right\rceil$ primes. By the same argument, the interval $[2\lambda, 4\lambda]$ also contains at least $\kappa$ primes. Hence $[\lambda, 4\lambda]$ contains at least $2\kappa$ primes in total.

We set the per-prime failure probability to
\[
\varepsilon := \frac{\mu}{\ln(2\Ho^2+1)}.
\]
Since \(0 < \mu < 1/2\) and \(\Ho \ge 1\), we have \(\ln(2\Ho^2) \ge \ln 2 > 0\), and hence
$\varepsilon \le \frac{1/2}{\ln 2} = \frac{1}{2\ln 2} < 1.$
Thus \(\varepsilon < 1\), so \(\varepsilon\) is indeed a valid failure probability.

Since $ \frac{3nd}{5\mu} \ln(2d \Hi^2 t^2)\cdot \ln(2\Ho^2+1)=\frac{3nd}{5}\ln(2d \Hi^2 t^2)\cdot \ln(2\Ho^2+1)/\mu =\frac{3nd}{5\varepsilon}\ln(2d \Hi^2 t^2)$, by Lemma \ref{lem:lucky-prime-any-eps}, choosing a prime in $[\lambda,2\lambda]$ yields a lucky prime with probability $\ge 1-\varepsilon=1-\frac{\mu}{\ln(2\Ho^2+1)}\ge 1-\frac{\mu}{\kappa}$.

Since $\frac{3nd}{5\mu} \ln(2d \Hi^2 t^2)\cdot \ln(2\Ho^2+1)
= \frac{3nd}{5\varepsilon}\ln(2d \Hi^2 t^2)$,
by Lemma~\ref{lem:lucky-prime-any-eps}, a prime chosen uniformly from $[\lambda, 2\lambda]$ is lucky with probability at least
\[
1-\varepsilon = 1 - \frac{\mu}{\ln(2\Ho^2+1)} \ge 1 - \frac{\mu}{\kappa}.
\]
Let $N_1$ be the number of primes in $[\lambda, 2\lambda]$ and let $u$ be the total number of unlucky primes greater than $\lambda$. Since at most $u$ unlucky primes lie in $[\lambda, 2\lambda]$, the fraction of unlucky primes in $[\lambda, 2\lambda]$ is at most $u/N_1$. Hence
$\frac{u}{N_1} \le \frac{\mu}{\kappa}.$

Now consider sampling $\kappa$ primes uniformly without replacement from $[\lambda, 4\lambda]$. Let $N_1$ and $N_2$ denote the number of primes in $[\lambda, 2\lambda]$ and $[2\lambda, 4\lambda]$, respectively. By the Rosser--Schoenfeld theorem, $N_1, N_2 \ge \kappa$. The conditional probability that the $i$-th draw is lucky, given that all previous draws were lucky, is at least
\[
1 - \frac{u}{N_1 + N_2 - (i-1)}, \qquad i = 1, \dots, \kappa.
\]
Thus the failure probability is bounded by
\[
\sum_{i=0}^{\kappa-1} \frac{u}{N_1 + N_2 - i}
\le \frac{u\kappa}{N_1 + N_2 - \kappa + 1}
\le \frac{u\kappa}{N_1}
\le \mu,
\]
where the penultimate inequality follows from $N_2 \ge \kappa$, and the last from $u/N_1 \le \mu/\kappa$.

Therefore, the probability that all $\kappa$ selected primes are lucky is at least $1 - \mu$.

Second, since each $p_i \ge \lambda$, we have
\[
p_1 \cdots p_\kappa \ge \lambda^\kappa \ge \lambda^{\ln(2\Ho^2+1)/\ln \lambda} = 2\Ho^2+1.
\]
Therefore, the rational reconstruction uniquely determines all integer coefficients of $G$ from their residues modulo the $p_i$.
\end{proof}

\subsection{Computing Partial Degrees (Algorithm~\ref{alg:int-degree-match})}
\label{sec:partial-deg}

Before presenting the modular GCD algorithm over the integers, we introduce a probabilistic subroutine that computes the partial degrees
\[
d_i = \deg_{x_i} \gcd(A,B), \qquad i = 1,\dots,n.
\]

These degrees serve a dual purpose. First, they are part of the correctness condition for the integer GCD: any candidate $H$ must satisfy $\deg_{x_i} H = d_i$ for all $i$. Second, and more importantly for the modular algorithm, they provide a criterion for detecting whether a prime $p$ is lucky.

Recall that a prime $p$ is called lucky for $(A,B)$ if
\[
\deg_{x_i} \gcd(A \bmod p, B \bmod p) = d_i \qquad \text{for all } i=1,\dots,n.
\]
Thus, once the true partial degrees $d_i$ are known, we can test whether a given prime $p$ is lucky by computing the modular GCD degrees and comparing them with $d_i$. This comparison is efficient and avoids the need to compute the full modular GCD for the purpose of prime selection.

The subroutine computes the $d_i$'s by reducing the multivariate problem to univariate GCD computations modulo a random prime, and taking the minimum over multiple random evaluations. The following algorithm formalizes this procedure.

The key observation is that  a unlucky prime or a bad random evaluation can only \emph{overestimate} the partial degree of the GCD, never underestimate it. Therefore, by taking the minimum over sufficiently many independent trials, we obtain the true degree with high probability.

Recall that $\mathbf{\check{a}}_{k \mapsto y}  = (a_1,\dots,a_{k-1}, y, a_{k+1},\dots,a_n)$, where $\mathbf{a} = (a_1,\dots,a_n)$. For notational simplicity, when $\mathbf{a} \in \mathbb{F}_p^n$, we write $F(\mathbf{\check{a}}_{k \mapsto y} )$ for the polynomial in $\mathbb{F}_p[y]$ obtained by reducing the coefficients of $F$ modulo $p$ and then substituting $x_k \mapsto y$ and $x_i \mapsto a_i$ for $i \ne k$. This convention is used throughout the remainder of the paper.

\begin{algorithm}
\caption{Computing Partial Degrees of the Integer GCD}
\label{alg:int-degree-match}
\begin{algorithmic}[1]
\Require
    \begin{itemize}
        \item $A, B \in \mathbb{Z}[x_1,\dots,x_n]$;
        \item A desired failure probability $\mu\in(0,1)$;
    \end{itemize}
\Ensure
    With probability $\ge 1-\mu$, output $d_i = \deg_{x_i} \gcd(A,B)$ for $i=1,\dots,n$.

\State $P_A \gets \operatorname{pp}(A)$, $P_B \gets \operatorname{pp}(B)$.

\State  Compute $d=\max_{i=1}^n\max\{\deg_{x_i}A,\deg_{x_i}B\}$ and $\Hi = \max\{\|P_A\|_\infty, \|P_B\|_\infty\}$ and $t=\max\{\|P_A\|_0,\|P_B\|_0\}$.

\State Initialize $d_i \gets \min\{\deg_{x_i} P_A, \deg_{x_i} P_B\}$ for $i=1,\dots,n$.
\State Set $\rho \gets \lceil \log_2\frac{1}{\mu} \rceil$.
\State Choose $\lambda$ as \[
\lambda \ge \max\left(21,8nd^2 + 8nd,\ \frac{9nd}{5}\ln(2d \Hi^2 t^2)\right).
\]

\For{$r = 1$ to $\rho$}
    \State Choose a prime $p$ uniformly from $[\lambda, 2\lambda]$.
    \State Choose $\mathbf{a} = (a_1,\dots,a_n)$ uniformly from $\mathbb{F}_p^n$.
    \For{$k = 1$ to $n$}
        \If{$\operatorname{lc}_{x_k}(P_A)(\mathbf{a}) \equiv 0 \pmod{p}$ or $\operatorname{lc}_{x_k}(P_B)(\mathbf{a}) \equiv 0 \pmod{p}$}
            \State \textbf{continue} to the next $r$
        \EndIf
        \State Compute
        \[
        A_k(y) \gets P_A(\mathbf{\check{a}}_{k \mapsto y} ), \qquad B_k(y) \gets P_B(\mathbf{\check{a}}_{k \mapsto y} )
        \]
        in $\mathbb{F}_p[y]$, where $\mathbf{\check{a}}_{k \mapsto y}  = (a_1,\dots,a_{k-1}, y, a_{k+1},\dots,a_n)$.
        \State $h_k(y) \gets \gcd(A_k(y), B_k(y))$ in $\mathbb{F}_p[y]$.
        \State $d_k \gets \min\{d_k, \deg h_k\}$.
    \EndFor
\EndFor

\State \Return $(d_1, \dots, d_n)$.
\end{algorithmic}
\end{algorithm}

\begin{lemma}\label{lem:int-degree-match}
Let $A, B \in \mathbb{Z}[x_1,\dots,x_n]$ be primitive, and let $Q = \gcd(A,B)$. 
For a prime $p$ and a vector $\mathbf{a} \in \mathbb{F}_p^n$, the following hold:

\begin{enumerate}
    \item If $\operatorname{lc}_{x_k}(A)(\mathbf{a}) \not\equiv 0 \pmod{p}$ and $\operatorname{lc}_{x_k}(B)(\mathbf{a}) \not\equiv 0 \pmod{p}$, then
    \[
    \deg \gcd(A(\mathbf{\check{a}}_{k \mapsto y} ), B(\mathbf{\check{a}}_{k \mapsto y} )) \ge \deg_{x_k} Q.
    \]
    \item If, in addition, $\operatorname{Res}_{x_k}(A/Q, B/Q)(\mathbf{a}) \not\equiv 0 \pmod{p}$, then
    \[
    \deg \gcd(A(\mathbf{\check{a}}_{k \mapsto y} ), B(\mathbf{\check{a}}_{k \mapsto y} )) = \deg_{x_k} Q.
    \]
\end{enumerate}
\end{lemma}

\begin{proof}
Write $A = Q \cdot A_1$ and $B = Q \cdot B_1$ with $\gcd(A_1, B_1) = 1$. 
Under the evaluation $x_k \mapsto y$ and $x_i \mapsto a_i$ for $i \ne k$, we have
\[
A(\mathbf{\check{a}}_{k \mapsto y} ) = Q(\mathbf{\check{a}}_{k \mapsto y} ) \cdot A_{1}(\mathbf{\check{a}}_{k \mapsto y} ), \qquad
B(\mathbf{\check{a}}_{k \mapsto y} ) = Q(\mathbf{\check{a}}_{k \mapsto y} ) \cdot B_{1p}(\mathbf{\check{a}}_{k \mapsto y} )
\]
in $\mathbb{F}_p[y]$. Hence
\[
\gcd(A(\mathbf{\check{a}}_{k \mapsto y} ), B(\mathbf{\check{a}}_{k \mapsto y} )) 
= c\cdot Q(\mathbf{\check{a}}_{k \mapsto y} ) \cdot \gcd(A_{1}(\mathbf{\check{a}}_{k \mapsto y} ), B_{1}(\mathbf{\check{a}}_{k \mapsto y} )),
\]
where the equality holds up to a nonzero constant factor $c$ in $\mathbb{F}_p$.

For the first claim, since $\gcd(A_{1}(\mathbf{\check{a}}_{k \mapsto y} ), B_{1}(\mathbf{\check{a}}_{k \mapsto y} ))$ is a polynomial in $y$ (possibly constant), we have
\[
\deg \gcd(A(\mathbf{\check{a}}_{k \mapsto y} ), B(\mathbf{\check{a}}_{k \mapsto y} ))
= \deg Q(\mathbf{\check{a}}_{k \mapsto y} ) + \deg \gcd(A_{1}(\mathbf{\check{a}}_{k \mapsto y} ), B_{1}(\mathbf{\check{a}}_{k \mapsto y} ))
\ge \deg Q(\mathbf{\check{a}}_{k \mapsto y} ).
\]
It remains to note that $\deg Q(\mathbf{\check{a}}_{k \mapsto y} ) = \deg_{x_k} Q$ whenever the leading coefficient of $Q$ with respect to $x_k$ does not vanish under the evaluation. Since $Q$ divides both $A$ and $B$, the leading coefficient of $Q$ divides both $\operatorname{lc}_{x_k}(A)$ and $\operatorname{lc}_{x_k}(B)$. Thus $\operatorname{lc}_{x_k}(A)(\mathbf{a}) \not\equiv 0 \pmod{p}$ and $\operatorname{lc}_{x_k}(B)(\mathbf{a}) \not\equiv 0 \pmod{p}$ imply $\operatorname{lc}_{x_k}(Q)(\mathbf{a}) \not\equiv 0 \pmod{p}$. Therefore $\deg Q(\mathbf{\check{a}}_{k \mapsto y} ) = \deg_{x_k} Q$, and the first claim follows.

For the second claim, the resultant $\operatorname{Res}_{x_k}(A_1, B_1)$ is nonzero in $\mathbb{Z}[x_1,\dots,x_n]$ because $\gcd(A_1, B_1) = 1$. If its evaluation at $\mathbf{a}$ is nonzero modulo $p$, then $A_{1}(\mathbf{\check{a}}_{k \mapsto y} )$ and $B_{1}(\mathbf{\check{a}}_{k \mapsto y} )$ are coprime in $\mathbb{F}_p[y]$. Hence
\[
\gcd(A_{1}(\mathbf{\check{a}}_{k \mapsto y} ), B_{1}(\mathbf{\check{a}}_{k \mapsto y} )) = 1.
\]
Therefore
\[
\deg \gcd(A(\mathbf{\check{a}}_{k \mapsto y} ), B(\mathbf{\check{a}}_{k \mapsto y} )) = \deg Q(\mathbf{\check{a}}_{k \mapsto y} ) = \deg_{x_k} Q.
\]
This proves the second claim.
\end{proof}

\begin{theorem}[Correctness of Algorithm~\ref{alg:int-degree-match}]\label{thm:int-degree-match}
Let $A, B \in \mathbb{Z}[x_1,\dots,x_n]$, and let $Q = \gcd(A,B)$. 
Algorithm~\ref{alg:int-degree-match} computes
\[
d_i = \deg_{x_i} Q, \qquad i=1,\dots,n,
\]
with probability at least $1-\mu$.
\end{theorem}

\begin{proof}
The algorithm first computes the primitive parts $P_A = \operatorname{pp}(A)$ and $P_B = \operatorname{pp}(B)$.
Since $\gcd(P_A, P_B) = \operatorname{pp}(\gcd(A,B)) = \operatorname{pp}(Q)$, and taking primitive part does not change the partial degrees, it suffices to prove the theorem under the assumption that $A$ and $B$ are already primitive. We assume this throughout.

We analyze the success probability of a single trial. A trial consists of choosing a prime $p \in [\lambda, 2\lambda]$ and a vector $\mathbf{a} \in \mathbb{F}_p^n$, then computing the univariate GCD degrees for all $k=1,\dots,n$.

Let $Q = \gcd(A,B)$ and write $A = Q \cdot A_1$, $B = Q \cdot B_1$ with $\gcd(A_1, B_1) = 1$. For each $k=1,\dots,n$, define
\[
\Gamma_k(\mathbf{x}) = \operatorname{lc}_{x_k}(A) \cdot \operatorname{lc}_{x_k}(B) \cdot \res_{x_k}(A_1, B_1) \in \mathbb{Z}[x_1,\dots,x_n],
\]
where $\res_{x_k}$ denotes the resultant with respect to $x_k$. Since $A_1$ and $B_1$ are coprime, each $\Gamma_k$ is a nonzero polynomial. Let
\[
\Gamma(\mathbf{x}) := \prod_{k=1}^n \Gamma_k(\mathbf{x}).
\]
Then $\Gamma$ is also nonzero.
By the standard resultant degree bound,
\[
\deg \Gamma_k \le 2d^2 + 2d,
\]
and hence
\[
\deg \Gamma \le \sum_{k=1}^n \deg \Gamma_k \le 2nd^2 + 2nd.
\]

If a chosen prime $p$ is lucky and $\Gamma(\mathbf{a}) \not\equiv 0 \pmod{p}$, then by Lemma~\ref{lem:int-degree-match}, for every $k=1,\dots,n$,
\[
\deg \gcd(A(\mathbf{\check{a}}_{k \mapsto y} ), B(\mathbf{\check{a}}_{k \mapsto y} )) = \deg_{x_k} Q.
\]
Thus, the computed degree is correct for every variable simultaneously.

We now bound the probability that both conditions hold for a single trial.

First, by Lemma~\ref{lem:lucky-prime-any-eps} with $\varepsilon = 1/3$, the probability that a randomly chosen prime $p \in [\lambda, 2\lambda]$ is lucky is at least $2/3$.

Second, conditioned on $p$ being lucky, we apply the Schwartz--Zippel lemma to $\Gamma$:
\[
\Pr_{\mathbf{a} \in \mathbb{F}_p^n}\bigl(\Gamma(\mathbf{a}) \equiv 0 \pmod{p}\bigr)
\le \frac{\deg \Gamma}{p}
\le \frac{2nd^2 + 2nd}{\lambda}.
\]
With our choice of $\lambda \ge 8nd^2 + 8nd$, this probability is at most $1/4$. Hence, conditioned on $p$ being lucky, a random evaluation is good with probability at least $3/4$.

Therefore, a single trial succeeds with probability at least
\[
\Pr(p \text{ is lucky}) \cdot \Pr(\text{evaluation is good} \mid p \text{ is lucky})
\ge \frac{2}{3} \cdot \frac{3}{4}
= \frac{1}{2}.
\]
Equivalently, a single trial fails with probability at most $1/2$.

Since the algorithm performs $\rho = \lceil \log_2\frac{1}{\mu} \rceil$ independent trials, the probability that all trials fail is at most
\[
\left(\frac{1}{2}\right)^\rho \le \mu.
\]
Thus, with probability at least $1-\mu$, at least one trial is good. Taking the minimum over all trials then yields the true partial degrees $\deg_{x_i} Q$ for all $i=1,\dots,n$, because a bad trial can only overestimate---never underestimate---the true degree by Lemma~\ref{lem:int-degree-match}.
\end{proof}

\begin{remarkstar}
Since a bad evaluation can only introduce spurious common factors, the degree of the univariate GCD can only be \emph{larger} than the true partial degree. Thus the minimum over independent trials is always correct whenever at least one trial is good.
\end{remarkstar}

\begin{theorem}[Complexity of Algorithm~\ref{alg:int-degree-match}]\label{thm:int-degree-match-complexity}
Let $A, B \in \mathbb{Z}[x_1,\dots,x_n]$ with partial degree at most $d$, coefficient size at most $\Hi$, and let $t = \max(\|A\|_0, \|B\|_0)$. Algorithm~\ref{alg:int-degree-match} runs in
\[
\widetilde{O}\bigl( n t \log^2 d \cdot \log \Hi \log\frac{1}{\mu} + n d \log \Hi \log\frac{1}{\mu} \bigr)
\]
bit operations.
\end{theorem}

\begin{proof}
We analyze each step of the algorithm.

\paragraph{Content removal.}
Computing the primitive parts $P_A$ and $P_B$ requires taking the GCD of all coefficients of $A$ and $B$, respectively. This costs $\widetilde{O}(t \log \Hi)$ bit operations, which is dominated by the main complexity bound.

\paragraph{Prime selection.}
Choosing a prime $p$ uniformly from $[\lambda, 2\lambda]$ requires testing $O(\log \lambda)$ candidates for primality. Since
\[
\lambda = O\bigl(nd^2 + nd\log(d\Hi t)\bigr),
\]
each primality test costs $\operatorname{poly}(\log \lambda)$ bit operations. This step is dominated by the main complexity bound.

\paragraph{Evaluation of polynomials at $\mathbf{\check{a}}_{k \mapsto y} $.}
For a single term $c x_1^{e_1}\cdots x_n^{e_n}$, reducing its coefficient modulo $p$ costs $\widetilde{O}(\log p + \log \Hi)$ bit operations. The evaluation at $\mathbf{a} = (a_1,\dots,a_n)$ is
\[
C = c a_1^{e_1}\cdots a_n^{e_n} \in \mathbb{F}_p,
\]
which costs $\widetilde{O}(n \log d \log p)$ bit operations using fast exponentiation modulo $p$. For each $k=1,\dots,n$, the evaluation at $\mathbf{\check{a}}_{k \mapsto y}  = (a_1,\dots,a_{k-1}, y, a_{k+1},\dots,a_n)$ is obtained as
\[
C \cdot \frac{y^{e_k}}{a_k^{e_k}},
\]
which requires $\widetilde{O}(\log d \log p)$ additional operations. Thus, evaluating all $n$ univariate images of a single term costs $\widetilde{O}(n \log d \log p)$ bit operations. Summing over all terms of $A$ and $B$ gives
\[
O\bigl( t\log p + t\log \Hi + n t \log d \log p \bigr)
\]
bit operations per trial.

\paragraph{Univariate GCD computation.}
For each $k=1,\dots,n$, we compute the GCD of two univariate polynomials in $\mathbb{F}_p[y]$ of degree at most $d$. Using the fast Euclidean algorithm, this costs $\widetilde{O}(d \log p)$ bit operations per GCD. Since there are $n$ such GCDs per trial, the cost is
\[
\widetilde{O}(n d \log p)
\]
bit operations per trial.

\paragraph{Total per trial.}
Combining the evaluation and GCD costs, one trial requires
\[
\widetilde{O}\bigl( n t \log d \log p + t\log \Hi + n d \log p \bigr)
\]
bit operations. Since $\log p = O(\log \lambda) = O(\log(nd) + \log\log \Hi + \log\log t)$ and $t \le (d+1)^n$, this simplifies to
\[
\widetilde{O}\bigl( n t \log^2 d \cdot \log\log \Hi + n d \log\log \Hi +t\log\Hi\bigr).
\]

\paragraph{Repeating $\rho$ times.}
The algorithm repeats the trial $\rho = \lceil \log_2\frac{1}{\mu} \rceil$ times. Hence the total cost is
\[
\widetilde{O}\bigl( n t \log^2 d \cdot \log\log \Hi \log\frac{1}{\mu} + n d \log\log \Hi \log\frac{1}{\mu} +t\log\Hi\log\frac{1}{\mu}\bigr).
\]

Finally, since $\log\log \Hi \le \log \Hi$ for $\Hi \ge 2$, we may replace $\log\log \Hi$ by $\log \Hi$ in the $\widetilde{O}$ notation. This also absorbs the cost of prime selection and coefficient reduction. Therefore, the total complexity is
\[
\widetilde{O}\bigl( n t \log^2 d \cdot \log \Hi \log\frac{1}{\mu} + n d \log \Hi \log\frac{1}{\mu} \bigr).
\]
\end{proof}

\subsection{Rational Reconstruction}

A key ingredient in our polynomial GCD algorithms over the integers is the recovery of coefficients from their residues modulo several primes. Given a modular GCD $\mathcal{G}_p = \gcd(A \bmod p, B \bmod p)$ over $\mathbb{F}_p$, its coefficients are elements of $\mathbb{F}_p$. To reconstruct the corresponding integer coefficients of $G = \gcd(A,B)$, we need to lift these modular residues back to $\mathbb{Z}$. This could be accomplished via the Chinese Remainder Theorem if the modular GCDs were simply the reductions of $G$ modulo $p$. However, by definition, the GCD over a field is monic, so $\mathcal{G}_p = \operatorname{monic}(G \bmod p)$; that is,
\[
\mathcal{G}_p = (G / \operatorname{lc}(G)) \bmod p.
\]
Thus the modular coefficients correspond to the rational coefficients of the monic polynomial $G / \operatorname{lc}(G)$, not to the integer coefficients of $G$ directly. Hence the Chinese Remainder Theorem alone does not suffice.

This difficulty is overcome by \emph{rational reconstruction}. Let $G = \sum_{i=1}^t c_i M_i$ with $c_i \in \mathbb{Z}$, and assume without loss of generality that the leading coefficient is $c_1 = \ell = \operatorname{lc}(G)$. Write
\[
G / \ell = M_1 + c'_2 M_2 + \cdots + c'_t M_t,
\]
where $c'_i = c_i / \ell \in \mathbb{Q}$ for $i = 2, \dots, t$. Suppose $|c_i| \le \Ho$ for all $i$, and let $p_1, \dots, p_{\kappa}$ be distinct primes such that $m = \prod_{i=1}^{\kappa} p_i > 2\Ho^2$. For each monomial $M_i$, rational reconstruction recovers the reduced fraction $c'_i = c_i / \ell$ from its residues modulo the primes $p_j$.

The bound $m > 2\Ho^2$ guarantees uniqueness of this recovery. Indeed, suppose two reduced fractions $a/b$ and $c/d$ with $|a|, |c|, |b|, |d|\le \Ho$ are congruent modulo $m$, i.e.,
\[
\frac{a}{b} \equiv \frac{c}{d} \pmod{m}.
\]
Then $ad - bc \equiv 0 \pmod{m}$. Since $|ad - bc| \le 2\Ho^2 < m$, we must have $ad - bc = 0$, hence $a/b = c/d$. Thus the recovered fraction is unique whenever $m > 2\Ho^2$. This is why we collect enough primes so that their product exceeds $2\Ho^2$, a standard requirement in rational reconstruction for integer polynomials. The bit complexity of each rational reconstruction is $\widetilde{O}(\log m)$  \cite{WangPan2003,monagan2004maximal}.

\subsection{Polynomial GCD over the Integers with Known Term and Coefficient Bounds (Algorithm~\ref{alg:intgcd-crt})}
\label{sec:int-crt}

We now present the complete algorithm (Algorithm~\ref{alg:intgcd-crt}) for computing the GCD of polynomials in $\mathbb{Z}[x_1,\dots,x_n]$ using small primes and rational construction.

\begin{algorithm}
\caption{Integer GCD via Small Primes}
\label{alg:intgcd-crt}
\begin{algorithmic}[1]
\Require
    \begin{itemize}
        \item Polynomials $A, B \in \mathbb{Z}[x_1, \dots, x_n]$ with $G = \gcd(A,B)$;
        \item An upper bound $T\ge \|G\|_0$;
        \item An upper bound $\Ho\ge \|G\|_{\infty}$;
        \item A desired failure probability $\mu > 0$.
    \end{itemize}
\Ensure
    With probability $\ge 1-\mu$, output $G$; or ``Failure".

\State Compute contents: $c_A \gets \operatorname{cont}(A)$, $c_B \gets \operatorname{cont}(B)$.
\State $c_G \gets \gcd(c_A, c_B)$.
\State $P_A \gets \operatorname{pp}(A) = A / c_A$, $P_B \gets \operatorname{pp}(B) = B / c_B$.

\State Let $\Hi \gets \max\{\|A\|_\infty, \|B\|_\infty\}$.
\State Let $d \gets \max_{i=1}^n \max\{\deg_{x_i} A, \deg_{x_i} B\}$, $t \gets \max\{\|A\|_0, \|B\|_0\}$.

\State Compute the partial degrees $d_i := \deg_{x_i} \gcd(P_A, P_B)$ using Algorithm~\ref{alg:int-degree-match} with failure probability $\mu/2$.

\State Choose $\lambda$ such that
\[
\lambda \ge \max\left(
    21,\ 
    \frac{9nd}{5} \ln(2d \Hi^2 t^2)\cdot \ln(2\Ho^2+1), 8nd^2 + 8nd 
\right).
\]

\State Let $\kappa \gets \left\lceil \ln(2\Ho^2) / \ln \lambda \right\rceil$.
\State Let $\varepsilon \gets \mu / (2\kappa)$.

\State Initialize an empty list $\mathcal{L}$ of lucky primes.

\algstore{intgcd-crt}
\end{algorithmic}
\end{algorithm}

\begin{algorithm}[p]
\begin{algorithmic}[1]
\algrestore{intgcd-crt}

\Statex \textbf{Collect $\kappa$ lucky primes into $\mathcal{L}$:}

\While{$|\mathcal{L}| < \kappa$}
    \State Choose a prime $p$ uniformly from $[\lambda, 4\lambda]$, distinct from previously chosen primes.

    \Statex \textbf{Lucky prime test:}
    \State Choose a random vectors $\mathbf{a} \in \mathbb{F}_p^n$.

        \For{$k = 1$ to $n$}
            \If{$\operatorname{lc}_{x_k}(P_A)(\mathbf{a}) \equiv 0 \pmod{p}$ or $\operatorname{lc}_{x_k}(P_B)(\mathbf{a}) \equiv 0 \pmod{p}$}
                \State Stop and \textbf{continue} to next prime.
            \EndIf
            \State Compute
            \[
            d_{k}^{(p)} = \deg \gcd(P_A(\mathbf{\check{a}}_{k \mapsto y} ), P_B(\mathbf{\check{a}}_{k \mapsto y} )).
            \]
    
    \EndFor

    \If{$d_k^{(p)} < d_k$ for some $k=1,\dots,n$}
        \State \Return ``Failure''
    \ElsIf{$d_k^{(p)} = d_k$ for all $k=1,\dots,n$}
        \State Add $p$ to $\mathcal{L}$. \Comment{$p$ is lucky.}
    \Else
        \State \textbf{continue} \Comment{$p$ is not lucky; try another prime.}
    \EndIf
\EndWhile

\Statex \textbf{Compute modular GCDs for each lucky prime:}

\For{each $p \in \mathcal{L}$}
    \State Compute $\mathcal{G}_p \gets \gcd(P_A \bmod p, P_B \bmod p)$ over $\mathbb{F}_p$ using Algorithm~\ref{alg:recursive-amplified} with failure probability $\varepsilon$.
    \If{$\mathcal{G}_p =$ ``Failure''}
        \State \Return ``Failure''
    \EndIf
    
    \State Store $\mathcal{G}_{p}$ for rational reconstruction.
\EndFor

\Statex \textbf{Recover integer GCD via rational reconstruction:}

\State Let $m = \prod_{p \in \mathcal{L}} p$, with $m > 2 \Ho^2$.

\State Suppose the monomials appearing in the modular GCDs are $M_1, \dots, M_t$.

\For{$i = 1$ to $t$}
    \State Collect the coefficients $u_{i,p} \in \mathbb{F}_p$ of $M_i$ in $\mathcal{G}_p$ for each $p \in \mathcal{L}$.
    \State Use rational reconstruction to recover the unique rational number
    \[
    \frac{c_i}{\ell_i} \in \mathbb{Q}
    \]
    such that  $\gcd(c_i, \ell_i) = 1$, $|c_i| \le \Ho$, $1 \le \ell_i \le \Ho$, and
    \[
    \frac{c_i}{\ell_i} \equiv u_{i,p} \pmod{p} \quad \text{for all } p \in \mathcal{L}.
    \]
\EndFor

\State Compute $\ell = \operatorname{lcm}(\ell_1, \dots, \ell_t)$.
\State For each $i = 1,\dots,t$, set $c_i' \gets \ell \cdot (c_i / \ell_i) \in \mathbb{Z}$.
\State Let $Q \in \mathbb{Z}[x_1,\dots,x_n]$ be the polynomial $\sum_{i=1}^t c_i' M_i$.
\State Let $P_G \gets \operatorname{pp}(Q)$.

\State Recover the full integer GCD: $G \gets c_G \cdot P_G$.
\State \Return $G$.

\end{algorithmic}
\end{algorithm}

\subsubsection*{Correctness and Complexity Analysis}

\begin{theorem}[Correctness of Algorithm~\ref{alg:intgcd-crt}]
With probability at least \(1-\mu\), Algorithm~\ref{alg:intgcd-crt} outputs the correct integer GCD \(G\).
\end{theorem}

\begin{proof}
Let \(\mathcal{E}_1\) be the event that the partial degrees \(d_i = \deg_{x_i} \gcd(P_A, P_B)\) are computed correctly by Algorithm~\ref{alg:int-degree-match}. By Theorem~\ref{thm:int-degree-match}, \(\Pr(\mathcal{E}_1) \ge 1 - \mu/2\).

Conditioned on \(\mathcal{E}_1\), the lucky prime test is deterministic: a prime \(p\) is accepted if and only if
\[
\deg \gcd(P_A \bmod p, P_B \bmod p) = d_i \qquad \text{for all } i=1,\dots,n.
\]
Thus every prime in \(\mathcal{L}\) is guaranteed to be lucky. Since there are only finitely many unlucky primes and the interval \([\lambda, 4\lambda]\) contains sufficiently many primes, the loop terminates with probability \(1\).

Let \(\mathcal{E}_2\) be the event that all finite field GCD computations are correct. For each \(p \in \mathcal{L}\), Algorithm~\ref{alg:recursive-amplified} succeeds with probability at least \(1 - \mu/(2\kappa)\). By the union bound,
\[
\Pr(\mathcal{E}_2 \mid \mathcal{E}_1) \ge 1 - \frac{\mu}{2}.
\]

Conditioned on \(\mathcal{E}_1 \cap \mathcal{E}_2\), every \(p \in \mathcal{L}\) is lucky and every \(\mathcal{G}_p\) is the correct modular GCD. For each \(p\), the monic GCD satisfies
\[
\mathcal{G}_p \equiv \frac{P_G}{\operatorname{lc}(P_G)} \pmod{p},
\]
where \(P_G = \operatorname{pp}(G)\).

Since the coefficients of \(P_G\) are bounded in absolute value by \(\Ho\), with \(1 \le \operatorname{lc}(P_G) \le \Ho\), and since \(\prod_{p \in \mathcal{L}} p > 2\Ho^2\), rational reconstruction uniquely recovers the reduced fractions representing the coefficients of \(P_G / \operatorname{lc}(P_G)\) (see \cite{WangPan2003,monagan2004maximal}). Thus the polynomial \(P_G / \operatorname{lc}(P_G)\) is uniquely recovered.

Taking the least common multiple \(\ell\) of the denominators of these fractions, and multiplying each reconstructed coefficient by \(\ell\), recovers the integer coefficients of \(P_G\) up to a common integer factor. The primitive part step removes this factor, yielding \(P_G = \operatorname{pp}(Q)\). Finally, multiplying by \(c_G = \gcd(\operatorname{cont}(A), \operatorname{cont}(B))\) recovers the full integer GCD \(G = c_G \cdot P_G\).

Therefore,
\[
\Pr(\text{success}) \ge \Pr(\mathcal{E}_1 \cap \mathcal{E}_2)
= \Pr(\mathcal{E}_1) \cdot \Pr(\mathcal{E}_2 \mid \mathcal{E}_1)
\ge \left(1 - \frac{\mu}{2}\right)^2
\ge 1 - \mu.
\]
\end{proof}

\begin{theorem}[Complexity of Algorithm~\ref{alg:intgcd-crt}]
The expected bit complexity of Algorithm~\ref{alg:intgcd-crt} is
\[
\widetilde{O}\Bigl(
nTD\log\Ho\log\Hi\log\frac{1}{\mu}
+ n (\|A\|_0+\|B\|_0) \log^2 T \log^3 D \log \Ho \log \Hi\log\frac{1}{\mu}
\Bigr).
\]
More cleanly, the complexity is asymptotically linear in \(n,\|A\|_0,\|B\|_0, T, D, \log\Ho, \log\Hi, \log\frac{1}{\mu}\), respectively.
\end{theorem}

\begin{proof}
We analyze each step:

\paragraph{Content computation.}
Computing the contents and primitive parts of \(A\) and \(B\) requires taking GCDs of all coefficients, which costs \(\widetilde{O}(t \log \Hi)\) bit operations.

\paragraph{Partial degree computation.}
By Theorem~\ref{thm:int-degree-match-complexity}, Algorithm~\ref{alg:int-degree-match} computes the partial degrees \(d_i\) with cost
\[
\widetilde{O}\bigl( n t \log^2 d \cdot \log \Hi \log\frac{1}{\mu} + n d \log \Hi \log\frac{1}{\mu} \bigr)
\]
bit operations. This step is performed once.

\paragraph{Prime selection.}
With our choice of \(\lambda\),
\[
\lambda \ge \max\left(
    21,\ 
    \frac{9nd}{5} \ln(2d \Hi^2 t^2)\cdot \ln(2\Ho^2+1),\ 
    8nd^2 + 8nd
\right),
\]
we have \(\log \lambda = \widetilde{O}(\log(nd) + \log\log t + \log\log \Hi + \log\log \Ho)\). For a prime \(p\), the lucky prime test succeeds with probability at least \(1/2\): the prime is lucky with probability at least \(2/3\) by Theorem \ref{thm:lucky-primes-crt}, and conditioned on luckiness, the evaluation succeeds with probability at least \(3/4\) as $\lambda\ge 8nd^2 + 8nd$. Thus the expected number of trials to find one lucky prime is \(O(1)\). Since \(\kappa = O(\log \Ho)\) primes are needed, the expected cost of prime selection is
\[
\widetilde{O}\bigl(\kappa \cdot \operatorname{poly}(\log \lambda)\bigr)
= \widetilde{O}\bigl(\log \Ho \cdot \operatorname{poly}(\log(nd) \cdot \log\log t \cdot \log\log \Hi)\bigr).
\]

The  complexity that of checking is $\widetilde{O}\bigl( n t \log^2 d \cdot \log \Hi \log\Ho + n d \log \Hi \log\Ho \bigr)$

\paragraph{Finite field GCD per prime.}
For each \(p \in \mathcal{L}\), reducing \(P_A\) and \(P_B\) modulo \(p\) costs \(\widetilde{O}(t \log \Hi)\) bit operations. Computing the GCD over \(\mathbb{F}_p\) using Algorithm~\ref{alg:recursive-amplified} costs
\[
\widetilde{O}\bigl( nTD\log p\log(\kappa/\mu) + n t \log^2 T \log^2 D\log p\log(\kappa/\mu) \bigr)
\]
bit operations per prime. Summing over \(\kappa = O(\log \Ho)\) primes gives
\[
\widetilde{O}\bigl( nTD \log\log\Hi\log \Ho\log\frac{1}{\mu} + n t \log^2 T \log^3 D \log\log\Hi\log \Ho\log\frac{1}{\mu} + t \log \Hi \log \Ho \bigr).
\]

\paragraph{Rational reconstruction.}
For each of the \(O(T)\) monomials, rational reconstruction from \(\kappa = O(\log \Ho)\) residues costs \(\widetilde{O}(\log \Ho)\) bit operations (see \cite{WangPan2003,monagan2004maximal}). The total cost is \(\widetilde{O}(T \log \Ho)\).

\paragraph{Content removal and final multiplication.}
Removing the content of the reconstructed polynomial and multiplying by \(c_G\) costs \(\widetilde{O}(T \log \Ho)\) bit operations, which is dominated by the rational reconstruction cost.

Combining all steps, with $\log\log\Hi = O(\log\Hi)$ and $t = O(\|A\|_0+\|B\|_0)$, the total expected bit complexity is
\[
\widetilde{O}\Bigl(
nTD\log\Ho\log\Hi\log\frac{1}{\mu}
+ n(\|A\|_0+\|B\|_0) \log^2 T \log^3 D \log \Ho \log \Hi\log\frac{1}{\mu}
\Bigr).
\]
\end{proof}

\begin{remark}
If we assume \(T = \max\{\|A\|_0, \|B\|_0, \|G\|_0\}\), then the term involving \(t\) is dominated by \(T\), and the complexity simplifies to
\[
\widetilde{O}\bigl( n T D \log \Ho \log \Hi \log\frac{1}{\mu} \bigr)
\]
bit operations.
\end{remark}

\subsection{Sparse Coefficient Bounds}
\label{sec:coeff-bounds}

In this section, we establish coefficient bounds for factors of sparse multivariate polynomials over the integers. These bounds provide a rigorous criterion for determining how many primes are needed to reconstruct the integer GCD from its modular images via the Rational reconstruction. 

Our ultimate goal is to remove the assumption that a coefficient bound $\Ho \ge \|G\|_\infty$ is known in advance. To this end, we derive a sparse factor bound that depends on the number of terms of the factor and, importantly, is independent of the number of variables $n$. Together with the classical Gelfond bound, this provides a stopping criterion for the guessing strategy of $\Ho$: the algorithm terminates once the current guess exceeds the smaller of these two bounds, thereby requiring no prior knowledge of the coefficient size of the GCD.

\subsection*{Classical Bounds: Gelfond and Mignotte}

We begin by recalling two classical bounds on the coefficients of factors of polynomials. The first, due to Gelfond, provides a general upper bound for multivariate factors.

\begin{theorem}[Gelfond's Inequality \cite{gelfond2015}]
Let $F \in \mathbb{Z}[x_1,\dots,x_n]$ be a nonzero polynomial, and let $G \in \mathbb{Z}[x_1,\dots,x_n]$ be any factor of $F$. Let $d_i = \deg_{x_i}(F)$ for $i = 1,\dots,n$. Then
\[
\|G\|_\infty \le e^{d_1 + d_2 + \cdots + d_n} \|F\|_\infty.
\]
\end{theorem}

Gelfond's inequality provides a general upper bound that depends on the sum of the partial degrees of $F$. In particular, it implies that the coefficients of any factor of $F$ are bounded by $\|F\|_\infty$ times an exponential factor that grows with the number of variables and their degrees. While this bound is useful for existential arguments, it does not exploit sparsity: it applies uniformly to all factors, irrespective of the number of terms.

The second classical bound, due to Mignotte, applies to univariate polynomials and is sharper in that setting.

\begin{lemma}[Mignotte's Bound for Univariate Factors]\label{lem:mignotte-univariate}
Let $F, G \in \mathbb{Z}[x]$ with $G \mid F$ in $\mathbb{Z}[x]$. Then
\[
\|G\|_\infty \le \|G\|_2 \le 2^{\deg G} \|F\|_2 \le 2^{\deg G} \|F\|_1.
\]
\end{lemma}

For multivariate polynomials, Mignotte's bound is typically extended via the standard Kronecker substitution $x_i \mapsto y^{d^{i-1}}$, yielding
\[
\|G\|_\infty \le 2^{O(d^n)} \|F\|_1,
\]
which depends exponentially on the number of variables $n$. This exponential dependence on $n$ is a significant limitation for high-dimensional sparse polynomial computations.

\subsection*{A Sparse Mignotte Bound for Multivariate Polynomial Factors}

We now establish a coefficient bound for factors of sparse multivariate polynomials that complements the classical Gelfond and Mignotte bounds by exploiting the sparsity of the factor. Unlike the classical bounds, our bound is sensitive to the sparsity of the factor itself: it depends only on the factor's degree and its number of terms, and is completely independent of the number of variables. 

Let $F, G, H \in \mathbb{Z}[x_1, \dots, x_n]$ satisfy $F = G \cdot H$. Let $T$ be an upper bound on the number of terms of $G$, and let $D$ be an upper bound on the total degree of $G$. Choose a random vector $\mathbf{s} = (s_1, \dots, s_n) \in [0, T-1]^n$ and consider the randomized Kronecker substitution
\[
\fs{F}(y) = \frac{F(x_1y^{s_1}, \dots, x_ny^{s_n})}{y^{k_F}},
\]
where $k_F$ is the lowest power of $y$ in $F(x_1y^{s_1}, \dots, x_ny^{s_n})$. Since the substitution is multiplicative, we have
\[
\fs{F} = \fs{G} \cdot \fs{H}.
\]

Assume $G = c_1 M_1 + \cdots + c_t M_t$ with $t \le T$, where each $M_i = x_1^{e_{i,1}} \cdots x_n^{e_{i,n}}$ is a monomial and $c_i \neq 0$. Under the substitution, each term becomes
\[
\fs{G} = c_1 M_1 y^{d_1} + \cdots + c_t M_t y^{d_t},
\]
where $d_i = e_{i,1} s_1 + \cdots + e_{i,n} s_n - k_G$ for $i = 1, \dots, t$.

For a fixed term index $k \in \{1,\dots,t\}$, consider the polynomial
\[
\Gamma_k(\mathbf{s}) = \prod_{i \neq k} (d_i - d_k).
\]
This polynomial is nonzero if and only if the $y$-exponent of the $k$-th term is distinct from all other $y$-exponents, i.e., the $k$-th term is non-colliding in $\fs{G}$. The degree of $\Gamma_k$ is at most $t-1 \le T-1$. 

By the Schwartz--Zippel lemma, if $\mathbf{s}$ is chosen uniformly from $[0, T-1]^n$, then
\[
\Pr(\Gamma_k(\mathbf{s}) = 0) \le \frac{\deg \Gamma_k}{T} \le \frac{T-1}{T} < 1.
\]
Since the failure probability is strictly less than $1$, there must exist at least one choice of $\mathbf{s} \in [0, T-1]^n$ such that $\Gamma_k(\mathbf{s}) \neq 0$. In fact, the probability of success is at least $1/T$, so such a good choice of $\mathbf{s}$ is guaranteed to exist.

Now, for such a good choice of $\mathbf{s}$, evaluate the $\x$-variables at $1$ to obtain the univariate polynomials
\[
f(y) := \fs{F}(1,\dots,1,y) = \frac{F(y^{s_1}, \dots, y^{s_n})}{y^{k_F}}, \qquad
g(y) := \fs{G}(1,\dots,1,y) = \frac{G(y^{s_1}, \dots, y^{s_n})}{y^{k_G}}.
\]
Since the separation transformation is multiplicative and the evaluation $\x \mapsto 1$ is a ring homomorphism, we have
\[
f(y) = g(y) \cdot h(y), \qquad h(y) := \fs{H}(1,\dots,1,y).
\]
By the construction of $\mathbf{s}$, the $k$-th term of $G$ is non-colliding in $\fs{G}(\x, y)$. After evaluating $\x \mapsto 1$, this term becomes $c_k y^{d_k}$ in $g(y)$, with no other terms sharing the same $y$-exponent. Hence $c_k$ is a coefficient of $g(y)$, and therefore
\[
|c_k| \le \|g(y)\|_\infty.
\]
Applying the univariate Mignotte bound (Lemma~\ref{lem:mignotte-univariate}) to the factorization $f = g \cdot h$, we obtain
\[
|c_k| \le \|g(y)\|_\infty \le 2^{\deg g} \|f\|_1.
\]
Since the substitution only permutes and shifts exponents, $\|f\|_1 \le \|F\|_1$ and $\deg g \le \deg(G) \cdot (T-1)$, it follows that
\[
|c_k| \le 2^{\deg(G) \cdot (T-1)} \|F\|_1.
\]

Since this holds for every term $c_k M_k$ of $G$, we obtain the following theorem.

\begin{theorem}[Sparse Mignotte Bound for Multivariate Factors]\label{thm:sparse-mignotte}
Let $F \in \mathbb{Z}[x_1, \dots, x_n]$ be a nonzero polynomial, and let $G \in \mathbb{Z}[x_1, \dots, x_n]$ be any factor of $F$. Let $T = \|G\|_0$ be the number of terms of $G$, and let $D = \deg G$ be its total degree. Then
\[
\|G\|_\infty \le 2^{D (T-1)} \|F\|_1.
\]
\end{theorem}

\begin{proof}
The proof follows from the above construction: for each term of $G$, we choose a randomized Kronecker substitution that isolates that term without collision, apply the univariate Mignotte bound, and take the maximum over all terms. The bound $2^{D(T-1)} \|F\|_1$ is uniform for all terms of $G$.
\end{proof}

\begin{remark}
Comparing Theorem~\ref{thm:sparse-mignotte} with the classical bounds:

\begin{itemize}
    \item Gelfond's inequality gives $\|G\|_\infty \le e^{d_1+\cdots+d_n}\|F\|_\infty$, which depends on the sum of partial degrees of $F$ and is independent of the sparsity of $G$.
    \item The standard multivariate Mignotte bound via Kronecker substitution gives $\|G\|_\infty \le 2^{O(d^n)}\|F\|_1$, which depends exponentially on the number of variables $n$.
    \item Our sparse Mignotte bound gives $\|G\|_\infty \le 2^{D(T-1)}\|F\|_1$, which is completely independent of $n$ and depends only on the total degree $D$ and the term count $T$ of $G$ itself.
\end{itemize}

This reveals a fundamental structural property of sparse polynomials: the coefficient size of a factor is controlled by its own degree and its own sparsity, not by the ambient dimension. In the extreme case where $T = 1$, i.e., $G$ is a single term, our bound gives $\|G\|_\infty \le \|F\|_1$, and this is tight when $\|F\|_0 = O(1)$. 
In contrast, Gelfond's inequality would give $\|G\|_\infty \le e^{d_1+\cdots+d_n}\|F\|_\infty$, which can be arbitrarily loose for high-dimensional inputs.
\end{remark}

\begin{corollary}\label{cor:gcd-coeff-bound}
Let $A, B \in \mathbb{Z}[x_1, \dots, x_n]$ and let $G = \gcd(A, B)$. Suppose $\|G\|_0 \le T$ and $\deg(G) \le D$. Then
\[
\|G\|_\infty \le 2^{D(T-1)} \min\{\|A\|_1,\|B\|_1\}.
\]
\end{corollary}

This corollary is used in our integer GCD algorithm to determine the number of primes required for the rational reconstruction of the coefficients of $G$.
In a recent breakthrough, Nahshon and Shpilka~\cite{NahshonShpilka2026} showed that if a cofactor $A/G$ or $B/G$ is sparse, then the height of $G$ can be bounded in terms of the sparsity of $A/G$ or $B/G$, potentially leading to sharper bounds when the cofactor is sparse.

\subsection{Divisibility Testing over Integers (Algorithm~\ref{alg:int-divisibility-test})}
\label{sec:int-divisibility}

Given polynomials $A, B, H \in \mathbb{Z}[x_1,\dots,x_n]$, we wish to verify whether $H = \gcd(A,B)$ with high probability, using a randomized reduction to univariate divisibility tests modulo a suitably chosen prime.

In Section~\ref{sec:verification}, we developed a probabilistic GCD verification algorithm over a field. In this section, we extend this approach to the integer setting. To remove the assumptions that a term bound $T$ and a coefficient bound $\Ho$ are known in advance, we employ a guessing strategy as in the field case: we guess $T$ and $\Ho$ via doubling, compute a candidate GCD, and verify it probabilistically. For this strategy to work over the integers, we need two independent verification subroutines that operate without prior knowledge of $T$ or $\Ho$:

\begin{enumerate}
    \item computing the partial degrees $\deg_{x_i} \gcd(A,B)$ for all $i$ (Section~\ref{sec:partial-deg}), and
    \item testing divisibility $H \mid F$ for primitive integer polynomials.
\end{enumerate}

The partial-degree computation has already been presented in Section~\ref{sec:partial-deg}. In this section, we focus on the second subroutine: testing whether a given primitive polynomial $H$ divides another primitive polynomial $F$ in $\mathbb{Z}[x_1,\dots,x_n]$.

\subsection*{Why Two Separate Subroutines?}

The verification algorithm is decomposed into two independent subroutines, each exploiting a different one-sided probabilistic guarantee. For partial-degree computation, random evaluations can only overestimate the true degree, so taking the minimum over trials yields the correct value as soon as one good trial occurs. For divisibility testing, if $H \mid F$ then divisibility persists under every modular evaluation, so a negative answer is always correct. A positive answer, however, is only correct with high probability, since $H \nmid F$ may become divisible after modular evaluation. Repeating the test with independent random choices drives the false-positive probability arbitrarily low.
This separation preserves the strongest possible error guarantee for each task and simplifies the analysis.

\subsection*{Divisibility Testing for Integer Polynomials}\label{subsec:int-divisibility}

\begin{theorem}\label{thm:int-divisibility}
Let $H, F \in \mathbb{Z}[x_1, \dots, x_n]$ be primitive polynomials with $\deg_{x_i} H \le \deg_{x_i} F$ for all $i=1,\dots,n$, and suppose $H \nmid F$. Let $d$ be a bound on the partial degree of $H$ and $F$, and let $\mathbb{H} \ge \max\{\|H\|_\infty, \|F\|_\infty\}$. Let $t = \max(\|H\|_0, \|F\|_0)$ and choose $\lambda$ such that
\[
\lambda \ge \max\left\{
    21,\ 
    \frac{9 n d}{5} \ln(2 d \mathbb{H}^2 t^2),\ 
    8  d^2 + 8  d
\right\}.
\]
If a prime $p$ is chosen uniformly from $[\lambda, 2\lambda]$ and a vector $\mathbf{a} = (a_1, \dots, a_n)$ is chosen uniformly from $\mathbb{F}_p^n$, then with probability at least $1/2$, there exists $k \in \{1,\dots,n\}$ such that
\[
H(\mathbf{\check{a}}_{k \mapsto y} ) \nmid F(\mathbf{\check{a}}_{k \mapsto y} ) \quad \text{in } \mathbb{F}_p[y],
\]
where $\mathbf{\check{a}}_{k \mapsto y}  := (a_1, \dots, a_{k-1}, y, a_{k+1}, \dots, a_n)$.
\end{theorem}

\begin{proof}
Let $Q = \gcd(H, F)$ in $\mathbb{Z}[x_1,\dots,x_n]$. Since $H$ is primitive and $H \nmid F$, there exists some $k_0 \in \{1,\dots,n\}$ such that
\[
\deg_{x_{k_0}} Q < \deg_{x_{k_0}} H.
\]
Indeed, if $\deg_{x_i} Q = \deg_{x_i} H$ for all $i=1,\dots,n$, then $Q$ and $H$ have the same degree in every variable. Since $Q$ divides $H$ and both are primitive, this would imply $Q = \pm H$, contradicting $H \nmid F$ (as $Q$ is a common divisor of $H$ and $F$). Hence at least one variable must have strict inequality.

Let $H_1 = H/Q$ and $F_1 = F/Q$. Then $\gcd(H_1, F_1) = 1$. Define
\[
\Gamma(\mathbf{x}) = \operatorname{lc}_{x_{k_0}}(H) \cdot \operatorname{lc}_{x_{k_0}}(F) \cdot \operatorname{Res}_{x_{k_0}}(H_1, F_1) \in \mathbb{Z}[x_1,\dots,x_n].
\]
Since $H_1$ and $F_1$ are coprime, $\Gamma$ is a nonzero polynomial. By the standard resultant degree bound,
\[
\deg \Gamma \le 2d^2 + 2d.
\]

If a chosen prime $p$ is lucky for $(H, F)$ and $\Gamma(\mathbf{a}) \not\equiv 0 \pmod{p}$, then the leading coefficients do not vanish and the resultant is nonzero modulo $p$. Consequently,
\[
\gcd(H(\mathbf{\check{a}}_{k_0 \mapsto y}), F(\mathbf{\check{a}}_{k_0 \mapsto y}))
= Q(\mathbf{\check{a}}_{k_0 \mapsto y}) \cdot c
\]
for some $c \in \mathbb{F}_p^*$, and hence
\[
\deg \gcd(H(\mathbf{\check{a}}_{k_0 \mapsto y}), F(\mathbf{\check{a}}_{k_0 \mapsto y}))
= \deg_{x_{k_0}} Q
< \deg_{x_{k_0}} H
= \deg H(\mathbf{\check{a}}_{k_0 \mapsto y}).
\]
Therefore $H(\mathbf{\check{a}}_{k_0 \mapsto y}) \nmid F(\mathbf{\check{a}}_{k_0 \mapsto y})$.

We now bound the probability that both conditions hold.

First, by Lemma~\ref{lem:lucky-prime-any-eps}, the probability that a randomly chosen prime $p \in [\lambda, 2\lambda]$ is lucky for $(H, F)$ is at least $2/3$.

Second, conditioned on $p$ being lucky, we apply the Schwartz--Zippel lemma to $\Gamma$:
\[
\Pr_{\mathbf{a} \in \mathbb{F}_p^n}\bigl(\Gamma(\mathbf{a}) \equiv 0 \pmod{p}\bigr)
\le \frac{\deg \Gamma}{p}
\le \frac{2d^2 + 2d}{\lambda}.
\]
With our choice of $\lambda \ge 8d^2 + 8d$, this probability is at most $1/4$. Hence, conditioned on $p$ being lucky, the evaluation is good with probability at least $3/4$.

Therefore, a single trial succeeds in detecting $H \nmid F$ with probability at least
\[
\Pr(p \text{ is lucky}) \cdot \Pr(\text{evaluation is good} \mid p \text{ is lucky})
\ge \frac{2}{3} \cdot \frac{3}{4}
= \frac{1}{2}.
\]

Thus, whenever $H \nmid F$, a single trial detects this fact with probability at least $1/2$.
\end{proof}

\begin{remark}
If $H \mid F$, then for every prime $p$ and every evaluation point $\mathbf{a}$, we have
\[
H(\mathbf{\check{a}}_{k \mapsto y} ) \mid F(\mathbf{\check{a}}_{k \mapsto y} ) \quad \text{in } \mathbb{F}_p[y]
\]
for all $k=1,\dots,n$. Hence the divisibility test never produces a false negative: it always returns ``true'' when $H \mid F$. This one-sided guarantee is the dual of the partial-degree computation, where bad evaluations can only overestimate the true degree.
\end{remark}

The following algorithm (Algorithm \ref{alg:int-divisibility-test}) tests whether a polynomial $H$ divides another polynomial $F$ in $\mathbb{Z}[x_1,\dots,x_n]$, using a randomized reduction to univariate divisibility tests modulo a random prime.

\begin{algorithm}
\caption{Probabilistic Divisibility Test over the Integers}
\label{alg:int-divisibility-test}
\begin{algorithmic}[1]
\Require
    \begin{itemize}
        \item $H, F \in \mathbb{Z}[x_1,\dots,x_n]$;
        \item A desired failure probability $\mu > 0$;
    \end{itemize}
\Ensure
    If $H \nmid F$, with probability $\ge 1-\mu$, return ``false''; if $H \mid F$, return ``true''.

\State Compute the contents $c_H \gets \operatorname{cont}(H)$ and $c_F \gets \operatorname{cont}(F)$. If $c_H \nmid c_F$ in $\mathbb{Z}$, return ``false''.

\State Let $P_H \gets \operatorname{pp}(H)$ and $P_F \gets \operatorname{pp}(F)$.

\State Compute $d$ on the partial degree bound of $H,F$ and $\mathbb{H} = \max\{\|H\|_\infty, \|F\|_\infty\}$.
\State Set $t = \max\{\|H\|_0, \|F\|_0\}$.
\State Set $\rho \gets \lceil \log_2\frac{1}{\mu} \rceil$.
\State Choose $\lambda$ as
\[
\lambda \ge \max\left\{
    21,\ 
    \frac{9 n d}{5} \ln(2 d \mathbb{H}^2 t^2),\ 
    8 d^2 + 8 d
\right\}.
\]

\For{$r = 1$ to $\rho$}
    \State Choose a prime $p$ uniformly from $[\lambda, 2\lambda]$.
    \State Choose $\mathbf{a} = (a_1,\dots,a_n)$ uniformly from $\mathbb{F}_p^n$.
    \For{$k = 1$ to $n$}
        \If{$\operatorname{lc}_{x_k}(P_H)(\mathbf{a}) \equiv 0 \pmod{p}$ or $\operatorname{lc}_{x_k}(P_F)(\mathbf{a}) \equiv 0 \pmod{p}$}
            \State \textbf{continue} to the next $r$
        \EndIf
        \State Compute
        \[
        u_k(y) \gets P_H(\mathbf{\check{a}}_{k \mapsto y} ), \qquad v_k(y) \gets P_F(\mathbf{\check{a}}_{k \mapsto y} )
        \]
        in $\mathbb{F}_p[y]$, where $\mathbf{\check{a}}_{k \mapsto y}  = (a_1,\dots,a_{k-1}, y, a_{k+1},\dots,a_n)$.
        \If{$u_k \nmid v_k$ in $\mathbb{F}_p[y]$}
            \State \Return ``false''
        \EndIf
    \EndFor
\EndFor

\State \Return ``true''
\end{algorithmic}
\end{algorithm}

\begin{theorem}\label{thm:int-divisibility-test-complexity}[Correctness of Algorithm~\ref{alg:int-divisibility-test}]
Let $H, F \in \mathbb{Z}[x_1,\dots,x_n]$. Algorithm~\ref{alg:int-divisibility-test} returns ``true'' if $H \mid F$ with probability $1$, and returns ``false'' if $H \nmid F$ with probability at least $1-\mu$.
\end{theorem}

\begin{proof}
If $H \mid F$, then for every prime $p$ and every evaluation $\mathbf{a}$, we have $H(\mathbf{a}_{x_k}) \mid F(\mathbf{a}_{x_k})$ in $\mathbb{F}_p[y]$. Hence the algorithm always returns ``true''.

If $H \nmid F$, by Theorem~\ref{thm:int-divisibility}, a single trial detects the failure with probability at least $1/2$. Repeating $\rho = \lceil \log_2\frac{1}{\mu} \rceil$ times ensures that the failure probability is at most $(1/2)^\rho \le \mu$.
\end{proof}

\begin{theorem}[Complexity of Algorithm~\ref{alg:int-divisibility-test}]
Let $H, F \in \mathbb{Z}[x_1,\dots,x_n]$ with partial degree at most $d$, coefficient size at most $\mathbb{H}$, and let $t = \max\{\|H\|_0, \|F\|_0\}$. Algorithm~\ref{alg:int-divisibility-test} runs in
\[
\widetilde{O}\bigl( n t \log^2 d \cdot \log\log \mathbb{H} \log\frac{1}{\mu} + n d \log\log \mathbb{H} \log\frac{1}{\mu}+(\|H\|_0\cdot \log\|H\|_{\infty}+ \|F\|_0\cdot \log \|F\|_{\infty})\log\frac{1}{\mu}\bigr)
\]
bit operations.
\end{theorem}

\begin{proof}
We analyze each step of the algorithm.

\paragraph{Content removal and primitivation.}
Computing the contents $c_H = \operatorname{cont}(H)$ and $c_F = \operatorname{cont}(F)$ requires extracting the GCD of all coefficients of $H$ and $F$, respectively, costing $\widetilde{O}(\|H\|_0\cdot \log\|H\|_{\infty}+ \|F\|_0\cdot \log \|F\|_{\infty})$ bit operations. Computing the primitive parts $P_H$ and $P_F$ requires dividing each coefficient by the corresponding content, also costing $\widetilde{O}(\|H\|_0\cdot \log\|H\|_{\infty}+ \|F\|_0\cdot \log \|F\|_{\infty})$  bit operations. The content divisibility check $c_H \mid c_F$ in $\mathbb{Z}$ costs $\widetilde{O}(\log \mathbb{H})$ bit operations. 

\paragraph{Prime selection.}
Choosing a prime $p$ uniformly from $[\lambda, 2\lambda]$ requires testing $O(\log \lambda)$ candidate integers for primality. Since
\[
\lambda = O\bigl(d^2 + nd\log(d\mathbb{H} t)\bigr),
\]
each primality test costs $O(\operatorname{poly}(\log \lambda))$ bit operations using standard algorithms. This step contributes only a lower-order term and is dominated by the main complexity bound.

\paragraph{Evaluation of polynomials at $\mathbf{\check{a}}_{k \mapsto y} $.}
For a single term $c x_1^{e_1}\cdots x_n^{e_n}$ with $c \in \mathbb{Z}$, reducing its coefficient modulo $p$ costs $O(\log p + \log |c|)$ bit operations. The evaluation at $\mathbf{a} = (a_1,\dots,a_n)$ is
\[
C = c \cdot a_1^{e_1}\cdots a_n^{e_n} \in \mathbb{F}_p,
\]
which costs $\widetilde{O}(n \log d \log p)$ bit operations using fast exponentiation modulo $p$. For each $k=1,\dots,n$, the evaluation at $\mathbf{\check{a}}_{k \mapsto y}  = (a_1,\dots,a_{k-1}, y, a_{k+1},\dots,a_n)$ is obtained as
\[
C \cdot \frac{y^{e_k}}{a_k^{e_k}},
\]
which requires $O(\log d \log p)$ additional operations. Thus, evaluating all $n$ univariate images of a single term costs $O(n \log d \log p)$ bit operations. Summing over all terms of $H$ and $F$ gives
\[
\widetilde{O}\bigl(\|H\|_0\cdot \log\|H\|_{\infty}+ \|F\|_0\cdot \log \|F\|_{\infty} + n t \log d \log p \bigr)
\]
bit operations per trial.

\paragraph{Univariate divisibility tests.}
For each $k=1,\dots,n$, we need to test whether $u_k(y) \mid v_k(y)$ in $\mathbb{F}_p[y]$, where $\deg u_k, \deg v_k \le d$. The total cost per trial is
\[
\widetilde{O}(n d \log p)
\]
bit operations.

\paragraph{Total per trial.}
Combining the evaluation and divisibility test costs, one trial requires
\[
\widetilde{O}\bigl( n t \log d \log p + \|H\|_0\cdot \log\|H\|_{\infty}+ \|F\|_0\cdot \log \|F\|_{\infty}  + n d \log p \bigr)
\]
bit operations. Since $\log p = O(\log \lambda) = O(\log(nd) + \log d + \log\log \mathbb{H} + \log\log t)$ and $t \le (d+1)^n$, this simplifies to
\[
\widetilde{O}\bigl( n t \log^2 d \cdot \log\log \mathbb{H} + n d \log\log \mathbb{H} + \|H\|_0\cdot \log\|H\|_{\infty}+ \|F\|_0\cdot \log \|F\|_{\infty}\bigr).
\]

\paragraph{Repeating $\rho$ times.}
The algorithm repeats the trial $\rho = \lceil \log_2\frac{1}{\mu} \rceil$ times. Hence the total cost is
\[
\widetilde{O}\bigl( n t \log^2 d \cdot \log\log \mathbb{H} \log\frac{1}{\mu} + n d \log\log \mathbb{H} \log\frac{1}{\mu}+(\|H\|_0\cdot \log\|H\|_{\infty}+ \|F\|_0\cdot \log \|F\|_{\infty})\log\frac{1}{\mu} \bigr).
\]

\end{proof}

\subsection{Polynomial GCD over the Integers without Priori Bounds (Algorithms~\ref{alg:intgcd-guess})}
\label{sec:int-no-bounds}

In this section, we present a GCD algorithm for integer polynomials that removes the assumption of a priori knowledge of both the term bound $T$ and the coefficient bound $\Ho$. As in the field case, we employ a doubling strategy: we guess $T = 2^0, 2^1, 2^2, \dots$ until verification succeeds. 

To keep the analysis tractable and achieve quasi-linear complexity, we separate the removal of the two bounds. We first assume that a term bound $T \ge \|G\|_0$ is given (Algorithm \ref{alg:intgcd-no-bounds}), and show how to eliminate the coefficient bound $\Ho$ with only a logarithmic overhead in the expected complexity. The term bound $T$ itself is then handled by an outer guessing loop, which we analyze in the next section.

A naive approach to eliminating $\Ho$ would be to use the Gelfond bound $\mathbb{H}_{o\max} = e^{nd} \Hi$ as a worst-case coefficient bound, which is independent of the true GCD. However, this would require choosing primes on the order of $\log \mathbb{H}_{o\max} = O(nd + \log \Hi)$, and consequently $\kappa = O(\log \mathbb{H}_{o\max}) = O(nd + \log \Hi)$ primes. Since each finite-field GCD costs $\tilde{O}(nTD\log p)$, this would yield a complexity of $\tilde{O}(n^2 T D^2\log p)$, which loses the desired linearity in $n$ and $D$. 

Instead, we take an \emph{adaptive} approach: we start with a small guess $\H = 2$ for the coefficient bound of the GCD, square it in each iteration (i.e., $\H \gets \H^2$), and stop once the correct GCD is recovered and verified. This squaring strategy ensures that the number of iterations is only $O(\log\log \Ho)$, while the number of primes needed per iteration is $O(\log \Ho)$; their product remains $\widetilde{O}(\log \Ho)$, preserving the overall linear complexity in $\log \Ho$.

\subsection*{Algorithm Framework}

We now give a high-level description of the algorithm; the detailed pseudocode is presented in Algorithm~\ref{alg:intgcd-no-bounds}.

The algorithm takes as input primitive polynomials $A, B \in \mathbb{Z}[x_1,\dots,x_n]$, a guessed term bound $T$, and a failure tolerance $\mu$. It proceeds in three nested levels:

\begin{enumerate}
    \item \textbf{Partial degree computation.} The algorithm first computes the partial degrees
    \[
    d_i = \deg_{x_i} \gcd(A,B), \qquad i=1,\dots,n,
    \]
    using Algorithm~\ref{alg:int-degree-match}. These degrees serve as a reference for detecting lucky primes.

    \item \textbf{Coefficient bound guessing.} Starting with a guess $\H = 2$ for the true coefficient bound $\Ho = \|G\|_\infty$, the algorithm squares $\H$ in each iteration and attempts to reconstruct $G$ using rational reconstruction. For a fixed $\H$, it collects $\kappa = O(\log \H)$ lucky primes $p$ satisfying
    \[
    \deg_{x_i} \gcd(A \bmod p, B \bmod p) = d_i \quad \text{for all } i.
    \]
    For each lucky prime, it computes the modular GCD $\mathcal{G}_p = \gcd(A \bmod p, B \bmod p)$ over $\mathbb{F}_p$. Once $\kappa$ primes are collected, rational reconstruction recovers a candidate polynomial $\mathcal{G}$ with coefficients bounded by $\H$. If $\|\mathcal{G}\|_\infty > \H$, the guess $\H$ was insufficient; the algorithm squares $\H$ and restarts the collection.

    \item \textbf{Verification.} After a candidate $\mathcal{G}$ is reconstructed, the algorithm verifies whether $\mathcal{G} \mid A$ and $\mathcal{G} \mid B$ using the probabilistic divisibility test (Algorithm~\ref{alg:int-divisibility-test}). If the tests pass, $\mathcal{G}$ is returned as the GCD. If they fail, the algorithm squares $\H$ and continues. If $\H$ exceeds either of the theoretical upper bounds from Gelfond's inequality or the sparse Mignotte bound, the algorithm terminates with ``Failure''.
\end{enumerate}

The guessed term bound $T$ is handled by an outer doubling loop: if the algorithm detects that $\|\mathcal{G}\|_0 > T$ during rational reconstruction or verification, it reports ``Failure''.

Below is a flowchart of the framework.

\begin{figure}[htbp]
\centering
\resizebox{0.9\textwidth}{!}{%
\begin{tikzpicture}[
    node distance=0.6cm and 1.0cm,
    box/.style={rectangle, draw, thick, minimum width=2.8cm, minimum height=0.9cm, align=center, rounded corners=2pt, font=\small},
    bluebox/.style={box, fill=blue!12, draw=blue!60},
    greenbox/.style={box, fill=green!12, draw=green!60},
    redbox/.style={box, fill=red!12, draw=red!60},
    orangebox/.style={box, fill=orange!12, draw=orange!60},
    arrow/.style={->, thick, >=stealth},
    dashedarrow/.style={->, thick, >=stealth, dashed},
    label/.style={font=\small\bfseries}
]

\node[box, fill=gray!10] (input) {Input:\\$A, B \in \mathbb{Z}[x_1,\dots,x_n]$\\ primitive};

\node[bluebox, below=of input] (partial) {Partial Degree Computation\\$\deg_{x_i}\gcd(A,B)$};
\draw[arrow] (input) -- (partial);

\node[orangebox, below=of partial] (guess) {\textbf{Outer Loop}\\Guess coefficient bound $\H$\\$\H \gets 2,4,16,\dots$};
\draw[arrow] (partial) -- (guess);

\node[greenbox, below=of guess] (collect) {\textbf{Inner Loop}\\Collect $\kappa$ lucky primes};
\draw[arrow] (guess) -- (collect);

\node[box, below=of collect, xshift=-3.2cm, minimum width=2.6cm] (lucky) {Lucky Prime Test\\compare degrees};
\node[box, below=of collect, xshift=0cm, minimum width=2.6cm] (modgcd) {Finite Field GCD\\$\mathcal{G}_p$};
\node[box, below=of collect, xshift=3.2cm, minimum width=2.6cm] (divcheck) {Modular Divisibility\\$\mathcal{G}_p \mid A_p,B_p$};
\draw[arrow] (collect) -- (lucky);
\draw[arrow] (collect) -- (modgcd);
\draw[arrow] (collect) -- (divcheck);

\node[redbox, below=of collect, yshift=-2.8cm] (reconstruct) {Rational Reconstruction\\$\mathcal{G} \gets \operatorname{pp}(Q)$};
\draw[arrow] (lucky) --(reconstruct);
\draw[arrow] (modgcd) -- (reconstruct);
\draw[arrow] (divcheck) -- (reconstruct);

\node[box, below=of reconstruct, minimum width=3.5cm, fill=green!20, draw=green!60] (verify) {Verification\\$\mathcal{G} \mid A$ and $\mathcal{G} \mid B$?};
\draw[arrow] (reconstruct) -- (verify);

\node[box, below=of verify, xshift=-2.0cm, minimum width=2.2cm, fill=red!20, draw=red!60] (output) {Output $\mathcal{G}$};
\node[box, below=of verify, xshift=2.0cm, minimum width=2.2cm, fill=orange!20, draw=orange!60] (continue) {$\H \gets \H^2$\\continue};
\draw[arrow] (verify) -- node[left, font=\small] {pass} (output);
\draw[arrow] (verify) -- node[right, font=\small] {fail} (continue);
\draw[dashedarrow] (continue.east) .. controls +(1.5,0) and +(1.5,0) .. (collect.east);
\draw[dashedarrow] (output.south) -- ++(0,-0.5);

\node[label, right=of guess, xshift=1.5cm, align=left] {Coefficient bound\\guessing strategy};
\node[label, right=of collect, xshift=2.2cm, align=left] {Parallel per prime\\$O(\log \Ho)$ primes};
\node[label, right=of reconstruct, xshift=1.8cm, align=left] {Coefficient bound+ \\rational reconstruction};


\end{tikzpicture}
}
\caption{Algorithm pipeline for Algorithm \ref{alg:intgcd-no-bounds}}
\label{fig:pipeline}
\end{figure}
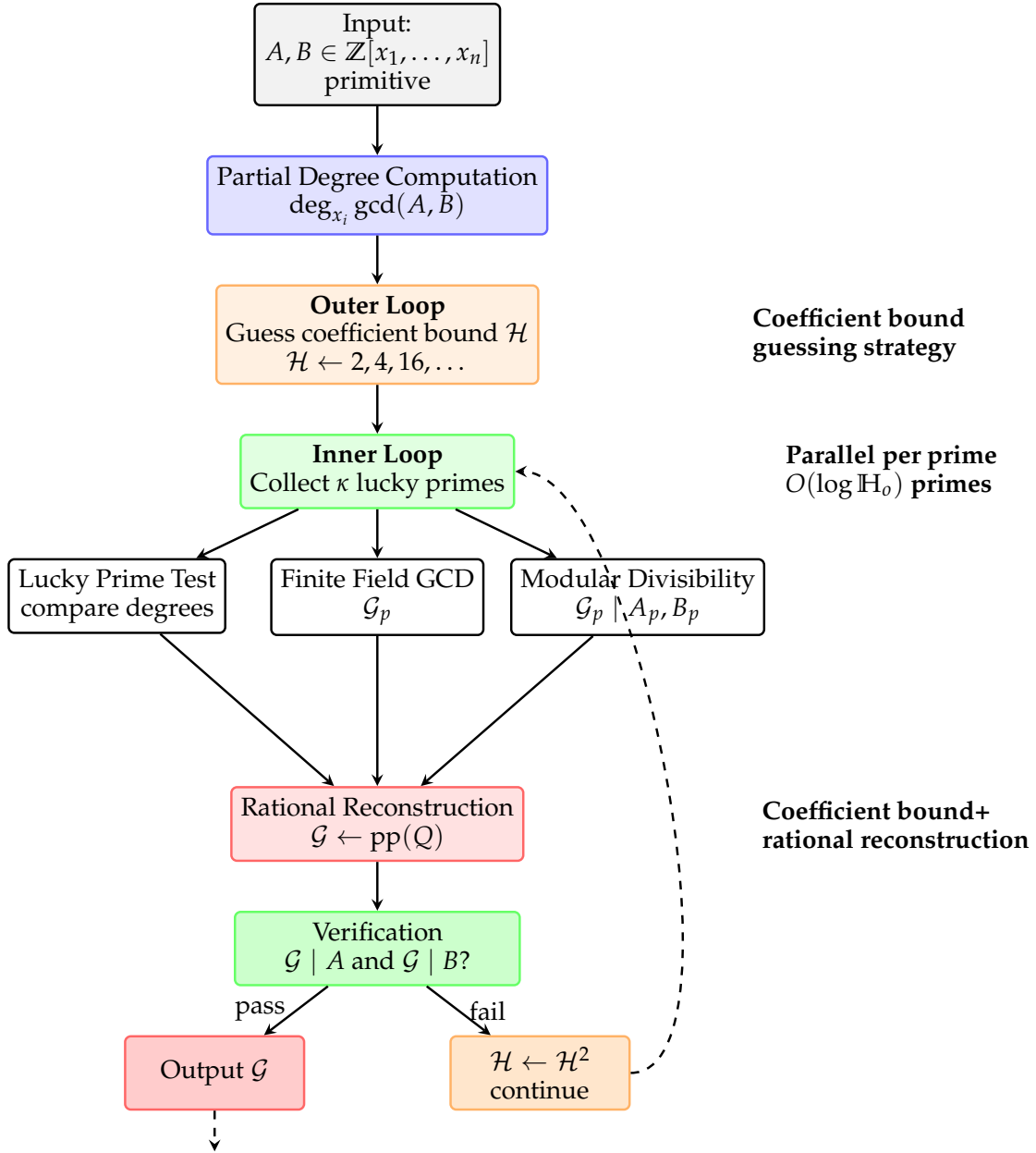

\begin{algorithm}
\caption{Polynomial GCD Algorithm over the Integers with a Given Term Bound}
\label{alg:intgcd-no-bounds}
\begin{algorithmic}[1]
\Require
    \begin{itemize}
        \item $A, B \in \mathbb{Z}[x_1,\dots,x_n]$ primitive;
        \item A guessed term bound $T$ for $\|G\|_0$, where $G = \gcd(A,B)$;
        \item A desired failure probability $\mu > 0$.
    \end{itemize}
\Ensure
    \begin{itemize}
        \item If $T \ge \|G\|_0$, returns $G = \gcd(A,B)$ with probability at least $1-\mu$;
        \item If $T < \|G\|_0$, returns ``Failure'' with probability at least $1-\mu$.
    \end{itemize}

\State Let $\Hi := \max\{\|A\|_\infty, \|B\|_\infty\}$.
\State Let $d := \max_{1 \le i \le n} \max\{\deg_{x_i} A, \deg_{x_i} B\}$, let $D$ be a bound on the total degree of $A$ and $B$, and let $t := \max\{\|A\|_0, \|B\|_0\}$.
\State Let $\mathbb{H}_{o\max} := \min\left\{ e^{nd}\min\{\|A\|_{\infty},\|B\|_{\infty}\},\; 2^{D(T-1)}\min\{\|A\|_1,\|B\|_1 \}\right\}$.

\State Choose $\lambda$ such that
\[
\lambda \ge \max\left\{
    21,\ 
    \frac{12nd}{5} \ln(2d \Hi^2 t^2) \cdot \ln(2e^{2nd}\Hi^2+1),\ 
    8nd^2 + 8nd
\right\}.
\]\label{alg12-4}

\State $L_{\max} := \left\lceil \log_2 \log_2 \mathbb{H}_{o\max} \right\rceil$.  \Comment{Maximum number of outer loop iterations.}
\State $\kappa_{\max} := \left\lceil \frac{\log_2(2\mathbb{H}_{o\max}^2 + 1)}{\log_2 \lambda} \right\rceil$.  \Comment{Maximum number of primes.}

\State Set $\varepsilon_1 \gets \mu/3$, $\varepsilon_2 \gets \mu/(3 \kappa^2_{\max})$, $\varepsilon_3 \gets \mu/(6 L_{\max})$.

\State Compute the partial degrees $d_i = \deg_{x_i} \gcd(A,B)$ for $i=1,\dots,n$ using Algorithm~\ref{alg:int-degree-match} with failure probability $\varepsilon_1$.\label{alg-14-step8} \Comment{These degrees are used for detecting lucky primes.}

\State Initialize $\H \gets 2$ as a guess for the coefficient bound $\Ho$ of $G$.
\State Initialize $\mathcal{L} \gets \emptyset$.

\While{$\H \le \mathbb{H}_{o\max}$}

    \State $\kappa \gets \left\lceil \frac{\log_2(2\H^2 + 1)}{\log_2 \lambda} \right\rceil$.

    \Statex \textbf{Collect $\kappa$ lucky primes and compute modular GCDs:}

    \While{$|\mathcal{L}| < \kappa$}
    \State Choose a prime $p$ uniformly from $[\lambda, 4\lambda]$, distinct from previously chosen primes.

\algstore{intgcd}
\end{algorithmic}
\end{algorithm}

\begin{algorithm}[p]
\begin{algorithmic}[1]
\algrestore{intgcd}

    \Statex \textbf{Lucky prime test:}
    \State Choose a random vector $\mathbf{a} \in \mathbb{F}_p^n$.

    \For{$k = 1$ to $n$}
        \If{$\operatorname{lc}_{x_k}(A)(\mathbf{a}) \equiv 0 \pmod{p}$ or $\operatorname{lc}_{x_k}(B)(\mathbf{a}) \equiv 0 \pmod{p}$}
            \State \textbf{continue} to next prime.
        \EndIf
        \State Compute
        \(
        d^{(p)}_{k} = \deg \gcd(A(\mathbf{\check{a}}_{k \mapsto y} ), B(\mathbf{\check{a}}_{k \mapsto y} )).
        \)
    \EndFor

    \If{$d_k^{(p)} < d_k$ for some $k=1,\dots,n$}
        \State \Return ``Failure''
    \ElsIf{$d_k^{(p)} = d_k$ for all $k=1,\dots,n$}
        \State Add $p$ to $\mathcal{L}$. \Comment{$p$ is lucky.}

        \State Compute $A_p:=A \mod p$ and $B_p:=B \mod p$.
        
        \State Compute $\mathcal{G}_p = \gcd(A_p, B_p)$ over $\mathbb{F}_p$ using Algorithm~\ref{alg:recursive-amplified} with term bound $T$ and failure probability $\varepsilon_2$.
       
           \If{$\mathcal{G}_p = \text{``Failure''}$ or $\|\mathcal{G}_p\|_0 > T$}
        \State \Return ``Failure''
           \EndIf

         \textbf{Verify modular divisibility:} 
         \State Check whether $\mathcal{G}_p \mid A_p$ and $\mathcal{G}_p \mid B_p$ using Algorithm \ref{alg:divisibility-field} with failure probability $\varepsilon_2$ for each divisibility test.
          \If{$\mathcal{G}_p \nmid A_p$ or $\mathcal{G}_p \nmid B_p$}
        \State \Return ``Failure''
          \EndIf
        
    \EndIf
\EndWhile

\Statex \textbf{Reconstruct GCD via rational reconstruction:}

\State Let $m = \prod_{p \in \mathcal{L}} p$, with $m > 2\H^2$.

\State Let $\mathcal{M}$ be the union of all monomials appearing in the polynomials $\mathcal{G}_p$ for $p \in \mathcal{L}$.

\If{$|\mathcal{M}| > T$}
    \State \Return ``Failure''
\EndIf

\For{each monomial $M_i \in \mathcal{M}$}
    \State Collect its coefficients $u_{i,p} \in \mathbb{F}_p$ in $\mathcal{G}_p$ for each $p \in \mathcal{L}$.
    \State Use rational reconstruction to recover the unique reduced fraction
   $ \frac{c_i}{\ell_i} \in \mathbb{Q}$
    with $\gcd(c_i, \ell_i) = 1$, $|c_i| \le \H$, $1 \le \ell_i \le \H$, and
   $\frac{c_i}{\ell_i} \equiv u_{i,p} \pmod{p} \quad \text{for all } p \in \mathcal{L}.$
\EndFor

\State Compute $\ell = \operatorname{lcm}(\ell_1, \dots, \ell_t)$. For each $i = 1,\dots,t$, set $c_i' \gets \ell \cdot (c_i / \ell_i) \in \mathbb{Z}$.
\State Let $Q \in \mathbb{Z}[x_1,\dots,x_n]$ be the polynomial $\sum_{i=1}^t c_i' M_i$.
Let $\mathcal{G} \gets \operatorname{pp}(Q)$.

\If{$\|\mathcal{G}\|_{\infty} > \H$}
    \State $\H \gets \H^2$
    \State \textbf{continue}
\EndIf

\State \textbf{Verification:} Check whether $\mathcal{G} \mid A$ and $\mathcal{G} \mid B$ using Algorithm~\ref{alg:int-divisibility-test} with failure probability $\varepsilon_3$ for each divisibility test.
\If{the divisibility tests pass}
    \State \Return $\mathcal{G}$
\EndIf

\State $\H \gets \H^2$. \Comment{Square the coefficient bound guess.}

\EndWhile

\State \Return ``Failure''.
\end{algorithmic}
\end{algorithm}

\subsection*{Correctness Analysis}

We first fix notation for coefficient bounds. Recall that
\[
\Hi := \max\{\|A\|_\infty, \|B\|_\infty\}, \qquad \Ho := \|G\|_\infty.
\]
By Gelfond's inequality and the sparse Mignotte bound (Theorem~\ref{thm:sparse-mignotte}),
\[
\Ho \le \mathbb{H}_{o,\max}
:= \min\left\{ e^{nd}\min\{\|A\|_\infty,\|B\|_\infty\},\;
2^{D(T-1)}\min\{\|A\|_1,\|B\|_1\}\right\}.
\]

The algorithm guesses $\Ho$ by squaring: $\H = 2,4,16,\dots$ until $\H \ge \Ho$. If $\H$ exceeds $\mathbb{H}_{o,\max}$ without success, it returns ``Failure''.

In Step~\ref{alg12-4}, $\lambda$ is chosen as
\[
\lambda \ge \max\left\{
    21,\ 
    \frac{12nd}{5} \ln(2d \Hi^2 t^2) \cdot \ln(2e^{2nd}\Hi^2+1),\ 
    8nd^2 + 8nd
\right\}.
\]
The second term comes from applying Theorem~\ref{thm:lucky-primes-crt} with $\mu = 1/4$, replacing $\Ho$ by its Gelfond upper bound $e^{nd}\Hi$. Hence, even in the worst case, there are enough primes in $[\lambda,4\lambda]$ so that their product exceeds $2(e^{nd}\Hi)^2$. We use the Gelfond bound here instead of the sparse Mignotte bound, as the latter depends on $T$ and would complicate the analysis of $\lambda$. While the sparse Mignotte bound may be sharper in some cases, the Gelfond bound suffices for our purposes and keeps the analysis independent of $T$.

We now prove correctness in two cases: $T \ge T_0$ and $T < T_0$, where $T_0 := \|G\|_0$.

\begin{theorem}\label{thm:intgcd-no-bounds-correctness}[Correctness of Algorithm~\ref{alg:intgcd-no-bounds}]
Let $A, B \in \mathbb{Z}[x_1,\dots,x_n]$ be primitive and let $G = \gcd(A,B)$ with $\|G\|_0 = T_0$. For any guessed term bound $T$ and failure tolerance $\mu > 0$, Algorithm~\ref{alg:intgcd-no-bounds} satisfies:
\begin{enumerate}
    \item If $T \ge T_0$, it returns $G$ with probability at least $1-\mu$;
    \item If $T < T_0$, it returns ``Failure'' with probability at least $1-\mu$.
\end{enumerate}
\end{theorem}

\begin{proof}
We analyze the two cases separately.

\paragraph{Case 1: $T \ge T_0$ (correct term bound).}

Assume the partial degrees $d_i = \deg_{x_i} G$ are computed correctly in Step~\ref{alg-14-step8}; this holds with probability at least $1 - \varepsilon_1 = 1 - \mu/3$.

For a guess $\H < \Ho$, rational reconstruction yields a polynomial $\mathcal{G}$ with $\|\mathcal{G}\|_\infty \le \H < \Ho = \|G\|_\infty$, so $\mathcal{G} \ne G$. Since $\deg_{x_i}\mathcal{G} = \deg_{x_i}G$ for all $i$, if $\mathcal{G}$ divided both $A$ and $B$, it would be a common divisor with the same partial degrees as $G$; primitivity would force $\mathcal{G} = \pm G$, contradicting $\|\mathcal{G}\|_\infty < \|G\|_\infty$. Thus at least one divisibility test fails. By Theorem~\ref{thm:int-divisibility}, each test rejects a non-divisor with probability at least $1-\varepsilon_3$; over two tests per iteration and at most $L_{\max}$ iterations, the failure probability is at most $L_{\max} \cdot 2\varepsilon_3 = \mu/3$.

For $\H \ge \Ho$, conditional on correct $d_i$, the following parts are deterministic once their inputs are correct:
\begin{itemize}
    \item The lucky prime test has one-sided error: it never accepts a non-lucky prime, since acceptance requires $d_k^{(p)} = d_k$ for all $k$, exactly the definition of luckiness. It may reject lucky primes, but this only affects runtime.
    \item Rational reconstruction is deterministic and recovers $G$ exactly if all $\mathcal{G}_p$ are correct and $\prod p > 2\Ho^2$.
    \item Divisibility verification always passes if $\mathcal{G} = G$.
\end{itemize}
Thus the only probabilistic failure when $\H \ge \Ho$ is an incorrect finite-field GCD computation. For each of at most $\kappa_{\max}$ primes, this occurs with probability $\varepsilon_2 = \mu/(3\kappa_{\max}^2)$, so the union bound gives failure probability at most $\kappa_{\max}\varepsilon_2 = \mu/(3\kappa_{\max}) \le \mu/3$.

The algorithm succeeds if:
\begin{enumerate}
    \item all divisibility tests reject incorrect candidates for $\H < \Ho$ (probability $\ge 1-\mu/3$),
    \item all finite-field GCD computations are correct when $\H$ first reaches $\Ho$ (probability $\ge 1-\mu/3$),
    \item the partial degrees $d_i$ are correct (probability $\ge 1-\mu/3$).
\end{enumerate}
By the union bound, the total success probability is at least $(1-\mu/3)^3 \ge 1-\mu$.

\paragraph{Case 2: $T < T_0$ (incorrect term bound).}

Assume first that the partial degrees $d_i$ are correct, which occurs with probability at least $1 - \varepsilon_1 = 1 - \mu/3$.

Let $\H$ be the first guess such that $\H \ge \Ho$. At this point, the algorithm has collected $\kappa$ lucky primes $p_1,\dots,p_\kappa$ with $\prod p_i > 2\H^2 \ge 2\Ho^2$.

For each $p_i$, let $G_{p_i} := G \bmod p_i$ be the true modular GCD. Among the collected primes $p_1,\dots,p_{\kappa}$, there are two possible cases:

\begin{enumerate}
   \item \textbf{Some $G_{p_i}$ has $\|G_{p_i}\|_0 > T$.} 
In this case, when computing $\gcd(A_{p_i}, B_{p_i})$ using Algorithm~\ref{alg:recursive-amplified} with the supplied term bound $T$, since $T$ is not a correct bound, the algorithm either returns ``Failure'' or returns an incorrect polynomial $\mathcal{G}_{p_i}$ with $\|\mathcal{G}_{p_i}\|_0 \le T$. In the latter case, we have $\mathcal{G}_{p_i} \nmid A_{p_i}$ or $\mathcal{G}_{p_i} \nmid B_{p_i}$. The probabilistic modular divisibility checks $\mathcal{G}_{p_i} \mid A_{p_i}$ and $\mathcal{G}_{p_i} \mid B_{p_i}$ will detect the failure with probability at least $1 - \varepsilon_2$ for each test. By the union bound, the probability that both tests correctly detect the failure is at least $1 - 2\varepsilon_2$. Thus, in this case, the algorithm returns ``Failure'' with probability at least 
\[
1 - 2\varepsilon_2 = 1 - \frac{2\mu}{3\kappa_{\max}} \ge 1 - \frac{2\mu}{3}.
\]

\item \textbf{Every $G_{p_i}$ has $\|G_{p_i}\|_0 \le T$.} 
In this case, the individual term-count checks do not detect the failure. If all finite field GCD computations are correct, then each $\mathcal{G}_{p_i} = c_i \cdot G_{p_i}$, and rational reconstruction recovers the full polynomial $G$ exactly. Since $T < T_0 = \|G\|_0$, the final check $\|\mathcal{G}\|_0 > T$ triggers a ``Failure''.

Unlike the previous case, $\|G_{p_i}\|_0 \le T$ implies that Algorithm~\ref{alg:recursive-amplified} is invoked with a correct term bound $T$ for each prime (since the true modular GCD has at most $T$ terms). Thus its success probability is at least $1 - \varepsilon_2$ per prime. By the union bound, the probability that all finite field GCD computations are correct is at least
\[
1 - \kappa_{\max} \cdot \varepsilon_2 = 1 - \frac{\mu}{3\kappa_{\max}} \cdot \kappa_{\max} = 1 - \frac{\mu}{3}.
\]
Therefore, in this case, the algorithm returns ``Failure'' with probability at least $1 - \mu/3$.

\end{enumerate}

Thus, conditional on $d_i$ being correct, in either case the algorithm returns ``Failure'' with probability at least $1 - 2\mu/3$. Together with the probability that $d_i$ is correct ($\ge 1 - \mu/3$), the total success probability is at least
\[
(1 - \mu/3)(1 - 2\mu/3) \ge 1 - \mu.
\]
Thus the stated bound $1-\mu$ is satisfied.
\end{proof}

\subsection*{Expected Complexity Analysis}

We analyze the expected complexity of Algorithm~\ref{alg:intgcd-no-bounds}. 
On the successful path - that is, assuming all probabilistic subroutines return correct results --the outer loop squares $\H$ until it reaches $\Ho$; hence the number of iterations is $O(\log\log \Ho)$.

\begin{theorem}\label{thm:intgcd-no-bounds-complexity}[Expected Complexity of Algorithm~\ref{alg:intgcd-no-bounds}]
Let $A, B \in \mathbb{Z}[x_1,\dots,x_n]$ be primitive, and let $G = \gcd(A,B)$. 
For any failure tolerance $\mu > 0$, the expected bit complexity of Algorithm~\ref{alg:intgcd-no-bounds} is
\[
\widetilde{O}\Bigl(
nTD\log\Ho\log\Hi\log\frac{1}{\mu}
+ n t \log T\log^4 D \cdot \log \Hi \log\Ho\log\frac{1}{\mu}
\Bigr),
\]
where the expectation is over the random choices of primes, evaluation points, and finite-field GCD subroutines.
\end{theorem}

\begin{proof}
We analyze the cost of each major step, distinguishing between the successful path and the contribution of error paths.

\paragraph{Partial degree computation.}
By Theorem~\ref{thm:int-degree-match-complexity}, Algorithm~\ref{alg:int-degree-match} computes the partial degrees $d_i$ with cost
\[
C_{\deg} = \widetilde{O}\bigl( n t \log^2 d \cdot \log \Hi \log\frac{1}{\mu} + n d \log \Hi \log\frac{1}{\mu} \bigr)
\]
bit operations. This step is performed once.

\paragraph{Lucky prime selection and testing.}
For a candidate prime $p$, the lucky prime test proceeds as follows: choose a random vector $\mathbf{a} \in \mathbb{F}_p^n$; for each $k=1,\dots,n$, evaluate $A(\mathbf{\check{a}}_{k \mapsto y} )$ and $B(\mathbf{\check{a}}_{k \mapsto y} )$, and compute the univariate GCD degree $d_k^{(p)} = \deg \gcd(A(\mathbf{\check{a}}_{k \mapsto y} ), B(\mathbf{\check{a}}_{k \mapsto y} ))$. If $d_k^{(p)} = d_k$ for all $k$, the prime is accepted as lucky.

The test succeeds --- that is, it correctly identifies a lucky prime with a good evaluation point --- with probability at least $1/2$. Hence the expected number of trials to find one lucky prime is $O(1)$.

For a single trial, we first reduce $A$ and $B$ modulo $p$, which costs $\widetilde{O}(t\log \Hi + t\log p)$ bit operations. Evaluating all $n$ univariate images of all terms in $A$ and $B$ costs $\widetilde{O}(n t \log d \log p)$ bit operations, and computing the $n$ univariate GCDs costs $\widetilde{O}(n d \log p)$ bit operations. Since $\log p = O(\log \lambda)$, the cost of one lucky prime test, including primality testing, is
\[
C_{\text{lucky}} = \widetilde{O}\bigl( n t \log d \cdot \log \lambda + n d \log \lambda + t\log\Hi + \operatorname{poly}(\log \lambda) \bigr),
\]
where the $\operatorname{poly}(\log \lambda)$ term accounts for the cost of finding a prime in the interval $[\lambda, 4\lambda]$.

We now analyze the expected cost of collecting lucky primes.

On the successful path, where the partial degrees $d_i$ are correct (which occurs with probability at least $1 - \mu/3$), every prime accepted into $\mathcal{L}$ is guaranteed to be lucky. The test may reject a lucky prime, but this only affects the number of trials, not correctness. Since $\kappa = O(\log \H)$ primes are needed in total, and $\H$ reaches $\Ho$ after $O(\log\log \Ho)$ iterations, the expected cost of collecting lucky primes is
\[
\widetilde{O}\bigl( \log \Ho \cdot (n t \log d \cdot \log \lambda + n d \log \lambda + t\log\Hi+\operatorname{poly}(\log \lambda)) \bigr).
\]

It remains to consider the case where the partial degrees $d_i$ are correct, but some finite field GCD computation fails in the iteration where $\H$ first reaches or exceeds $\Ho$. In this case, although the coefficient bound is already sufficient, a failed modular GCD may corrupt the rational reconstruction, causing the algorithm to continue squaring $\H$ until it reaches the worst-case bound $\mathbb{H}_{o,\max}$, potentially requiring $\kappa_{\max}$ primes instead of $O(\log \Ho)$.

We set the failure probability of the modular divisibility check (Algorithm~\ref{alg:divisibility-field}) to
\[
\varepsilon_2 = \frac{\mu}{3\kappa^2_{\max}}.
\]

Now consider the iteration where $\H$ first reaches or exceeds $\Ho$, and suppose some $\mathcal{G}_p$ is computed incorrectly. In this event, the modular divisibility check $\mathcal{G}_p \mid A_p$ and $\mathcal{G}_p \mid B_p$ consists of two divisibility tests. Each test detects a failure with probability at least $1 - \varepsilon_2$, so by the union bound, the probability that both tests correctly detect the failure is at least $1 - 2\varepsilon_2$.

By the union bound over at most $\kappa_{\max}$ primes, the probability that an incorrect $\mathcal{G}_p$ escapes detection is at most
\[
\kappa_{\max} \cdot 2\varepsilon_2
= \kappa_{\max} \cdot 2 \cdot \frac{\mu}{3\kappa^2_{\max}}
= \frac{2\mu}{3\kappa_{\max}}.
\]
In this error event, the algorithm may need to collect up to $\kappa_{\max}$ primes instead of $O(\log \Ho)$. The additional number of lucky prime tests is at most $\kappa_{\max}$. Multiplying by the probability of the error event gives
\[
\kappa_{\max} \cdot \frac{2\mu}{3\kappa_{\max}} = \frac{2\mu}{3} = O(1).
\]
Thus, the contribution of this error path to the expected number of lucky prime tests is bounded by a constant, and hence its contribution to the expected cost is
\[
O(C_{\text{lucky}}),
\]
which is dominated by the correct-path cost.

If the partial degrees $d_i$ are incorrect (which occurs with probability $\mu/3$), the algorithm may accept non-lucky primes. However, once a lucky prime and a good evaluation point are found, the discrepancy $d_k^{(p)} < d_k$ will be detected, and the algorithm returns ``Failure'' immediately. The probability of detecting such an error in a single trial is at least $1/2$. Thus, if the algorithm performs $j$ lucky prime tests before stopping, the probability that the error remains undetected for the first $j-1$ trials is at most $(1/2)^{j-1}$. The expected cost contributed by this error path is therefore bounded by
\[
\sum_{j=1}^{\infty} \left(\frac{1}{2}\right)^{j-1} \cdot j \cdot C_{\text{lucky}}
= O(C_{\text{lucky}}),
\]
which is dominated by the cost of the correct-path case.

Therefore, the overall expected cost of prime selection and lucky prime testing is
\[
\widetilde{O}\bigl( \log \Ho \cdot (n t \log d \cdot \log \lambda + n d \log \lambda + t\log\Hi+\operatorname{poly}(\log \lambda)) \bigr).
\]

We note that the complexity analysis applies uniformly to both $T \ge T_0$ and $T < T_0$: the algorithm performs exactly the same computations in both cases. The only difference is that when $T < T_0$, the final check $\|\mathcal{G}\|_0 > T$ returns ``Failure'' instead of outputting $G$. Hence the expected complexity bound is independent of whether the guessed term bound is correct.

\paragraph{Finite field GCD computation and modular divisibility check.}
For each lucky prime $p \in \mathcal{L}$, Algorithm~\ref{alg:recursive-amplified} computes $\mathcal{G}_p = \gcd(A \bmod p, B \bmod p)$ over $\mathbb{F}_p$, followed by the modular divisibility check $\mathcal{G}_p \mid A_p$ and $\mathcal{G}_p \mid B_p$ using Algorithm~\ref{alg:divisibility-field}. These two steps are performed together for each prime.

On the successful path, where the partial degrees $d_i$ are correct and all computations are correct, the algorithm needs $\kappa = O(\log \Ho)$ primes in the final iteration when $\H$ first reaches or exceeds $\Ho$. By Theorem~\ref{rem-ampl-com}, each GCD computation costs
\[
C_{\text{GCD}} = \widetilde{O}\!\left( n T D \cdot \log \frac{\kappa_{\max}}{\mu}\cdot\log\lambda + n t \log^2 T \log^2 D \cdot \log \frac{\kappa_{\max}}{\mu}\cdot\log\lambda \right)
\]
bit operations, which is \[
C_{\text{GCD}} = \widetilde{O}\!\left( n T D \cdot \log \frac{1}{\mu}\cdot(\log\log \Hi)^2 + n t \log^2 T \log^4 D \cdot (\log\log \Hi)^2 \log\frac{1}{\mu} \right)
\]

By Lemma~\ref{lem:divisibility-field}, the modular divisibility check for each prime costs
\[
C_{\text{div}} = \widetilde{O}\bigl( n (t+T) \log^2 d \log\frac{\kappa_{\max}}{\mu}\log\lambda + n d \log\frac{\kappa_{\max}}{\mu}\log\lambda \bigr)
\]
bit operations, since all arithmetic is performed modulo $p = O(\lambda)$.
This is dominated by $C_{\text{GCD}}$, so the total cost per prime is $C_{\text{GCD}} + C_{\text{div}} = O(C_{\text{GCD}})$. Hence the total cost on the successful path is
\[
\widetilde{O}\bigl( \log \Ho \cdot C_{\text{GCD}} \bigr).
\]

We now bound the contribution of error paths, still assuming $d_i$ is correct. Suppose a finite field GCD computation fails when $\H$ first reaches or exceeds $\Ho$, producing an incorrect $\mathcal{G}_p$. The modular divisibility check consists of two tests: $\mathcal{G}_p \mid A_p$ and $\mathcal{G}_p \mid B_p$. By Lemma~\ref{lem:divisibility-field}, each test detects the failure with probability at least $1 - \varepsilon_2$, so by the union bound, the probability that both tests correctly detect the failure is at least $1 - 2\varepsilon_2$. We set
\[
\varepsilon_2 = \frac{\mu}{3\kappa^2_{\max}}.
\]
By the union bound over at most $\kappa_{\max}$ primes, the probability that an incorrect $\mathcal{G}_p$ escapes detection is at most
\[
\kappa_{\max} \cdot 2\varepsilon_2
= \kappa_{\max} \cdot 2 \cdot \frac{\mu}{3\kappa^2_{\max}}
= \frac{2\mu}{3\kappa_{\max}}.
\]
In this error event, the algorithm may need to compute up to $\kappa_{\max}$ primes instead of $O(\log \Ho)$. The additional cost is at most $\kappa_{\max} \cdot (C_{\text{GCD}} + C_{\text{div}}) = \kappa_{\max} \cdot O(C_{\text{GCD}})$. Thus the contribution of this error path to the expected complexity is at most
\[
\frac{2\mu}{3\kappa_{\max}} \cdot \kappa_{\max}\cdot O( C_{\text{GCD}})=\frac{2\mu}{3} \cdot  O( C_{\text{GCD}}),
\]
which is dominated by the successful path cost since $\mu\le 1$.

If the partial degrees $d_i$ are incorrect (which occurs with probability $\mu/3$), the algorithm may accept non-lucky primes. However, once a lucky prime and a good evaluation point are found, the discrepancy $d_k^{(p)} < d_k$ will be detected, and the algorithm returns ``Failure'' immediately. The probability of detecting such an error in a single trial is at least $1/2$. Thus, if the algorithm performs $j$ lucky prime tests before stopping, the probability that the error remains undetected for the first $j-1$ trials is at most $(1/2)^{j-1}$. The expected cost contributed by this error path is therefore bounded by
\[
\sum_{j=1}^{\infty} \left(\frac{1}{2}\right)^{j-1} \cdot j \cdot (C_{\text{GCD}} + C_{\text{div}})
= O(C_{\text{GCD}}),
\]
which is dominated by the cost of the correct-path case.

\paragraph{Rational reconstruction.}

For the rational reconstruction, we analyze the worst-case cost. The algorithm may iterate up to $L_{\max}$ times and, in the worst case, $\H$ reaches $\mathbb{H}_{o,\max}$. In each iteration, reconstructing $O(T)$ coefficients from $\kappa = O(\log \H)$ residues costs $\widetilde{O}(T \log \H)$ bit operations. Summing over all iterations gives
\[
\sum_{\H = 2,4,16,\dots}^{\mathbb{H}_{o,\max}} \widetilde{O}(T \log \H)
= \widetilde{O}(T \log \mathbb{H}_{o,\max}).
\]
Since $\mathbb{H}_{o,\max} \le e^{nd}\Hi$, this is bounded by
\[
\widetilde{O}\bigl(T \log(e^{nd}\Hi)\bigr)
= \widetilde{O}(nTd + T\log \Hi),
\]
which is dominated by the finite field GCD cost and hence absorbed into the overall complexity bound.

\paragraph{Divisibility verification.}
 In each outer iteration, Algorithm~\ref{alg:int-divisibility-test} checks whether $\mathcal{G} \mid A$ and $\mathcal{G} \mid B$. Denote $\H_k = 2^{2^k}$.
By Theorem~\ref{thm:int-divisibility-test-complexity}, each divisibility test on polynomials of size bounded by $\H_k$ costs
\[
\begin{aligned}
\widetilde{O}\Bigl(
& n (t+T) \log^2 d \cdot \log\log(\max(\Hi,\H_k)) \log(L_{\max}/\mu) \\
&+ n d \log\log(\max(\Hi,\H_k)) \log(L_{\max}/\mu) \\
&+ t\log\Hi \log(L_{\max}/\mu) 
+ T\log\H_k \log(L_{\max}/\mu) 
\Bigr).
\end{aligned}
\]

We use the worst-case bound $\H_k \le \mathbb{H}_{o,\max} = O(e^{nd}\Hi)$. Thus the cost of one test is bounded by
\[
\begin{aligned}
\widetilde{O}\Bigl(
& n (t+T) \log^2 d \cdot \log\log(e^{nd}\Hi) \cdot \log(L_{\max}/\mu)
+ n d \log\log(e^{nd}\Hi) \cdot \log(L_{\max}/\mu) \\
&+ t\log\Hi \log(L_{\max}/\mu)
+ T\log(e^{nd}\Hi) \log(L_{\max}/\mu)
\Bigr).
\end{aligned}
\]

Since there are at most $L_{\max}$ outer iterations, the total cost of divisibility verification is

\[
\begin{aligned}
\widetilde{O}\Bigl(
& L_{\max} n (t+T) \log^2 d \cdot \log\log(e^{nd}\Hi) \cdot \log(L_{\max}/\mu) \\
&+ L_{\max} n d \log\log(e^{nd}\Hi) \cdot \log(L_{\max}/\mu) \\
&+ L_{\max} t\log\Hi \log(L_{\max}/\mu) \\
&+ L_{\max} T\log(e^{nd}\Hi) \log(L_{\max}/\mu) 
\Bigr).
\end{aligned}
\]

Now $L_{\max} = \lceil \log_2 \log_2 \mathbb{H}_{o,\max} \rceil = O(\log(nd) + \log\log\Hi)$. Substituting this bound yields
\[
\begin{aligned}
\widetilde{O}\Bigl(
& n (t+T) \log^4 d \cdot (\log\log\Hi)^2 \log\frac{1}{\mu}
+ n d (\log\log\Hi)^2 \log\frac{1}{\mu} \\
&+ t\log\Hi \log(nd) \log\frac{1}{\mu}
+ nTd\log\Hi \log\frac{1}{\mu}
\Bigr).
\end{aligned}
\]

\paragraph{Combining all terms.}
Combining the costs of partial degree computation (Theorem~\ref{thm:int-degree-match-complexity}), lucky prime selection, finite field GCD computation with modular divisibility check, rational reconstruction, and divisibility verification, we obtain the following expected bit complexity:
\[
\widetilde{O}\Bigl(
nTD\log\Ho\log\Hi\log\frac{1}{\mu}
+ n t \log^2 T\log^4 D \cdot \log \Hi \log\Ho\log\frac{1}{\mu}
\Bigr).
\]
The expectation is taken over the random choices of primes, evaluation points, and the internal finite-field GCD subroutines.
\end{proof}

\subsection*{GCD Computation without a Priori Term Bound}

In this section, we remove the assumption that an upper bound $T$ on the number of terms of the GCD is supplied as input. 
We employ a doubling guess-and-verify strategy: the algorithm successively guesses $T = 2^k$ for $k = 0,1,2,\dots$, 
invokes Algorithm~\ref{alg:intgcd-no-bounds} with the current guess, and verifies each candidate. 
The complete procedure is given as Algorithm~\ref{alg:intgcd-guess}.

\begin{algorithm}[H]
\caption{Polynomial GCD Algorithm over the Integers}
\label{alg:intgcd-guess}
\begin{algorithmic}[1]
\Require
    \begin{itemize}
        \item $A, B \in \mathbb{Z}[x_1,\dots,x_n]$;
        \item A target error bound $0 < \mu < 1$.
    \end{itemize}
\Ensure
    The GCD $G = \gcd(A,B)$ with probability at least $1-\mu$, or ``Failure".

\State Compute contents: $c_A \gets \operatorname{cont}(A)$, $c_B \gets \operatorname{cont}(B)$.
\State $c_G \gets \gcd(c_A, c_B)$.
\State $P_A \gets \operatorname{pp}(A) = A / c_A$, $P_B \gets \operatorname{pp}(B) = B / c_B$.

\State Compute the partial degree bound $d = \max_{1 \le i \le n}  \max\{\deg_{x_i} P_A, \deg_{x_i}P_B\}$.

\For{each $k = 0, 1, 2, \dots, \lceil n\log_2(d+1) \rceil$}
    \State Set $T \gets 2^k$.
    \State Set $\delta_k \gets \mu / 3^{k+1}$.
    
    \State Run Algorithm~\ref{alg:intgcd-no-bounds} with inputs $P_A, P_B, T, \delta_k$ to obtain a result.
    
    \If{the result is a polynomial $\mathcal{G} \in \mathbb{Z}[x_1,\dots,x_n]$}
        \State \Return $c_G \cdot \mathcal{G}$.
    \ElsIf{the result is ``Failure''}
        \State \textbf{continue} to the next guess.
    \EndIf
\EndFor

\State \Return ``Failure''.
\end{algorithmic}
\end{algorithm}

\begin{theorem}[Correctness of Algorithm~\ref{alg:intgcd-guess}]
Let $A, B \in \mathbb{Z}[x_1,\dots,x_n]$ and let $G = \gcd(A,B)$. 
For any target error bound $0 < \mu < 1$, Algorithm~\ref{alg:intgcd-guess} returns $G$ with probability at least $1-\mu$.
\end{theorem}

\begin{proof}
Let $T_0 := \|G\|_0$ and let $k^* := \lceil \log_2 T_0 \rceil$. 
Since $T_0 \le (d+1)^n$, we have $k^* \le \lceil n\log_2(d+1) \rceil$, so the true term bound is reached within the loop.

For $k < k^*$, the guessed bound $T = 2^k$ is strictly less than $T_0$. 
By Theorem~\ref{thm:intgcd-no-bounds-correctness}, Algorithm~\ref{alg:intgcd-no-bounds} returns ``Failure'' with probability at least $1 - \delta_k$.

For $k = k^*$, we have $T = 2^{k^*} \ge T_0$. 
By the same theorem, Algorithm~\ref{alg:intgcd-no-bounds} returns the correct GCD with probability at least $1 - \delta_{k^*}$.

Let $\mathcal{E}_k$ denote the event that the algorithm fails to detect $T$ is incorrect for $k < k^*$, and let $\mathcal{E}_{k^*}$ denote the event that it fails to compute the correct GCD at $k = k^*$. 
By the union bound,
\[
\Pr(\text{failure}) \le \sum_{k=0}^{k^*-1} \delta_k + \delta_{k^*}.
\]
Substituting $\delta_k = \mu / 3^{k+1}$,
\[
\Pr(\text{failure})
\le \sum_{k=0}^{k^*-1} \frac{\mu}{3^{k+1}} + \frac{\mu}{3^{k^*+1}}
= \frac{\mu}{3}\left(1 + \frac{1}{3} + \cdots + \frac{1}{3^{k^*-1}}\right) + \frac{\mu}{3^{k^*+1}}
= \frac{\mu}{2} + \frac{\mu}{2 \cdot 3^{k^*+1}}
< \mu.
\]
Thus the success probability is at least $1-\mu$.
\end{proof}

\begin{theorem}[Expected Complexity of Algorithm~\ref{alg:intgcd-guess}]
Let $A, B \in \mathbb{Z}[x_1,\dots,x_n]$ and let $G = \gcd(A,B)$ with $T_0 := \|G\|_0$. 
For any target error bound $0 < \mu < 1$, the expected bit complexity of Algorithm~\ref{alg:intgcd-guess} is
\[ \widetilde{O}\Bigl( 
nT_0 D\log\Ho\log\Hi\log\frac{1}{\mu}
+ n t \log^4 T_0\log^4 D \cdot \log \Hi \log\Ho\log\frac{1}{\mu}
\Bigr)
\]
where $t := \max\{\|A\|_0,\|B\|_0\}$, $\Hi := \max\{\|A\|_\infty,\|B\|_\infty\}$, and $\Ho := \|G\|_\infty$.
\end{theorem}

\begin{proof}
Algorithm~\ref{alg:intgcd-guess} invokes Algorithm~\ref{alg:intgcd-no-bounds} for each guess $T = 2^k$, $k = 0, 1, \dots, k^*$, until the correct term bound is found. 
By Theorem~\ref{thm:intgcd-no-bounds-complexity}, the cost at guess $k$ is
\[
C_k = 
\widetilde{O}\Bigl(
nTD\log\Ho\log\Hi\log(1/\delta_k)
+ n t \log^2 T\log^4 D \cdot \log \Hi \log\Ho\log(1/\delta_k)
\Bigr).
\]

Let $E_k$ be the event that the algorithm stops at guess $k$. The events $E_0, E_1, \dots, E_{k_{\max}}$ form a partition of the probability space, where $k_{\max}=\lceil n\log_2(d+1)\rceil$. When $E_k$ occurs, the algorithm performs $C_0 + C_1 + \cdots + C_k$ operations. Thus the expected cost is
\[
\mathbb{E}[\text{cost}]
= \sum_{k=0}^{k_{\max}} \Pr(E_k) \sum_{i=0}^{k} C_i.
\]
Equivalently,
\[
\mathbb{E}[\text{cost}]
= \sum_{k=0}^{k_{\max}} C_k \cdot \Pr(E_k \cup E_{k+1} \cup \cdots \cup E_{k_{\max}}).
\]

We split the sum into two parts: $k \le k^*$ and $k > k^*$.

For $k \le k^*$, the probabilities are bounded by $1$, so the contribution is at most
\[
\sum_{k=0}^{k^*} C_k 
= \widetilde{O}\Bigl( k^* \cdot C_{k^*} \Bigr)
= \widetilde{O}\Bigl( 
nT_0 D\log\Ho\log\Hi\log\frac{1}{\mu}
+ n t \log^4 T_0\log^4 D \cdot \log \Hi \log\Ho\log\frac{1}{\mu}
\Bigr),
\]
since $C_k$ is dominated by its last term.

For $k > k^*$, the probability of reaching guess $k$ decays exponentially.

Now consider the tail probability for $k = k^* + s$ where $s \ge 1$. We need to bound
\[
\Pr(E_{k^*+s}) + \Pr(E_{k^*+s+1}) + \cdots + \Pr(E_{k_{\max}})
= 1 - \Pr(E_0 \cup E_1 \cup \cdots \cup E_{k^*+s-1}).
\]

Let $A = E_0 \cup E_1 \cup \cdots \cup E_{k^*+s-2}$. Then
\[
\Pr(E_0 \cup E_1 \cup \cdots \cup E_{k^*+s-1})=\Pr(A \cup E_{k^*+s-1})
= \Pr(A) + \Pr(E_{k^*+s-1} \mid \bar{A}) \Pr(\bar{A}).
\]
Since $\Pr(E_{k^*+s-1} \mid \bar{A}) \ge 1 - \delta_{k^*+s-1}$ (conditioned on not having stopped earlier, the algorithm returns a correct polynomial at step $k^*+s-1$), we have
\[
\Pr(A \cup E_{k^*+s-1})
\ge \Pr(A) + (1 - \delta_{k^*+s-1}) \Pr(\bar{A}).
\]
Thus
\[
1 - \Pr(A \cup E_{k^*+s-1})
\le \delta_{k^*+s-1} \Pr(\bar{A})
\le \delta_{k^*+s-1}.
\]
Therefore,
\[
\Pr(E_{k^*+s}) + \Pr(E_{k^*+s+1}) + \cdots + \Pr(E_{k_{\max}})
\le \delta_{k^*+s-1}.
\]

Substituting this bound into the expected cost formula
\[
\mathbb{E}[\text{cost}]
= C_0 + C_1 \Pr(E_1 \cup \cdots \cup E_{k_{\max}}) + \cdots + C_{k_{\max}} \Pr(E_{k_{\max}}),
\]
we obtain
\[
\mathbb{E}[\text{cost}]
\le \sum_{k=0}^{k^*} C_k
+ C_{k^*+1} \cdot \delta_{k^*}
+ C_{k^*+2} \cdot \delta_{k^*+1}
+ \cdots + C_{k_{\max}} \cdot \delta_{k_{\max}-1}.
\]

Since $T = 2^k$ and $\delta_k = \varepsilon / 3^{k+1}$, both $T$ and $\log(1/\delta_k)$ grow monotonically with $k$; hence $C_k$ is increasing (up to the polylogarithmic factors absorbed in the $\widetilde{O}$ notation). Therefore,
for the tail, writing $s = k - k^*$, we have
\[
\sum_{k=k^*+1}^{k_{\max}} \delta_{k-1} C_k
=  \sum_{s=1}^{k_{\max}-k^*} \frac{\mu}{3^{k^*+s}} \, C_{k^*+s}.
\]

Recall that
\[
C_{k^*+s} = \widetilde{O}\!\left(
n 2^{k^*+s} D\log\Ho\log\Hi \log(1/\delta_{k^*+s})
+ nt \log^2(2^{k^*+s}) \log^4 D \log \Hi \log\Ho\log(1/\delta_{k^*+s})
\right).
\]

Since $\log(1/\delta_{k^*+s}) = \log(3^{k^*+s+1}/\varepsilon) = O((k^*+s)\cdot \log\frac{1}{\varepsilon})$, as the same method presented in Theorem \ref{thm:gcd-no-t},
the first term in the tail is bounded by
\[
\sum_{s=1}^{\infty}  \frac{\mu}{3^{k^*+s}}
\cdot n D 2^{k^*+s}
\cdot \log\Ho\log\Hi (k^*+s)^{O(1)}\cdot\log\frac{1}{\mu}
= \widetilde{O}\bigl(\mu \, n D \log\Ho\log\Hi \log\frac{1}{\mu}\bigr),
\]
and the second term is bounded by
\[
\sum_{s=1}^{\infty} \frac{\mu}{3^{k^*+s}}
\cdot nt
\cdot (k^*+s)^{O(1)}\log\frac{1}{\mu} \log\Ho\log\Hi \log^4 D
= \widetilde{O}\bigl(\mu \, nt \log\Ho\log\Hi \log^4 D \log\frac{1}{\mu}\bigr).
\]

In both cases, the summand contains the factor $(2/3)^{k^*+s}$ or $(2/3)^{k^*+s}$ times a polynomial in $k^*+s$; hence each series converges geometrically. Consequently, the expected cost is dominated by the $k \le k^*$ terms, yielding
\[
\mathbb{E}[\text{cost}]
= \widetilde{O}\Bigl( 
nT_0 D\log\Ho\log\Hi\log\frac{1}{\mu}
+ n t \log^4 T_0\log^4 D \cdot \log \Hi \log\Ho\log\frac{1}{\mu}
\Bigr).
\]
\end{proof}

\section{Experimental Results}
\label{sec:experiments}

\subsection{Implementation and Optimizations}

We have implemented the algorithms presented in this paper in Maple 2023. 
All experiments were performed on a machine with an Intel(R) Core(TM) i7-9700 CPU at 3.00 GHz and 8.00 GB of RAM, running Windows 11.

We focus our experiments on the integer GCD algorithm, as it integrates all components of our method: the field algorithm, modular reduction, rational reconstruction, and the adaptive guessing strategy, and thus best demonstrates the full power of our approach. Moreover, the integer polynomial GCD is one of the most commonly used operations in symbolic computation, making it a natural benchmark for practical evaluation.

To achieve higher performance, we have incorporated several optimizations into our implementation:

\begin{enumerate}
    \item \textbf{Collision detection via exponent bounds.} Instead of computing three modular GCDs at $\b, \b^2, \b^3$ and evaluating a determinant, we detect collisions by choosing a prime $p$ that is much larger than the partial degree bound $d$. During derivative recovery, if any computed exponent exceeds $d$, a collision is detected. This reduces the cost of collision detection from three univariate GCDs to a simple bound check.
    
    \item \textbf{Early termination on excessive collisions.} When the number of colliding terms exceeds a threshold, we abort the current lifting iteration and restart with a doubled term bound $T \leftarrow 2T$. This prevents the algorithm from wasting time on iterations that are unlikely to recover a sufficient number of terms.
    
    \item \textbf{Tighter initial coefficient bound.} Rather than starting the coefficient guessing loop from $\H = 2$, we initialize $\H$ to the input coefficient bound $\Hi = \max\{\|A\|_\infty, \|B\|_\infty\}$. This reduces the number of outer iterations and improves performance on inputs with large coefficients.
    
    \item \textbf{Deterministic divisibility checking.} For the final verification step, we use Maple's built-in exact division test rather than our probabilistic divisibility algorithm. This eliminates the small probability of false positives in the verification stage.
\end{enumerate}

\subsection*{Test Polynomial Generation}
For each test instance, we generate polynomials $A = G \cdot A_1$ and $B = G \cdot B_1$, where $G$ is a random sparse polynomial with the prescribed term count $T$ and total degree $D$, and $A_1, B_1$ are random sparse polynomials with term counts uniformly chosen from $[5,20]$ and degrees from $[10,50]$. All polynomials are in $n$ variables $x_1,\dots,x_n$. The coefficients are random integers in the range $[-99, 99]$. This construction ensures that the true GCD is exactly $G$, with term count $T$, with high probability.

For each data point, we report the average running time over $10$ independent runs, with the same $A,B$ used for both our algorithm and the built-in Maple \texttt{gcd} command. Garbage collection is forced before each timing measurement to reduce memory interference. All reported times are in seconds.

\subsection{Scalability with Respect to $n$, $T$, and $D$}

We first evaluate the scalability of our GCD algorithm (Algorithm~\ref{alg:intgcd-guess}) with respect to the three fundamental parameters: 
the number of variables $n$, the term count $T$, and the total degree $D$.

\paragraph{Scalability in $n$.}
We fixed $T = 50$ and $D = 800$, and varied $n$ from $2$ to $44$. 
Figure~\ref{fig:n-scalability} reports the running time as a function of $n$. 
Both algorithms exhibit linear growth in $n$, consistent with the theoretical bound $\widetilde{O}(n)$. Our algorithm has a slightly higher slope than the built-in \texttt{gcd}.

\begin{figure}[htbp]
\centering
\includegraphics[width=0.8\textwidth]{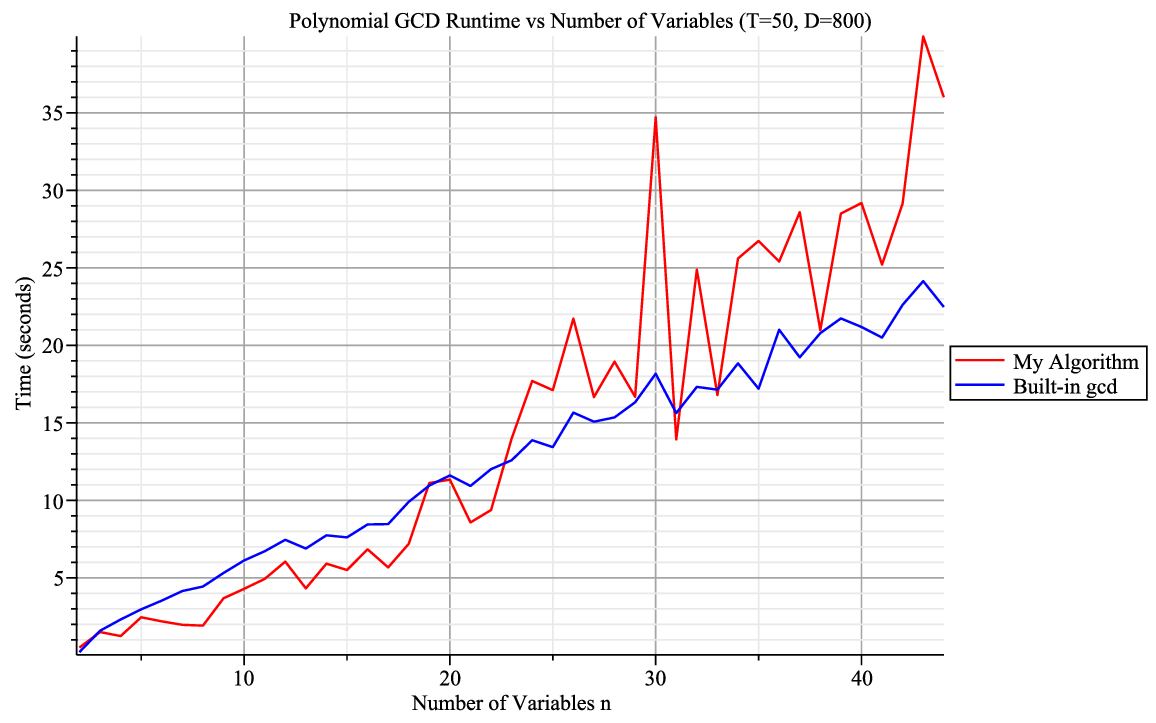}
\caption{Running time vs.\ the number of variables $n$, with $T=50$, $D=800$.}
\label{fig:n-scalability}
\end{figure}

\paragraph{Scalability in $T$.}
We fixed $n = 8$ and $D = 1000$, and varied $T$ from $10$ to $330$. 
Figure~\ref{fig:T-scalability} shows the running time as a function of $T$. 
The observed growth of our algorithm is asymptotically linear in $T$, matching the theoretical complexity bound; the built-in \texttt{gcd} exhibits similar linear behavior but with a slightly different slope. For small $T$ (sparse cases), our algorithm performs slightly better; as $T$ grows, the difference narrows.

\begin{figure}[htbp]
\centering
\includegraphics[width=0.8\textwidth]{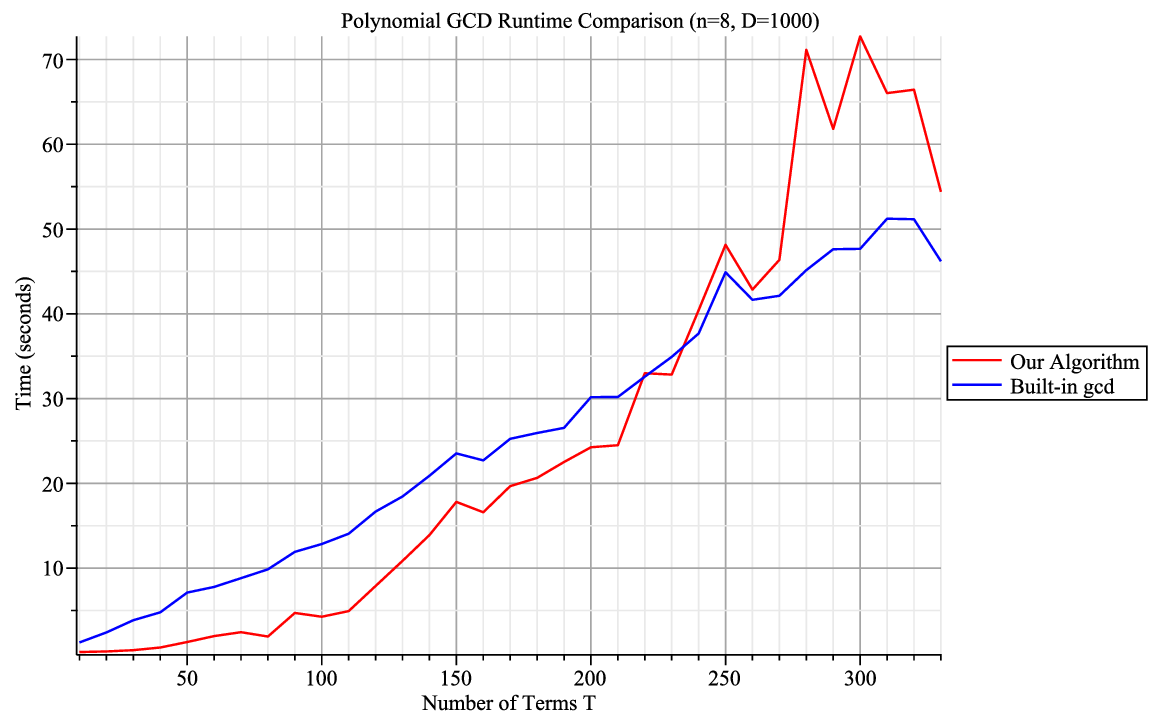}
\caption{Running time vs.\ the term count $T$, with $n=8$, $D=1000$.}
\label{fig:T-scalability}
\end{figure}

\paragraph{Scalability in $D$.}
We fixed $n = 8$ and $T = 100$, and varied $D$ from $500$ to $3000$. 
Figure~\ref{fig:D-scalability} reports the running time as a function of $D$. 
Our algorithm shows very mild growth in $D$, since it extracts derivative information by dividing coefficients. In contrast, the built-in \texttt{gcd} exhibits faster growth in $D$, confirming that classical algorithms suffer from increasing overhead as the degree grows. Our algorithm significantly outperforms the built-in \texttt{gcd} when $D$ is large.

\begin{figure}[htbp]
\centering
\includegraphics[width=0.8\textwidth]{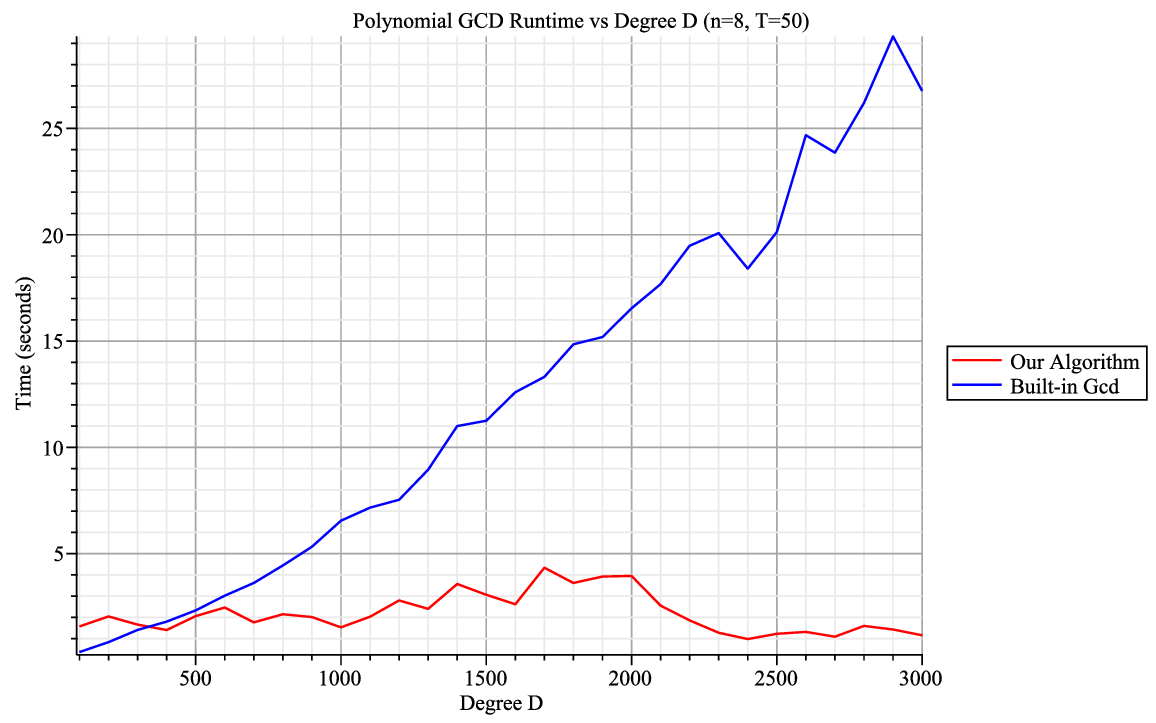}
\caption{Running time vs.\ the total degree $D$, with $n=8$, $T=100$.}
\label{fig:D-scalability}
\end{figure}

\section{Conclusion}
\label{sec:conclusion}
We have presented the first sparse polynomial GCD algorithm with expected complexity $\widetilde{O}(n\cdot T\cdot D)$ over fields of characteristic zero or sufficiently large, and the same linear guarantee over $\mathbb{Z}[x_1,\dots,x_n]$. The key innovation is a derivative-assisted Hensel lifting framework that performs a single $z^2$-lift per variable, reducing sequential depth from $O(D)$  or $O(n\log d)$ to $O(1)$ and enabling parallel extraction of derivative information. Collisions are handled iteratively, incurring only a logarithmic overhead. 

Maple experiments confirm the theoretical predictions: the algorithm is competitive with the built-in \texttt{gcd} in general and significantly faster when $D$ is large.

The derivative-driven lifting framework suggests broader applications in factorization, interpolation, and algebraic equation solving, which we leave for future work.

\section*{Appendix}

\begin{lemma}
Let $\K$ be a field, $\b=(b_1,\dots,b_n)\in\K^n$, and let
\[
g(\x,y)=\frac{g_1(\x,y)}{g_2(\x,y)}\in \K(\x,y)
\]
be a rational function with $g_1,g_2\in\K[\x,y]$. Then
\[
g(z+b_1,\dots,z+b_n,y)\equiv g(\b,y)+S(g)(\b,y)\cdot z \pmod{z^2}.
\]

\end{lemma}

\begin{proof}
Set $(g_1)_0:=g_1(\b,y)$ and $(g_2)_0:=g_2(\b,y)$. Since $g(\b,y)$ is well defined, we have $(g_2)_0\neq 0$.

By the polynomial case, expanding in $z$ to first order gives
\[
\begin{aligned}
g_1(z+b_1,\dots,z+b_n,y) &\equiv (g_1)_0 + S(g_1)(\b,y)\,z \pmod{z^2},\\
g_2(z+b_1,\dots,z+b_n,y) &\equiv (g_2)_0 + S(g_2)(\b,y)\,z \pmod{z^2}.
\end{aligned}
\]
Hence
\[
g(z+b_1,\dots,z+b_n,y)
= \frac{g_1(z+\b)}{g_2(z+\b)}
\equiv
\frac{(g_1)_0 + S(g_1)(\b,y)\,z}{(g_2)_0 + S(g_2)(\b,y)\,z}
\pmod{z^2}.
\]
Since $(g_2)_0\neq 0$, we have the formal expansion
\[
\frac{1}{(g_2)_0 + S(g_2)(\b,y)\,z}
\equiv
\frac{1}{(g_2)_0}\left(1 - \frac{S(g_2)(\b,y)}{(g_2)_0}\,z\right)
\pmod{z^2}.
\]
Substituting and collecting terms yields
\[
\begin{aligned}
g(z+b_1,\dots,z+b_n,y)
&\equiv
\left((g_1)_0 + S(g_1)(\b,y)\,z\right)
\cdot
\frac{1}{(g_2)_0}
\left(1 - \frac{S(g_2)(\b,y)}{(g_2)_0}\,z\right)
\pmod{z^2}\\
&\equiv
\frac{(g_1)_0}{(g_2)_0}
+
\left(
\frac{S(g_1)(\b,y)}{(g_2)_0}
-
\frac{(g_1)_0 S(g_2)(\b,y)}{(g_2)_0^2}
\right) z
\pmod{z^2}.
\end{aligned}
\]
On the other hand, for the rational function $g=g_1/g_2$,
\[
\frac{\partial g}{\partial x_i}
=
\frac{(\partial g_1/\partial x_i)g_2 - g_1(\partial g_2/\partial x_i)}{g_2^2},
\]
so summing over $i$ gives
\[
S(g)=\sum_{i=1}^n \frac{\partial g}{\partial x_i}=\frac{S(g_1)g_2 - g_1 S(g_2)}{g_2^2}.
\]
Evaluating at $(\b,y)$,
\[
S(g)(\b,y)
=
\frac{S(g_1)(\b,y)(g_2)_0 - (g_1)_0 S(g_2)(\b,y)}{(g_2)_0^2}.
\]
Therefore
\[
\frac{S(g_1)(\b,y)}{(g_2)_0}
-
\frac{(g_1)_0 S(g_2)(\b,y)}{(g_2)_0^2}
=
S(g)(\b,y),
\]
and $\frac{(g_1)_0}{(g_2)_0}=g(\b,y)$. Substituting back proves the claim.
\end{proof}

\begin{lemma}\label{lm-3}
Let $\K$ be a field, $\b=(b_1,\dots,b_n)\in\K^n$, and let
\[
g(\x,y)=\frac{g_1(\x,y)}{g_2(\x,y)}\in \K(\x,y)
\]
be a rational function with $g_1,g_2\in\K[\x,y]$. Then
\[
g(z+b_1,\dots,2z+b_k,\dots,z+b_n,y)
\equiv
g(\b,y)+\left(S(g)+\frac{\partial g}{\partial x_k}\right)(\b,y)\,z
\pmod{z^2}.
\]
\end{lemma}

\begin{proof}
Set $(g_1)_0:=g_1(\b,y)$ and $(g_2)_0:=g_2(\b,y)$. Since $g(\b,y)$ is well defined, we have $(g_2)_0\neq 0$.

By the polynomial case, expanding in $z$ to first order gives
\[
\begin{aligned}
g_1(z+b_1,\dots,2z+b_k,\dots,z+b_n,y)
&\equiv
(g_1)_0+\left(S(g_1)+\frac{\partial g_1}{\partial x_k}\right)(\b,y)\,z
\pmod{z^2},\\
g_2(z+b_1,\dots,2z+b_k,\dots,z+b_n,y)
&\equiv
(g_2)_0+\left(S(g_2)+\frac{\partial g_2}{\partial x_k}\right)(\b,y)\,z
\pmod{z^2}.
\end{aligned}
\]
Hence
\[
g(z+b_1,\dots,2z+b_k,\dots,z+b_n,y)
\equiv
\frac{(g_1)_0+(S(g_1)+\partial_k g_1)(\b,y)\,z}
{(g_2)_0+(S(g_2)+\partial_k g_2)(\b,y)\,z}
\pmod{z^2},
\]
where we write $\partial_k g_i := \dfrac{\partial g_i}{\partial x_k}$.
Since $(g_2)_0\neq 0$, we have the formal expansion
\[
\frac{1}{(g_2)_0+(S(g_2)+\partial_k g_2)(\b,y)\,z}
\equiv
\frac{1}{(g_2)_0}
\left(1-\frac{(S(g_2)+\partial_k g_2)(\b,y)}{(g_2)_0}\,z\right)
\pmod{z^2}.
\]
Substituting and collecting terms yields
\[
\begin{aligned}
g(z+b_1,\dots,2z+b_k,\dots,z+b_n,y)
&\equiv
\frac{(g_1)_0}{(g_2)_0}
+
\Bigg(
\frac{(S(g_1)+\partial_k g_1)(\b,y)}{(g_2)_0}
-
\frac{(g_1)_0(S(g_2)+\partial_k g_2)(\b,y)}{(g_2)_0^2}
\Bigg) z
\pmod{z^2}.
\end{aligned}
\]
On the other hand, for the rational function $g=g_1/g_2$,
\[
S(g)+\frac{\partial g}{\partial x_k}
=
\frac{(S(g_1)+\partial_k g_1)g_2 - g_1(S(g_2)+\partial_k g_2)}{g_2^2}.
\]
Evaluating at $(\b,y)$,
\[
\left(S(g)+\frac{\partial g}{\partial x_k}\right)(\b,y)
=
\frac{(S(g_1)+\partial_k g_1)(\b,y)(g_2)_0
-
(g_1)_0(S(g_2)+\partial_k g_2)(\b,y)}
{(g_2)_0^2}.
\]
Therefore
\[
\frac{(S(g_1)+\partial_k g_1)(\b,y)}{(g_2)_0}
-
\frac{(g_1)_0(S(g_2)+\partial_k g_2)(\b,y)}{(g_2)_0^2}
=
\left(S(g)+\frac{\partial g}{\partial x_k}\right)(\b,y),
\]
and $\frac{(g_1)_0}{(g_2)_0}=g(\b,y)$. Substituting back proves the claim.
\end{proof}

\bibliographystyle{abbrv}
\bibliography{mybibfile}

\begin{table}[p]
\centering
\renewcommand{\arraystretch}{1.3}
\caption{Notational Conventions}
\label{tab:notation}
\begin{tabular}{c|l}
\toprule
\textbf{Symbol} & \textbf{Description} \\
\midrule
$n$ & number of variables \\
$T$ & term count of the GCD (or a bound thereof) \\
$D$ & total degree of the input polynomials (or a bound) \\
$d$ & partial degree bound \\
$\K$ & base field \\
$\mathbb{Z}$ & ring of integers \\
$\Hi$ & coefficient bound of input polynomials: $\max\{\|A\|_\infty,\|B\|_\infty\}$ \\
$\Ho$ & coefficient bound of the GCD: $\|G\|_\infty$ \\
$\mathbb{H}_{o,\max}$ & worst-case coefficient bound for the GCD \\
$\mathcal{L}$ & set of collected lucky primes \\
$\kappa$ & number of primes needed for rational reconstruction \\
$\kappa_{\max}$ & maximum number of primes required in the worst case \\
$L_{\max}$ & maximum number of outer iterations in the coefficient guessing strategy \\
$\mu,\varepsilon$ & desired failure probability \\
$p$ & a prime for modular reduction \\
$\lambda$ & lower bound for the prime selection interval $[\lambda, 4\lambda]$ \\
$\H$ & current guess for the coefficient bound $\Ho$ \\
$G^*$ & current approximation to the GCD in the recursive algorithm \\
$\|F\|_{0}$ & number of nonzero terms of a polynomial $F$ \\
$\operatorname{cont}(F)$ & content of $F$ (GCD of its coefficients) \\
$\operatorname{pp}(F)$ & primitive part of $F$: $F / \operatorname{cont}(F)$ \\
$\operatorname{lc}_{x_k}(F)$ & leading coefficient of $F$ with respect to $x_k$ \\
$\res_{x_k}(F,G)$ & resultant of $F$ and $G$ with respect to $x_k$ \\
$\widetilde{O}$ & big-O notation suppressing polylogarithmic factors \\
\bottomrule
\end{tabular}
\end{table}

\end{document}